\documentclass{amsbook} 

\usepackage[utf8]{inputenc}

\usepackage[all]{xy}
\usepackage{epsfig}
\usepackage{wrapfig}
\usepackage{xcolor}
\usepackage{amssymb}
\usepackage[colorlinks,linkcolor=blue,anchorcolor=green,citecolor=blue]{hyperref}
\usepackage{hyperref}

\newtheorem{maintheorem}{Theorem}

\newtheorem{maincorollary}[maintheorem]{Corollary}
\newtheorem{theorem}{Theorem}[section]

\newtheorem{problem}[theorem]{Problem}

\newtheorem{example}[theorem]{Example}
\newtheorem{proposition}[theorem]{Proposition}
\newtheorem{lemma}[theorem]{Lemma}
\newtheorem{corollary}[theorem]{Corollary}
\newtheorem{fact}[theorem]{Fait}

\newtheorem{definition}[theorem]{Definition}
\newtheorem{remark}[theorem]{Remark}
\newtheorem{remarks}[theorem]{Remarks}

\newtheorem{question}[theorem]{Question}
\newtheorem{exercise}[theorem]{Exercise}

\newtheorem*{keylemma}{Key Lemma}

\numberwithin{equation}{chapter}

\newcommand\cA{{\mathcal A}}
\newcommand\cG{{\mathcal G}}

\newcommand\cV{{\mathcal V}}

\newcommand\Diff{{\operatorname{Diff}}}
\newcommand\diam{{\operatorname{diam}}}
\newcommand\diffsym{{\operatorname{\Delta}}}
\newcommand\eps{\epsilon}

\newcommand\GL{{\operatorname{GL}}}
\newcommand\Id{{\operatorname{Id}}}

\newcommand\Lip{{\operatorname{Lip}}}

\newcommand\Prob{\mathbb P}
\newcommand\Proberg{\mathbb P_{{\operatorname{erg}}}}

\newcommand\supp{{\operatorname{supp}}}
\renewcommand\top{{\operatorname{top}}}

\newcommand\vf{\varphi}

\newcommand\NN{\mathbb N}
\newcommand\CC{\mathbb C}

\newcommand\RR{\mathbb R}

\newcommand\ZZ{\mathbb Z}

\newcommand\heps{\widehat{\eps}}
\newcommand\hf{\widehat{f}}
\newcommand\hM{\widehat{M}}
\newcommand\hmu{\widehat{\mu}}
\newcommand\hnu{\widehat{\nu}}
\newcommand\hpi{\widehat{\pi}}
\newcommand\hvf{\widehat{\vf}}
\newcommand\hx{\widehat{x}}
\newcommand\hy{\widehat{y}}

\newcommand\wto{\stackrel*\longrightarrow}
\newcommand\hX{\widehat{X}}
\newcommand\loc{{\operatorname{loc}}}
\newcommand\Len{{\operatorname{Length}}}
\newcommand\cM{\mathcal M}

\newcommand\PP{\mathbb P}
\newcommand\hp{\widehat p}
\newcommand\hV{\widehat{V}}

\newcommand\SProb{\operatorname{S\Prob}}
\newcommand\tm{\widetilde{m}}
\newcommand\cU{\mathcal U}
\newcommand\HD{\operatorname{HD}}

\newcommand\tf{\widetilde{f}}

\newcommand\tmu{\widetilde{\mu}}
\newcommand\tnu{\widetilde{\nu}}
\newcommand\wsto{\rightharpoonup}

\newcommand\proj{\operatorname{proj}}
\newcommand\Orth{\operatorname{O}}
\newcommand\Aut{\operatorname{Aut}}
\newcommand\cF{\mathcal F}
\newcommand\Cocycle{\operatorname{Cocycle}}
\newcommand\Endo{\operatorname{End}}

\newcommand\reg{{\rm reg}}
\newcommand\hP{\widehat{P}}
\newcommand\hv{\widehat{v}}

\newcommand\ii[1]{\;[\!\!\![\; #1\;[\!\!\![\;}
\newcommand\Neutral{\mathfrak N}
\newcommand\hZ{\widehat{Z}}
\newcommand\cR{\mathcal R}
\newcommand\hU{\widehat{U}}
\newcommand\hK{\widehat{K}}

\newcommand\hGamma{\widehat{\Gamma}}
\newcommand\hyp{{\rm hyp}}
\newcommand\dens{\operatorname{dens}}
\newcommand\hsigma{\widehat{\sigma}}
\newcommand\hW{\widehat{W}}

\newcommand\mat[1]{\left(\begin{matrix} #1 \end{matrix}\right)}
\newcommand\alter[1]{\left\{\begin{array}{ll} #1 \end{array}\right.}
\newcommand\new[1]{\emph{#1}\index{#1}}
\newcommand\NEW[1]{{\bf #1}\index{#1}}
\newcommand\deletedframe[1]{}
\newcommand\pause{ }

\makeindex

\begin{document}

\title[Discontinuity of Exponents vs.  Entropy]{
Discontinuity of Lyapunov exponents vs Entropy\\ for smooth surface diffeomorphisms\footnote{Final version published in Banach Center Publications, volume 131 (2026) --- supported partially by the Simons Foundation Award No. 663281 granted to the Institute of Mathematics of the Polish Academy of Sciences for the years 2021-2023..}}
\author[J. Buzzi]{J\'er\^ome BUZZI (CNRS Orsay)}

\maketitle

\tableofcontents

\chapter*{Introduction}

Lyapunov exponents are fundamental invariants in smooth ergodic theory describing the asymptotic infinitesimal behavior along typical orbits.
This text aims to explain how and why to control Lyapunov exponents using entropy for smooth surface diffeomorphisms. 
It fits  into the framework of our recent joint works with Sylvain CROVISIER and Omri SARIG \cite{BCS-1,BCS-2,BCS-3,BCS-4}. We will focus especially on the continuity property of exponents for measures near the maximal entropy measure, by presenting a simplified version of \cite{BCS-2}.

Our exposition is geared towards advanced students and researchers in dynamics that are not necessarily familiar with smooth ergodic theory. We follow a mini-course that we gave at IMPAN in Warsaw in April 2023 as a series of 5 double lectures during a Simons Semester on Dynamics at IMPAN. The application to the SPR property of surface diffeomorphisms was then presented at the IMPAN conference center in Bedlewo during the Beyond Uniform Hyperbolicity conference.\footnote{The videos can be found at the following URLs: {\tt https://youtu.be/lvkMl$\underline{\;\;}$SiKP0} (minicourse at IMPAN) and {\tt https://youtu.be/D-gB-bbWCaI} (conference at Bedlewo).} Some notations are listed on page~\pageref{chap-notations}.

\medbreak

{\bf Acknowledgments.} The author would like to thank the organizers of the Simons Semester
program "Topological, smooth and holomorphic dynamics, ergodic theory, fractals" 
for the financial support and hospitality.\footnote{This work is partially supported by the Simons Foundation Award No.~663281 granted to the Institute of Mathematics of the Polish Academy of Sciences
for the years 2021-2023.}

\subsection*{Semicontinuity of Lyapunov exponents and entropy}
We consider a $C^\infty$ smooth diffeomorphism $f$ of a compact surface (more precisely, a boundaryless and compact two-dimensional $C^\infty$ Riemannian manifold). Given an invariant Borel probability measure  $\mu\in\Prob(f)$, its \emph{top Lyapunov exponent} describes the asymptotics of the norm of the iteration of the differential:
 $$
    \lambda^+(\mu) := \int \lambda^+(x) \, d\mu(x) \text{ where }
    \lambda^+(x):=\lim_{n\to\infty}\frac1n\log\|D_xf^n\|.
 $$
The set $\Prob(f)$ of $f$-invariant Borel probability measures is equipped with the weak (star) topology which turns it into a compact metrizable space. One can ask: \emph{what are the continuity properties of the map $\mu\mapsto\lambda^+(\mu)$ over $\Prob(f)$?}

It is an important and popular fact that this map is  upper semicontinuous but may fail to be lower semicontinuous. The first, positive statement follows from Kingman's ergodic theorem for subadditive processes that we recall below. The second, negative statement is a consequence of classical perturbative techniques in so-called Newhouse domains (see Proposition~\ref{propNewhouseDomain}).

\smallbreak

Another equally important functional defined over $\Prob(f)$ is the \emph{Kolmogorov-Sinaï entropy} $h(\mu)$. For now, we will just mention the word ``complexity'' and note that the entropy vanishes if the measure is carried by a periodic orbit. Entropy and exponents are closely linked. For instance, the Ruelle-Margulis inequality\footnote{In fact, the Ledrappier-Young formula \cite{LY2} expresses exactly the entropy as a sum of the exponents  weighted by some dimensions.} (Theorem~\ref{thm-Ruelle-Margulis}) states:
 $$
    \forall \mu\in\Prob(f)\quad h(\mu) \leq \int \max\left(\lambda^+(x),0\right)\, d\mu(x).
 $$

Since we are assuming $f$ to be $C^\infty$ smooth, the function $\mu\in\Prob(f)\longmapsto h(\mu)$ is upper semicontinuous by a theorem of Newhouse (Theorem~\ref{thm-Newhouse}, based on Yomdin's analysis of the iteration of $C^\infty$ smooth maps, Theorem~\ref{thm-Yomdin}). Unless the entropy function is identically zero over $\Prob(f)$, it is not lower semicontinuous by a closing lemma on surfaces due to Katok \cite{Katok-periodic}.

\section*{The results}

\subsection*{Lower semicontinuity ratios of exponent and entropy}
Let us observe that, in the setting of $C^\infty$ diffeomorphisms, the Kolmogorov-Sinaï entropy and the Lyapunov exponent turn out to satisfy \emph{the same semicontinuity properties}, though for different reasons.
Our results will show that the two continuity defects are nevertheless quantitatively related by a kind of ``relative Ruelle-Margulis inequality".

\medbreak

Our statements deal with sequences of ergodic measures and their weak star limit. Our results take the best and simplest form when this limit happens to be ergodic as in the following statement:

\begin{maintheorem}\label{maintheoremErgodic}
Let $f$ be a $C^\infty$ diffeomorphism of a closed surface.\footnote{A closed manifold is a compact and boundaryless manifold.}
Let $(\nu_k)_{k\ge1}$ be a sequence of ergodic invariant probability measures for $f$. Assume that their top exponents and their entropies converge to some positive limits. If $\mu$ is an ergodic accumulation point of $(\nu_k)_{k\ge1}$ in the weak star topology, then
  \begin{equation}\label{ineq1}
    0<^{(1)} \lim_k\frac{ h(\nu_k)}{h(\mu)} \le^{(2)} \lim_k\frac{\lambda^+(\nu_k)}{\lambda^+(\mu)}\le^{(3)}  1.
  \end{equation}
In particular, continuity of the entropy: $h(\mu)=\lim_k h(\nu_k)>0$, implies continuity of the exponent: $\lambda^+(\mu)=\lim_k \lambda^+(\nu_k)>0$.
\end{maintheorem}

Note that the assumption $\lim_k h(\nu_k)>0$ together with upper semicontinuity and the Ruelle inequality ensure that  all terms in the above inequality are well-defined.

\medbreak

 Our new result is the middle inequality (2) which can be compared to the Ruelle inequality. The other inequalities are obvious or classical: (1) is the positivity assumption, and  (3) is the upper semicontinuity of the Lyapunov exponent (recalled below, see Proposition~\ref{prop-average-exp-usc}).

\medbreak

The above statement is sharp in the following sense: 

\begin{proposition}\label{propNewhouseDomain}
On any closed surface, there is a $C^\infty$ diffeomorphism $f$ such that for any pair $0\le\alpha\le\beta\le 1$, one can find a sequence of ergodic $\nu_k\in\Prob(f)$ converging weakly to an ergodic measure $\mu$, and such that
 $$
   \lim_k h(\nu_k)=\alpha\cdot h(\mu), \text{ and }\lim_k \lambda^+(\nu_k)=\beta\cdot \lambda^+(\mu)>0.
 $$
One can also arrange that, additionally, $h(\mu)=0$ or that $h(\mu)>0$.
\end{proposition}

\begin{proof}
This is \cite[Prop 8.1]{BCS-2} extended to the cases $h(\mu)=0$ or $\alpha=0$. These extensions do not require any change in the proof given in \cite{BCS-2}.
\end{proof}

\medbreak

In particular, this construction shows that when $\lim_k h(\nu_k)=h(\mu)=0$, the discontinuity of the exponent is arbitrary ($\beta$ can take any value in $(0,1]$). Thus, the assumption that $\lim_k h(\mu_k)>0$ cannot be dropped from Theorem~\ref{maintheoremErgodic}.

\medbreak

The assumption that the limit is ergodic is a strong restriction. It can be removed at the price of a somewhat more abstract conclusion as follows.

\begin{maintheorem}\label{theoremWithoutErgodicity}
Let $f$ be a $C^\infty$ diffeomorphism of a closed surface.
Let $(\nu_k)_{k\ge1}$ be a sequence of ergodic invariant probability measures for $f$. Assume that their top exponents and their entropy converge to some positive limits. If $\mu$ is an arbitrary accumulation point of $(\nu_k)_{k\ge1}$ in the weak star topology, then there is a decomposition
 $$
    \mu = (1-\beta)\mu_0+\beta\mu_1 \text{ with }
    0<\beta\le 1 \text{ and } \mu_0,\mu_1\in\Prob(f)
 $$
satisfying $h(\mu_1)>0$ and
 \begin{equation}\label{ineq2}
    0<\lim_k \frac{h(\nu_k)}{h(\mu_1)} \le \beta = \lim_k \frac{\lambda^+(\nu_k)}{\lambda^+(\mu_1)}.
 \end{equation}
\end{maintheorem}

\begin{proof}[Deduction of Theorem~\ref{maintheoremErgodic} from Theorem~\ref{theoremWithoutErgodicity}]
Let $f,\nu_1,\nu_2,\dots,\mu$ be as in the hypothesis of Theorem~\ref{maintheoremErgodic}. We apply Theorem~\ref{theoremWithoutErgodicity} and get the decomposition $\mu=(1-\beta)\mu_0+\beta\mu_1$. The ergodicity of $\mu$ implies that this decomposition is trivial: $\mu_1=\mu$ since $\beta>0$. The inequalities in \eqref{ineq2} reduce to those in \eqref{ineq1}.
\end{proof}

In the remainder of this introduction, we deduce four direct applications of these results before stating the special case we will prove in these lectures.

\section*{Four applications}
We obtain the semicontinuity of some dimensions of invariant probability measures and use it for the construction of natural measures.

\medbreak

\subsubsection*{Semicontinuity of the Hausdorff dimension of ergodic measures}
We deduce from Theorem~\ref{maintheoremErgodic} the semicontinuity of the Hausdorff dimension over ergodic invariant measures with entropy bounded away from zero.

\medbreak

We first recall some basic facts.
We denote by $\HD(X)$ the classical Hausdorff dimension of any subset $X$ of a metric space \cite{falconer-book}. This notion is extended to any probability Borel measure $\nu$ by setting
 $$
    \HD(\nu):= \inf\{\HD(X):\nu(X)=1\}.
 $$
Given a non necessarily ergodic measure $\nu\in\Prob(f)$ whose top exponent $\lambda^+(\nu)$ and bottom exponent\footnote{The bottom exponent $\lambda^-(\nu)$ is the top exponent with respect to the inverse $f^{-1}$.} $\lambda^-(\nu)$ are both nonzero, the ``thermodynamical'' unstable and stable dimensions of $\mu$ are defined as
 $$
   \delta^+(\nu):=\frac{h(\nu)}{\lambda^+(\nu)} \text{ and }
   \delta^-(\nu):=-\frac{h(\nu)}{\lambda^-(\nu)}.
 $$
When $\nu\in\Prob(f)$ is ergodic and \emph{hyperbolic of saddle type} (i.e. with both positive and negative exponents, none vanishing),  these numbers are indeed (nonnegative) Hausdorff dimensions and we have L.-S. Young formula \cite{LSY-2d}
 \begin{equation}\label{eqYoungFormula}
    \HD(\nu) = \delta^+(\nu) + \delta^-(\nu). 
 \end{equation}

\medbreak

Let $\Proberg(f)\subset\Prob(f)$ be the subset of ergodic measures and, given $h>0$,  let  $\Proberg^h(f)\subset\Proberg(f)$ be the further subset of those ergodic invariant probability measures of $f$ with entropy at least $h$. 

\begin{corollary}\label{coroHDusc}
Let $f$ be a $C^\infty$ diffeomorphism of a closed surface. For each $0<h\le h_\top(f)$, the Hausdorff dimension $\nu\mapsto\HD(\nu)$ is upper semicontinuous over $\Proberg^h(f)$: Given any sequence of measures $\nu_k\in\Proberg^h(f)$ admitting a weak star limit $\mu$, if that limit is ergodic, then
 $$
    \HD(\mu) \ge \limsup_k \HD(\nu_k).
 $$
\end{corollary}

\begin{remark}
By Newhouse theorem \ref{thm-Newhouse}, the limit of any weakly converging sequence in $\Proberg^h(f)$ has entropy at least $h$. However to apply Theorem~\ref{maintheoremErgodic}, we need ergodicity which we have to assume, since it is automatic in our setting. Now, in contrast to $\{\nu\in\Prob(f):h(f,\nu)\ge h\}$, 
$\Proberg^h(f)$ is not compact, hence the above does not yield the existence of an ergodic measure maximizing the Hausdorff dimension, even in restriction to $\Proberg^h(f)$ for $h>0$.

We are able to remove the ergodicity assumption only when dealing with the more technical, thermodynamical dimensions   $\delta^+,\delta^-$.
\end{remark}

\begin{proof}[Proof of the Corollary]
Let $f$ and $h$ as above. From the variational principle and the existence of MMEs for $C^\infty$ diffeomorphisms, we have $\Proberg^h(f)\ne\emptyset$. Let $(\nu_k)$ be a sequence of measures in $\Proberg^h(f)$. Assume that it converges in the weak star topology to some limit $\mu$. Since $f$ is $C^\infty$, Newhouse theorem gives $h(\mu)\ge h$. By assumption, $\mu$ is ergodic so $\mu\in\Proberg^h(f)$. 

Given $H:=\limsup_k \HD(\nu_k)=\limsup_k \delta^+(\nu_k)+\delta^-(\nu_k)$, we can find a subsequence such that the following limits exist and satisfy the following conditions:
 $$
   \lim_k \delta^+(\nu_k)+\delta^-(\nu_k) = H,\; 
   \lim_k h(\nu_k),\;
   \lim_k \lambda^+(\nu_k),\;
   \lim_k \lambda^-(\nu_k). 
 $$
By Ruelle inequality, $\lim_k \lambda^+(\nu_k)\ge h>0$, hence we can apply Theorem~\ref{maintheoremErgodic}. The inequality from this theorem is equivalent to
 $$
    \lim_k \delta^+(\nu_k) = \frac{\lim_k h(\nu_k)}{\lim_k\lambda^+(\nu_k)} \le \frac{h(\mu)}{\lambda^+(\mu)}=\delta^+(\nu).
 $$
Note that this argument shows the upper semicontinuity of $\delta^+$ in $\Proberg^h(f)$. 

Noting that $\lambda^-(f,\cdot)=-\lambda^+(f^{-1},\cdot)$ and $h(f^{-1},\cdot)=h(f,\cdot)$, we see that  $\delta^-(f,\cdot)=\delta^+(f^{-1},\cdot)$, hence the same applies to $\delta^-$. We obtain
 $$
    \limsup_k \delta^+(\nu_k)+\delta^-(\nu_k) \leq \delta^+(\mu)+\delta^-(\mu).
 $$
Finally, the Ruelle-Margulis inequality implies that every $\mu\in\Proberg^h(f)$, for any $h>0$, is hyperbolic of saddle type, hence L.S. Young formula \eqref{eqYoungFormula} applies. It yields the upper semicontinuity of the Hausdorff dimension on $\Proberg^h(f)$.
\end{proof}

We noted that the above Corollary does not ensure the existence of an invariant ergodic measure with maximal Hausdorff dimension. We can ask whether this is just a technical problem.

\begin{question}
Given a smooth diffeomorphism $f$ of a closed surface, must $\sup\{\HD(\mu):\mu\in\Proberg(f)\}$ be achieved? What about $\sup\{\HD(\mu):\mu\in\Proberg(f)$ with $h(\mu)>0\}$? Or $\sup\{\HD(\mu):\mu\in\Proberg^h(f)\}$  for any fixed $0<h<h_\top(f)$?
\end{question}

\begin{remark}
In our approach, the difficulty comes from the fact that, for a given sequence $(\nu_k)$ as above, the decomposition $\mu=(1-\beta)\mu_0+\beta\mu_1$ provided by Theorem~\ref{theoremWithoutErgodicity} may be different when considering $f^{-1}$ instead of $f$. See Chapter~\ref{chapter-neutral}.
\end{remark}

\medbreak

\subsubsection*{Maximizing the unstable dimension}
In this case, we are able to deal with nonergodic limits to  find a measure maximizing the unstable dimension. Indeed, the non-necessarily ergodic version of our result, Theorem~\ref{theoremWithoutErgodicity}, gives the following existence result:

\begin{corollary}\label{coroUSCdelta}
Let $f$ be a $C^\infty$ diffeomorphism of a closed surface. For each $h>0$ such that $\Proberg^h(f)\ne\emptyset$, the function $\mu\mapsto\delta^+(\mu)$ achieves its supremum over $\Proberg^h(f)$.
\end{corollary}

Note that this is not an immediate consequence of the upper semicontinuity of $\delta^+$ over $\Proberg^h(f)$, since the latter is not necessarily compact.

\medbreak

We first make the following two observations:

\begin{lemma}\label{lem-no-sink-source}
If measures $\nu_k\in\Proberg(f)$ weakly converge to $\mu$ in the weak topology with $\nu_k(\{x:$ sink or source$\})=0$ for all large $k$, then the same holds for $\mu$, i.e., $\mu(\{x:$ sink or source$\})=0$.
\end{lemma}

\begin{proof}[Proof of the Lemma]
Assume by contradiction that $\mu(\{p\})>0$ for some source or sink $p$ in the setting of the lemma. Let $B$ be the forward or backward basin of $p$. The convergence $\nu_k\to\mu$ forces $\nu_k(B)>0$ for large $k$. By ergodicity, a.e. $\nu_k$-typical orbit  enters $B$. Hence $\nu_k$ is the periodic measure defined by $p$, a contradiction.
\end{proof}

\begin{lemma}\label{lem-dim-ergdec}
Let $f$ be a $C^1$ diffeomorphism of a closed surface and let $\mu\in\Prob(f)$ be such that its ergodic decomposition does not contain any periodic sink. If  $\delta^+(\mu)>0$, then  hyperbolic ergodic components $m^+$ with $\delta^+(m^+)\ge\delta^+(\mu)$ appear with positive measure in the ergodic decomposition of $\mu$.
\end{lemma}

\begin{proof}
Let $\mu = \int \mu_\xi\, d\xi$ be the ergodic decomposition of $\mu$. Note that $\lambda^+(\mu_\xi)\ge0$ a.e. (otherwise $\mu_\xi$ would be a sink). Let 
 $$
   H:=\{\xi:\lambda^+(\mu_\xi)>0>\lambda^-(\mu_\xi)\}.
 $$
We have
 $$
   \lambda^+(\mu) \stackrel{\rm def}{=} \int \lambda^+(\mu_\xi)\, d\xi \stackrel{(1)}{\ge} \int_H \lambda^+(\mu_\xi)\, d\xi \text{ and }
   h(\mu) = \int h(\mu_\xi)\, d\xi \stackrel{(2)}{=} \int_H h(\mu_\xi)\, d\xi,
 $$
where (1) follows from $\lambda^+(\mu_\xi)\ge0$ a.e. and (2) follows from Ruelle-Margulis inequality applied to $f$ and $f^{-1}$.
Therefore
 $$
   \delta^+(\mu) = \frac{\int h(\mu_\xi)\, d\xi}{\int \lambda^+(\mu_\xi)\, d\xi} \leq \frac
      {\int_H \delta^+(\mu_\xi) \cdot \lambda^+(\mu_\xi)\, d\xi}
      {\int_H \lambda^+(\mu_\xi)\, d\xi}.
 $$
Understanding the right hand side as an average, this inequality implies that the set $\{\xi : \delta^+(\mu_\xi)\ge \delta^+(\mu)\}$ has positive measure for $\lambda^+\cdot d\xi$ hence for $d\xi$.
\end{proof}

\begin{proof}[Proof of the Corollary]
Let $\nu_k\to\mu$ with $\nu_k\in\Proberg^h(f)$ for some $h>0$.
Since $\lambda^+(\nu_k)\ge h(\nu_k)\ge h>0$, Theorem~\ref{theoremWithoutErgodicity} applies to a subsequence with converging exponents and entropies and yields a decomposition $\mu=(1-\beta)\mu_0+\beta\mu_1$ satisfying $\beta>0$ and
 $$
    \limsup_k \delta^+(\nu_k)  = \limsup_k \frac{h(\nu_k)}{\lambda^+(\nu_k)} \le \frac{h(\mu_1)}{\lambda^+(\mu_1)} = \delta^+(\mu_1).
  $$
Note that $\mu_1$ is not necessarily ergodic or equal to $\mu$.

We apply this to a sequence of measures $\nu_k\in\Proberg^h(f)$ such that 
 $$
   \lim_k \delta^+(\nu_k)=\sup\{\delta^+(\nu):\nu\in\Proberg^h(f)\}>0.
 $$
To conclude, we first note that if $\mu_1$ had some periodic sink in its ergodic decomposition, this would be the case for $\mu$, but this is excluded by Lemma~\ref{lem-no-sink-source}. Hence we can conclude the proof of the corollary by lemma~\ref{lem-dim-ergdec}.
\end{proof}

\subsection*{Existence of S.R.B. measures}
Our results give some  natural characterizations of the existence of  hyperbolic \emph{Sinai-Ruelle-Bowen measure} (SRB for short), a fundamental class of measures in smooth ergodic theory \cite{LY1,LY2}. These measures $\mu$ are a subject of large interest \cite{Y-SRB} in part because, according to  \cite{Pugh-Shub}, they are \new{physical measures}: their ergodic basin, i.e., the set of points $x$ such that the weak limit $\lim_n (1/n)\sum_{k=0}^{n-1} \delta_{f^kx}$ exists and is $\mu$, has  positive volume. 

\begin{definition}
For surface diffeomorphisms, the hyperbolic \new{SRB measures} are the ergodic measures $\mu\in\Proberg(f)$ such that
 $$
    h(\mu)=\lambda^+(\mu)>0.
 $$
\end{definition}

The above is equivalent to $\delta^+(\mu)$ being well-defined and equal to $1$.
Therefore Corollary~\ref{coroUSCdelta} yields:

\begin{corollary}
Let $f$ be a $C^\infty$ diffeomorphism of a closed surface. There exists a hyperbolic SRB measure if and only if there is some $h>0$ satisfying
 $$
   \sup\{\delta^+(\mu):\mu\in\Proberg^h(f)\}=1.
 $$
\end{corollary}

One can translate the above in terms of \emph{geometric pressure} and the corresponding variational problem. If $\mu\in\Prob(f)$ has no sinks or sources in its ergodic decomposition, its unstable space is a.e. one-dimensional and therefore its geometric pressure is
 $$
    P_{\rm geo}(\mu) := h(\mu) - \lambda^+(\mu) \le 0.
 $$
This quantity is always nonnegative by Ruelle-Margulis inequality. Observe that among hyperbolic measures, the pressure is zero if and only if the thermodynamical unstable dimension satisfies $\delta^+(\mu)=1$. Thus:

\begin{corollary}
Let $f$ be a $C^\infty$ diffeomorphism of a closed surface. There exists a hyperbolic SRB measure if and only if there is some $h>0$ satisfying
 $$
   \sup\{P_{\rm geo}(\mu):\mu\in\Proberg^h(f)\}=0.
 $$
\end{corollary}

The above characterizations using Hausdorff dimension and geometric pressure are rather abstract and seem difficult to check. A further work \cite{BCS-3} using, not the results of \cite{BCS-2}, but the underlying techniques and ideas explained in these lectures established a celebrated conjecture of Viana \cite{Viana-ICM} in the case of  $C^\infty$ smooth surface diffeomorphisms: \emph{positive Lyapunov exponents on a set of positive Lebesgue measure imply the existence of an SRB measure with positive entropy.}
Burguet \cite{Burguet-SRB} independently established this result (together with a control on the ergodic basins).

\subsection*{Entropy-continuity of the exponent}
The original motivation of \cite{BCS-2} was to establish some uniform hyperbolicity estimates for ergodic measures with large entropy. These estimates lead to a spectral gap and many consequences \cite{BCS-4} along the lines of the \emph{strong positive recurrence} of symbolic dynamics \cite{Gurevich-Savchenko,Cyr-Sarig}. 

\medbreak

These estimates can be deduced from the following consequence of Theorem~\ref{theoremWithoutErgodicity}:

\begin{maincorollary}\label{maincorollary-ECLE}
Let $f$ be a $C^\infty$ smooth diffeomorphism of a closed surface with $h_\top(f)>0$.
Let $\mu_1,\mu_2,\dots$ be ergodic invariant Borel probability measures with weak star limit~$\mu$. 

If $\lim_{k\to\infty} h(\mu_k)=h_\top(f)$, then the top Lyapunov exponents also converge, as follows
 $$\lim_{k\to\infty} \lambda^+(\mu_k)=\lambda^+(\mu).$$
\end{maincorollary}

By Newhouse theorem \ref{thm-Newhouse}, $h(f,\mu)=h_\top(f)$, i.e., $\mu$ is a measure maximizing the entropy (or MME).

\begin{proof}
Let $f$, $\nu_1,\nu_2,\dots$, and $\mu$ be as above. Theorem~\ref{theoremWithoutErgodicity} yields some decomposition $\mu=(1-\beta)\mu_0+\beta\mu_1$ with $\mu_0,\mu_1\in\Prob(f)$ and $0<\beta\le1$ and such that
 $$
    h_\top(f) = \lim_k h(\nu_k) \le \beta \cdot h(\mu_1) 
    \text{ and }
    \beta = \lim_k \frac{\lambda^+(\nu_k)}{\lambda^+(\mu_1)}.
 $$
Since $h(\mu_1)\le h_\top(f)$, we have $0\ne h_\top(f)\le\beta\cdot h_\top(f)$ so $\beta=1$. Thus $\mu=\mu_1$ and $\lim_k \lambda^+(\nu_k)=\lambda^+(\mu)$.
\end{proof}

\begin{exercise}
Deduce from Theorem~\ref{theoremWithoutErgodicity} that any $C^\infty$ diffeomorphism of a closed surface has an MME \emph{without using Newhouse theorem} \ref{thm-Newhouse}.
\end{exercise}

In these notes, we will deduce
Corollary~\ref{maincorollary-ECLE} from the following  entropy bound.

\begin{maintheorem}\label{maintheorem-htop}
Let $f$ be a $C^\infty$ smooth diffeomorphism of a closed surface with $h_\top(f)>0$.
Let $\nu_1,\nu_2,\dots$ be ergodic invariant Borel probability measures of $f$ with weak star limit~$\mu$, possibly nonergodic.
Assume that $\lim_k \lambda^+(\nu_k)$ exists.

Then there is a decomposition
 $$
    \mu = (1-\beta)\mu_0+\beta\mu_1 \text{ with }
    0<\beta\le 1 \text{ and } \mu_0,\mu_1\in\Prob(f)
 $$
such that  
 $$
   \lim_{k\to\infty} \lambda^+(\nu_k)=\beta \cdot \lambda^+(\mu_1) \text{ and }
 \limsup_{k\to\infty} h(\nu_k) \le\beta\cdot h_\top(f).
 $$
\end{maintheorem}

This only implies a weaker version of Theorem \ref{theoremWithoutErgodicity}, where $h(\mu_1)$ in \eqref{ineq2} is replaced by $h_\top(f)$. However, this is enough to establish Corollary~\ref{maincorollary-ECLE} while avoiding a number of technical complications (see Chapter~\ref{chap-beyond}).

\section*{Proof of the results and organization of the lectures} 
Let us summarize our argument for Theorem~\ref{maintheorem-htop} and how it will be explained in these notes.

\medbreak
\subsection*{Projective dynamics and linear theory}
In chapter \ref{chap-oseledets}, we review the basic theory of Lyapunov exponents of linear cocycles and explain the classical formulas expressing the Lyapunov exponents as averages of the continuous dilation function with respect to well-chosen measures. These expressions involve the projective dynamics $\hf$ on $\hM:=\{(x,E):x\in M,\; E$ one-dimensional linear subspace of $T_xM\}$ defined by $(x,E)\mapsto (fx,Df_xE)$ and the dilation function
  $$
     \vf:\hM\longrightarrow\RR,\quad (x,E)\longmapsto \log  |D_xf_{|E}|_{x,fx}\;,
  $$
where $|A|_{x,fx}$ denotes the Jacobian of a linear map $A:T_xM\to T_{fx}M$ with respect to the Euclidean norms on $T_xM$ and $T_{fx}M$ defined by some given Riemannian structure on $M$. Since $E$ is one-dimensional this is the same as the operator norm $\|D_xf_{|E}\|$.

Given an ergodic measure $\nu$ on the surface $M$, the relevant measure is the $\hf$-invariant lift $\hnu^+$ to the graph defined by the Oseledets unstable sub-bundle
 $$
   \hGamma^+:=\left\{(x,E)\in\hM : \exists\lim_{n\to\pm\infty}\frac1n\log\|Df^{n}_x|E\|>0\right\}.
  $$
\medbreak

In light of these remarks, the starting point of the proof of Theorem~\ref{maintheorem-htop} is to a consider a sequence $\nu_k\wsto \mu$ (weak star convergence) with $h(\nu_k) \to h_\top(f)$. We write $\hnu_k^+$ for the corresponding $\hf$-invariant lifts. Up to passing to a subsequence, these measures weakly converge\footnote{We call \new{weak convergence}, resp. \new{weak topology}, what analysts call the weak star convergence, resp. weak star topology. That is the toplogy defined by continuous functions with compact support.} to some $\hmu$. The first key remark is that $\hpi_*(\hmu)=\mu$ (by continuity), but that $\hmu$ may differ from the unstable lift $\hmu^+$ of $\mu$.

\medbreak

\subsection*{Part 1: Discontinuity and neutral blocks}
Indeed, we first observe that any discontinuity of the Lyapunov exponent 
 $$
    \lambda^+(\mu)-\lim_k\lambda^+(\nu_k)= \hmu^+(\vf)-\hmu(\vf)
 $$
corresponds to a loss of mass of $\lim_k\hnu_k^+$ out of the unstable bundle $\hGamma^+$, i.e., $\hmu(\hGamma^+)<\hnu_k^+(\hGamma^+)=1$. This can happen, since the Oseledets unstable section $\hGamma^+$ is not necessarily closed, only measurable. By invariance, the mass can only escape to invariant bundles defined by nonpositive exponents.

In Chapter~\ref{chapter-neutral}, we relate this loss of mass to so-called \emph{neutral blocks} in the orbits of generic points $(x,E)$ for $\tnu^+_k$ for $k$ large: the neutral blocks in the orbit of some $(x,E)$ are long orbit segments of integers $k$ for which the unstable direction $D_xf^k(E)$  does not grow significantly (typically because it is too close to the stable or neutral direction as implied by the loss of mass). More formally, for $\alpha>0$ small and $L\ge1$ large, $[a,b[$ is an $(\alpha,L)$-neutral block if
 $$
   \forall 0\le i<b-a\quad \| Df^{i}|E^+_{f^a(x)} \| \le e^{\alpha i}
 $$
(with respect to a very small number $\alpha>0$). We will see that neutral blocks are contained in maximal ones and that one can define the union $\operatorname{Neutral}(\alpha,L)$ of all neutral blocks in an equivariant way.

We describe the neutral dynamics by considering  measures restricted to the union of neutral blocks in the appropriate limit (possibly for some further subsequence):
 $$
   \lim_{\tiny\begin{array}{c}\alpha\to0\\ L\to\infty\end{array}}\lim_k 1_{\operatorname{Neutral}(\alpha,L)}\cdot\hnu_k^+=(1-\beta)\hmu_0 \text{ where }\hmu_0\in\Prob(\tf) \text{ and }\beta\in[0,1].
 $$
Setting $\beta\cdot\hmu_1:=\hmu-(1-\beta)\hmu_0$, we have obtained a decomposition
 $$
   \hmu=(1-\beta)\hmu_0+\beta\cdot\hmu_1.
 $$
The decomposition announced in Theorem~\ref{theoremWithoutErgodicity} is the projection of the above decomposition to the surface.

The two parts $\hmu_0$ and $\hmu_1$ describe respectively the behavior inside and outside of the neutral blocks. When the limit $\mu$ happens to be ergodic, both $\hmu_0$ and $\hmu_1$ must project to $\mu$. In the general case, things can be more complicated. But, in any case the neutral part satisfies $\hmu_0(\vf)=0$ and thus gives no contribution to the limit of the exponent, so
 $$
   \lim_k \lambda^+(\nu_k)=\int \phi\, d\tmu=\beta\int \phi\, d\tmu_1=\beta\cdot \lambda^+(\mu_1).
 $$

\bigbreak

The second and third parts of the proof deduce the consequences of the above in terms of entropy. In chapter~\ref{chap-entropy}, we review some basic facts about the entropy of maps.  We could deduce the following entropy bound from the above decomposition:
 $$
     \limsup_k h(\nu_k)\le\beta\cdot h(f,\mu_1).
 $$
To avoid extra complications, we will show the weaker bound: $\limsup_k h(\nu_k)\le\beta\cdot h_\top(f)$ as stated in Theorem~\ref{maintheorem-htop}.

The proof of this entropy bound (parts 2 and 3 below) will require a more complicated non-linear argument in order to extend the information given by the differential to unstable curves around typical orbits. Indeed, we will show that ``there is no entropy contribution from the neutral blocks''.

\medbreak
\subsection*{Part 2: Entropy from parametrizations of curves}
To exploit the action of the differential $Df$ on the unstable direction we   bound the entropy of a measure  using the behavior of its generic unstable curves.  In fact (see \S\ref{cor-entropy-from-unstable}), there is a \emph{formula of Ledrappier and Young} expressing the entropy of an ergodic measure as the complexity of its unstable disintegrations along unstable Pesin manifolds. These Pesin manifolds are, in our case, smooth curves tangent to the unstable bundle at typical points and the dynamics along this bundle is precisely what is controlled by the neutral blocks.

To implement this control, we will use ``good parametrizations'', i.e., parametrizations at small scales (both in position and direction) with bounds on their derivatives. This is provided by \emph{Yomdin's theory} applied to the projective dynamics which we will explain in Chapter~\ref{chap-yomdin}.

\medbreak
\subsection*{Part 3: Nonlinear estimates}
The last part of the proof of Theorem~\ref{maintheorem-htop} is more delicate. We need to estimate the growth of the number of good reparametrizations of typical unstable curves. The information at our disposal is two-fold:
 \begin{enumerate}
  \item the long neutral blocks;
  \item the topological entropy.
\end{enumerate}
How to use (1) is the technical core of the proof. It is presented in Chapter~\ref{chap-reparam-neutral}. Of course, the differential controls  the expansion along the unstable curve on the surface. But to use the differential at \emph{one point} requires the tangent to the curve to be almost constant, while, in general, even curves whose lengths stay small may have exponentially growing oscillations. 
To solve this difficulty, we use that the unstable direction is a measurable section. 
Now,  a measurable function has some modulus of continuity over a set of measure close to $1$ by Lusin theorem.
Thus a tight control in position will yield some control in direction. We will propagate these two joint controls using what we call below a double scale: one on the surface and one on the projective extension. This will show that neutral blocks do not contribute to entropy even in the projective extension.

\medbreak

Outside of these neutral blocks, we will use (2), yielding a control in terms of dynamical balls by the topological entropy. A standard application of Yomdin's theory will translate the number of dynamical balls into an estimate on the number of reparametrizations so that we will be able to use alternately (1) and (2), as explained in Chapter~\ref{chap-proof-main}.

\medbreak

In the final chapters \ref{chap-beyond}~and~\ref{chap-further}, we explain the additional ideas needed for the more powerful Theorems \ref{maintheoremErgodic} and \ref{theoremWithoutErgodicity} and then point out some further developments and a few open problems.

\medskip

{\bf Thanks.} The author would like to thank Hengyi LI for his careful reading which helped me remove a number of mistakes.
The author is also indebted to the anonymous referee for a very careful reading of a first version of these notes which improved the writing significantly.

\numberwithin{section}{chapter}

\chapter*{Some notations}
\label{chap-notations}
In these lectures, we consider the following objects and quantities:
 \begin{itemize}
  \item $M$ is a closed surface (compact, boundaryless manifold of dimension $2$);
  \item $f\in\Diff^r(M)$ with $r\ge1$, i.e., $f$ is a $C^r$-smooth diffeomorphism of $M$, often $r=\infty$;
  \item $\|\sigma\|_{C^r}$ is the $C^r$ size of some parametrized disk $\sigma:[0,1]\to M$, see Definition~\ref{def-Cr-size};
    \item $\hf:\hM\to\hM$ is the projective extension of $f:M\to M$ through the factor map $\hpi:\hM\to M$;
  \item $B_f(x,\eps,n)$ is the dynamical ball of order $n$ with center $x$ and radius $\eps$, see~\eqref{eqDefBn}; $B^u_f(x,\eps,n)$ is the unstable dynamical ball, see~\eqref{eqDefBun}; $B^u_{\hf}(x,\eps,\heps,n)$ is the doubly scaled dynamical ball in $\hM$, see~\eqref{eqDefBdouble};
  \item $h_\top(f)$ is the topological entropy (see Definition~\ref{def-htop-Bowen});
  \item If $\mu$ is a positive Borel measure and $g\in L^1(\mu)$, then we denote by $\mu(g)$ the integral $\int g\, d\mu$.
  \item $\Prob(M)$ denotes the compact, metrizable set of all Borel probability measures. It is equipped with the weak topology provided by the identification of finite measures with linear functionals over continuous functions\footnote{Since $M$ is compact, those functions are bounded and therefore weak star topology coincides with the weak topology from probability theory.};
  \item $\Prob(f)\subset\Prob(M)$ denotes the subset of all invariant measures;
  \item $\Proberg(f)\subset\Prob(f)$ denotes the subset of all ergodic measures. Geometrically these are the extremal points of $\Prob(f)$. In many cases, it is a dense proper subset of $\Prob(f)$;
  \item If $h:M\to N$ is a measurable map, its pushforward is $h_*:\Prob(M)\to\Prob(N)$ is defined by $f_*\mu = \mu\circ f^{-1}$.
\end{itemize}  
Given $\mu\in\Prob(f)$, we consider:
 \begin{itemize}
  \item the (average)  top Lyapunov exponent: $\lambda^+(f,\mu)$ (see Definition~\ref{def-top-average-LE}); 
  \item the (average)  bottom Lyapunov exponent: $\lambda^-(f,\mu):=\lambda^+(f^{-1},\mu)$;
  \item the Lyapunov spectrum at $x\in M$ is $\lambda^1(f,x)>\lambda^2(f,x)>\dots>\lambda^{r(x)}(f,x)$;
  \item the Kolmogorov-Sinai entropy: $h(f,\mu)$ (see Definition~\ref{def-hKS-Katok});
  \item $W^u(f,x),W^s(f,x)$, the unstable and stable manifolds (see Definition~\ref{def-WuWs});
  \item $\{\mu^u_x\}_{x\in M},$ some unstable disintegration of $\mu$ (see Definition~\ref{def-u-disint});
 \end{itemize}
For background on dynamical systems, entropy, and exponents, we refer the reader to introductory texts such as \cite{KH-book} or \cite{Mane-book}.
Some other useful notations that are perhaps not universal are:
 \begin{itemize}
  \item $\NN:=\{0,1,2,\dots\}$, $\ZZ:=\{\dots,-1,0,1,\dots\}$, $\RR$ is the set of real numbers and $\CC$ is the set of complex numbers;
  \item  $\RR.v$ is the one-dimensional subspace generated by some nonzero $v$ in a vector subspace;
  \item $C^{k+\alpha}$ with $k\in\NN$ and $0\le\alpha<1$ denotes the maps that  can be differentiated $k$ times with a $k$th differential that is $C^{\alpha}$ ($C^0$ means continuous and $C^\alpha$ for $0<\alpha<1$ means H\"older with exponent $\alpha$);
  \item $D_xf$ or, sometimes, $f'(x)$ denotes the differential of a map $f$ at a point $x$ (the second notation is often used when $x\in\RR$);
  \item $D^s_xf$ denotes the differential of $f$ of order $s\in\NN$ at the point $x$ (considered as an $s$-multilinear map).
 \end{itemize}

\part{Linear Theory}

\begin{center}
{\Large Introduction}
\end{center}
\bigbreak

This part is devoted to the linear part of our argument. We start by reviewing basic notions, mostly without proof (but with references), in the general setting of vector bundles:
\begin{enumerate}
 \item the definition of \emph{Lyapunov exponents} via Kingman's theorem (Theorem~\ref{thm-Kingman});
 \item the \emph{Oseledets sub-bundles} given by Oseledets' theorem (Theorem~\ref{thm-Oseledets});
 \item the projective dynamics (Definition~\ref{def-projective}) and the Lyapunov exponents as averages of the  continuous \emph{dilation function} there (Proposition~\ref{prop-lift-projective}).
\end{enumerate}

Then we perform the linear part of our argument:
 \begin{enumerate}
  \item we explain how to read the Lyapunov exponents in terms of ergodic projective lifts (see Proposition~\ref{prop-lift-simple}) and relating exponent discontinuities to loss of mass (see Proposition~\ref{prop-disc-LE-loss-mass}).
  \item we build the neutral blocks and the corresponding decomposition of limiting measure on the projective dynamics (Theorem~\ref{thm-neutral-decomposition-projective}).
\end{enumerate}  

\chapter{Lyapunov exponents and Oseledets spaces}

\section{Vector bundles and linear cocycles}\label{chap-oseledets}

It is convenient to formulate the results of this chapter in the setting of topological vector bundles. The reader who is familiar with the basic notions of a bundle (mainly vector or projectivized bundles) can skip this section.

\subsection{General bundles}
A smooth $G$-\new{bundle} over $X$ with fiber $F$ is a manifold which ``locally looks like'' a product of two manifolds $X\times F$ and such that the changes of coordinates along $F$ are restricted to some group $G$. More formally:

\begin{definition}
A \emph{smooth $G$-bundle over $X$ with fiber $F$} is given by manifolds $V,X,F$,  and a group $G$ smoothly acting on $F$ 
such that
there are a covering of $V$ by open sets $V_i$ and diffeomorphisms $h_i:V_i\to U_i\times F$ with $U_i$ open sets of $X$ satisfying:
 \begin{itemize}
  \item[--] there is a submersion\footnote{A \new{submersion} is a smooth map $p:V\to W$ such that, for every $x\in V$, $D_xp:T_xV\to T_{p(x)}W$ is onto.} $p:V\to X$ satisfying $p(x)=\proj_X(h_i(x))$ for any $i\in I$ such that $x\in U_i$. We set $V_x:=p^{-1}(x)$.
 \item[--]  if $x\in U_i\cap U_j$, then, 
 writing $\proj_F:(x,v)\mapsto v)$, the map
 $g_{x,ij}:F\to F$ given by 
 $$
    g_{x,ij}:v\mapsto \proj_F\circ h_j\circ h_i^{-1}(x,v) 
 $$
belongs to $G$.
 \end{itemize}
When it creates no confusion, we will often denote the bundle $(G,F,V,X,h)$ just by the submersion $p:V\to X$.
The collection $\{h_i:i\in I\}$ is a called a \emph{bundle atlas}.  The group $G$ is the \emph{structure group} of $F$, the \emph{fiber model} (or just the fiber). The spaces $V$ and $X$ are called the \emph{total} and \emph{base} spaces. 
\end{definition}

The notion of bundle generalizes that of product:
 
\begin{example}
Given a pair of manifolds $X,F$ and some group $G$ acting smoothly on $F$, $V:=X\times F$ can be seen as a $G$-bundle over $X$ with space $F$ by setting $V_*:=V$, $U_*=X$, $h_*=\Id:V_*\to U_*\times F$. This is the \emph{trivial bundle}.
\end{example}

\subsection{Section of a bundle}

\begin{definition}
A \new{section} of a bundle $p:V\to X$ is a map $\Gamma:X\to V$ with $p\circ\Gamma=\Id_X$. A \emph{local section} at $x\in X$ is a map $\Gamma:U\to V$ such that $U$ is a neighborhood of $x$ and $p\circ\Gamma=\Id_U$.
\end{definition}

A vector field over some smooth manifold is exactly a section of its tangent bundle.

\begin{exercise}
For every $x\in X$, build a continuous local section at $x$. Do continuous sections always exist?
\end{exercise}

We note that \emph{measurable} sections always exist.

\subsection{Vector and projective bundles}
These are the most relevant classes of bundles for us. We will denote by $\GL(W)$  the linear group of a vector space $W$ and by $\Orth(E)$ is the group of isometries of the Euclidean space $E$.

\begin{definition}
A \new{vector bundle} of dimenson $d$ is a bundle whose fiber is $\RR^d$ and group is $G\subset\GL(\RR^d)$.
\end{definition}

The tangent bundle is the most important type of vector bundles for us:

\begin{example}
Let $M$ be a smooth manifold of dimension $d$. Then its \new{tangent bundle} $TM=\bigsqcup_{x\in M} T_xM$ has the natural structure of a smooth vector bundle of dimension $d$.
\end{example}

\begin{definition}\label{def-Riemann-bundle}
A \new{Riemannian vector bundle} is a vector bundle $p:V\to M$ together with a continuous family $\|\cdot\|_{x\in M}$ of Euclidean norms on the fibers $V_x$.
\end{definition}

\begin{example}
For the tangent bundle of a smooth Riemannian manifold, a Riemannian structure is just the usual notion.
\end{example}

\begin{exercise}
Check that a vector bundle of dimension $d$, $p:V\to M$, with  a Riemannian structure has a natural structure of a smooth $\Orth(\RR^d)$-bundle over $M$ with fiber $\RR^d$.
\end{exercise}

\begin{exercise}
Let $M$ be a smooth Riemannian manifold and let $N$ be a sub-manifold. The \new{normal bundle} at $N$ is $N^\perp:=\bigcup_{x\in N} (T_xN)^\perp$. More precisely, it is the vector bundle $p:N^\perp\to N$ with fibers $N^\perp_x:=\{v\in T_xM:\forall w\in T_xN\; v\cdot w=0\}$. 

Check that $N^\perp$ is indeed a smooth vector bundle with a Riemannian structure.
\end{exercise}

\begin{definition}
Given a Riemannian vector bundle $p:V\to M$, the \new{unit (vector) bundle} is the restriction $V^1:=\bigsqcup_{x\in M} \{ v\in V_x:\|v\|_x=1\}$.
\end{definition}

\begin{exercise}
Given a vector bundle $p:V\to M$, show that its unit vector bundle $p:V^1\to M$ is a smooth bundle with model fiber $\mathbb S^{d-1}$ and group $\Orth(\RR^d)$. 
\end{exercise}

We will also consider the projectivization of vector bundles:

\begin{definition}\label{def-projectivization}
If $p:V\to X$ is a  vector bundle with fiber model $\RR^d$ and group $G\subset \GL(\RR^d)$, then its \new{projectivization} is the $\PP(G)$-bundle $\hp:\hV\to X$ with:
 \begin{itemize}
  \item[--] structure group: $\PP(G):=\{\RR.v\mapsto \RR.(A.v):A\in G\}$;
  \item[--] fiber model: the projective space: $\PP(F):=\{E\subset F:E$ one-dimensional vector subspace$\}/\sim$ with the natural topology of the  quotient of $F\setminus 0$ by the map $v\mapsto\RR.v$;
  \item[--] fibers: $\hV_x:=\PP(V_x)$.
 \end{itemize}
\end{definition}

An obvious but nice feature of the projectivization of a vector bundle is that the fibers are compact.

\begin{exercise}
Show that the projectivization of a vector bundle of dimension $d$ is indeed a $\PP(\GL(\RR^d))$-bundle.
\end{exercise}

\begin{definition}
The \emph{sum} $V^1\oplus\dots \oplus V^k$ of a finite family of vector bundles $p_i:V^i\to X$, $i=1,\dots,k$, with fiber models $\RR^{d_i}$ over the same space $X$, is the vector bundle $p:V\to X$ with fiber model $\RR^{d_1+\dots+d_k}$ over the same space $X$ defined by taking the fibers $V_x$ to be the abstract direct sum $V^1_x\times\dots \times V^k_x$.
\end{definition}

When, moreover, all the bundles $V^1,\dots,V^k$ are contained in some common vector bundle $W$ (i.e., each $V^i_x$ is a linear subspace of $W_x$ for all $x\in X$), one can define $V_x:=V^1_x\oplus\dots\oplus V^k_x\subset W_x$ for all $x\in X$.

\begin{exercise}
In the case where all the bundles $V^1,\dots,V^k$ are contained in some common vector bundle $W$, formalize the above definition by defining precisely the space and some atlas. 
\end{exercise}

\begin{remark}
We leave to the diligent reader the adaptation of these topological definitions to more diverse situations, e.g., measured bundles where the base is a standard probability space (i.e., equipped with a $\sigma$-field which is the Borel $\sigma$-field of some Polish topology). Such bundles are a natural setting for the Oseledets theorem below.
\end{remark}

\subsection{Bundle morphisms}

A morphism between bundles is a surjection that maps fiber to fiber respecting the $G$-structure. We first define:

\begin{definition}
Given a bundle $p:V\to X$ with fiber $F$ and group $G$, a \new{fiber map} is a map $L:V_x\to V_y$ for some $x,y\in X$ satisfying the following. If $x\in U_i$, $y\in U_j$, then, defining $h_{i,x}:=(\proj_F\circ h_i:V_x\to F)$ and similarly $h_{j,y}$, we have
 \begin{equation}\label{eq-fiber-G}
     h_{j,y}\circ L \circ h_{i,x}^{-1} \in G.
 \end{equation}
\end{definition}

For instance, if $V$ is a vector bundle, then fiber maps are linear maps between fibers.

\begin{definition}
Let $V,V'$ be two $G$-bundles with bases $X$ and $X'$ and the same fiber  model $F$ and group $G$. Let $(h_i:V_i\to U_i\times F)_{i\in I}$ and $(h_j:V'_j\to U'_j)_{j\in J}$ be bundle atlases for $V$ and $V'$.

A \emph{bundle morphism} from  $V$ to $V'$ is a continuous map $\cF:V\to V'$  such that
 \begin{itemize}
  \item[--] there is a continuous map $f:X\to X'$ such that $p\circ\cF=f\circ p$;
  \item[--] for all $x\in V$, $\cF_x:V_x\to V_{fx}$ is a fiber map.
\end{itemize}

A \new{bundle endomorphism} is a bundle morphism between a bundle and itself.

A \new{bundle isomorphism} is a bundle morphism $F:V\to V'$ such that $F^{-1}:V'\to V$ is a well-defined map and also a bundle morphism. A \new{bundle automorphism} is an isomorphism of a bundle to itself. 
\end{definition}

Note that the bundle morphism $\cF:V\to V'$ uniquely determines the base map $f:M\to M'$ as $f(x)=p\circ \cF(v)$ for any $v\in V$. 

\begin{exercise}
Show that the base map is well-defined and automatically smooth. Check that if, for $i=0,\dots,r-1$, $\cF_i:V_i\to V_{i+1}$ is a morphism of bundles over some map $f_i:M_i\to M_{i+1}$, then $\cF_{r-1}\circ\cF_{r-2}\circ\dots\cF_0:V_0\to V_{r}$ is a morphism  over $f_r\circ\dots\circ f_0:M_0\to M_{r}$.
\end{exercise}

Bundle morphisms can be seen as cocycles that we define now.
Let $p:V\to X$ be a bundle. For $x,y\in X$, $\operatorname{Mor}(V_x,V_y)$ is the set of maps $L:V_x\to V_y$ satisfying \eqref{eq-fiber-G}.

\begin{definition}
Given a self-map $f:X\to X$, a \new{cocycle} in $V$ over $X$ is a section $A:X\to \Endo_f(V)$, where $\Endo_f(V)$  is the bundle over $X$ defined by $\bigsqcup_{x\in X} \operatorname{Mor}(V_x,V_{fx})$. 

The associated bundle endomorphism is
 $$
    F_A:V\longrightarrow V,\; v\longmapsto A_{\proj_X(v)}(v)
 $$
\end{definition}

\begin{exercise}
Check that bundle endomorphisms and cocycles can be canonically identified.  
\end{exercise}

\begin{definition}
Let $p:V\to X$ be a bundle and $f:X\to X$ be a smooth map.

The \new{composition of two cocycles} $A:X\to\Endo_f(V)$ and $B:X\to\Endo_g(V)$ is the cocycle $A\circ B:X\to\Endo_{f\circ g}(V)$ defined by
 $$
     (A\circ B)_x = A_x\circ B_x.
 $$
If $f:X\to X$ is a diffeomorphism, the \new{inverse cocycle} of a cocycle $A$ is a cocycle $A^{-1}:X\to\Endo_{f^{-1}}(V)$ such that $A^{-1}\circ A=A\circ A^{-1}=\Id_V$, the \new{identity cocycle} equal to the identity on each fiber of $V$. 

For $n\in\NN$, the \new{cocycle iterated $n$ times} is:
 $$
    A^n_x(v) := A_{f^{n-1}x} \circ A_{f^{n-2}x}\circ \dots \circ A_x.
 $$
If $A$ is invertible, then $A^{-n}:=(A^{-1})^n$.
\end{definition}

\begin{exercise}
Show that $F_{A\circ B} = F_A\circ F_B$ and check the \new{cocycle identity} $A^{n+m}_x=A^n_{f^mx}\circ A^m_x$.
\end{exercise}

Diffeomorphisms are the key examples of vector bundle isomorphisms for our purposes:

\begin{example}
If $M,M'$ are smooth manifolds and $f:M\to M'$ is a diffeomorphism, then $f$ defines an isomorphism of the tangent bundles: $Df:TM\to TM'$ according to the natural formula: $Df(x,v)=(f(x),D_xf.v)$ for any $(x,v)\in TM$.
\end{example}

Given a measure on the base, one can define various ``norms'' of bundle morphisms (we skip specifying the spaces over which these functionals are indeed norms):

\begin{definition}\label{def-operator-norms}
A Riemannian structure on a vector bundle $V$ induces a family of operator norms $(\|\cdot\|_{x,y})_{x,y\in M}$ as follows. If $A:V_x\to V_y$ is a fiber morphism then
 $$
   \|A\|_{x,y} := \sup_{v\in V_x,\; \|v\|_x\le 1} \|A(v)\|_{y}.
 $$
\end{definition}

\begin{exercise}
Given a Riemannian structure $(\|\cdot\|_x)_{x\in M}$ on $V$, show that for every bundle morphism $\cF:V\to V$ the following hold:
 \begin{itemize}
  \item[--] $x\mapsto \| \cF_x \|_{x,f(x)}$ is well-defined and measurable;
  \item[--] $\|\cF_x^n\|_{x,f^n(x)}\le \prod_{k=0}^{n-1}\|\cF_{f^kx}\|_{f^kx,f^{k+1}x}$.
 \end{itemize}
\end{exercise}

\begin{definition}
If $V\to M$ is a Riemannian vector bundle and $\cF$ a bundle automorphism, then the \new{$L^p(\mu)$-norm of a bundle morphism} $\cF$  with respect to $\mu\in\Prob(M)$ is
 $$\begin{aligned}
    &\|\cF\|_{L^p(\mu)} := \left( \int_M \|\cF_x\|_{V_x,V_{fx}}^p \,d\mu(x)\right)^{1/p} \text { if }p\in[1,\infty),\\
    &\|\cF\|_{L^\infty(\mu)} := \sup\left\{t>0:\mu(\{x\in M:\|\cF_x\|_{V_x,V_{fx}}>t\})>0\right\} .
 \end{aligned}$$
\end{definition}

\subsection{Projectivization}

The projective bundle is usually not the image of the vector bundle by a morphism. Instead we have:

\begin{exercise}
If $p:V\to X$ is a vector bundle, $p:V^1\to M$ the unit bundle and $\hp:\hV\to X$ is its projectivization, then
 $
    (x,v) \mapsto (x,\RR.v)
 $
is a bundle morphism from $V^1$ tp $\hV$.

Show by an example that $\hV$ is not necessarily the image of $V$ by a morphism.
\end{exercise}

The projectivization is functorial as described in the following.

\begin{exercise}
If $L:V\to V'$ is a vector bundle morphism then its projectivization $\PP(L):\PP(V)\to\PP(V')$ is a (projectivised) bundle morphism. If $L$ is an isomorphism, then so is $\PP(L)$.
\end{exercise}

\section{Ergodic theorems for Lyapunov exponents}

We now define the Lyapunov exponents of a bundle automorphism (recall that it is the same as a $\ZZ$-cocycle). Our setting is as follows. We are given:
 \begin{itemize}
  \item  a diffeomorphism $f:X\to X$ preserving a Borel probability measure $\mu$;
  \item a smooth, Riemannian vector bundle $p:V\to X$ over the previous compact manifold $X$;
   \item a bundle automorphism $F:V\to V$.
\end{itemize}   
The Riemannian structure on $V$ is a family of norms $(\|\cdot\|_{x})_{x\in M}$ according to Def. \ref{def-Riemann-bundle}. It determines operator norms $(\|\cdot\|_{x,y})_{x,y\in X}$ as in Definition~\ref{def-operator-norms}.

\subsection{Top Lyapunov exponent}

\begin{definition}\label{def-top-average-LE}
The \new{top Lyapunov exponent (average)} of a cocycle $F$ in $V$ over the basis $(f,\mu)$ is the following limit (if it exists):
 \begin{equation}\label{eq-def-top-exp}
     \lambda^+(F,\mu):=\lim_n \frac1n\int_X \log\|F^n_x\|_{x,f^n(x)} \, d\mu(x).
  \end{equation}
The \new{top Lyapunov exponent (pointwise)} at $x\in X$ is the following limit,  if it exists:
 $$
    \lambda^+(F,x):=\lim_n  \frac1n \log\|F^n_x\|_{x,f^n(x)}.
 $$
\end{definition}

The key remark is that:
$\vf(x,n)= \log\|F_x^n\|_{x,f^n(x)}$ is a \new{subadditive process} on $(f,\mu)$, i.e., for all $x\in M$, $n,m\in\NN$, it satisfies
 \begin{equation}\label{eq-subadd}
 \vf(x,n+m)\le \vf(x,m)+\vf(f^mx,n).
 \end{equation}
Such a process is said to be \emph{integrable} if $\vf(x,n)\in L^1(\mu)$ for all $n\ge1$; it is said to be \emph{bounded} if $|\vf(x,n)|\le C\cdot n$ for all $n\ge1$.

The \new{constant of a sub-additive process} is defined as
 \begin{equation}\label{eq-const-subadd}
   \gamma(\vf) := \inf_{n\ge1} \frac1n\int \vf(x,n)\, \mu(dx).
  \end{equation}
It is easily checked that $\gamma(\vf)\in [-\infty,\int_X \vf(x,1)\,d\mu]$ if the process is integrable and that $|\gamma(\vf)|\le C$ if it is bounded.

\begin{remark}
The natural setting here (and to apply Kingman's ergodic theorem below) is to consider a measurable bundle $p:V\to X$ over a probability endomorphism $(X,\mu,f)$ together with a measurable family of pseudo-norms $(\|\cdot\|_{x,y})_{x,y\in X}$ satisfying:
 \begin{itemize}
  \item[--] $(x,y)\mapsto \|\cF_x\|_{x,f(x)}$ is measurable and finite $\mu$-a.e.;
  \item[--] for any $x\in X$ and any bundle endomorphisms $\cF,\cG:V\to V$
 $$
    \| \cF_{f(x)}\circ\cF_x \|_{x,f^2(x)} \le \|\cF_{f(x)}\|_{f(x),f^2(x)} \cdot \|\cF_x \|_{x,f(x)}.
  $$
 \end{itemize}
\end{remark}

\subsection{Existence of the average top exponent}
The average exponent has very nice properties as a direct consequence of the subadditivity and Fekete lemma.

\begin{proposition}\label{prop-average-exp}
Let $F:V\to V$ be a bundle endomorphism of a bundle with a Riemannian structure.

If $\log^+\|F_x\|_{x\in X}\in L^1(\mu)$, then the average top Lyapunov exponent \eqref{eq-def-top-exp} exists in $\RR\cup\{-\infty\}$. It does not depend on the choice of the Riemannian structure on $V$.
\end{proposition}

We note that the average exponent $\lambda^+(F,\mu)$ is the constant of the associated subadditive process.

\medbreak
The proof the above proposition relies on the following elementary fact:

\begin{lemma}[Fekete]
If $a_1,a_2,\dots$  is a sequence of real numbers which is subadditive $(\forall n,m\; a_{m+n}\le a_m+a_n)$ then the following limit exists:
 $$
    \lim_n \frac1n a_n = \inf_{n\ge1} \frac1n a_n \in [-\infty,a_1].
 $$
\end{lemma}

\begin{exercise}
Deduce Fekete's lemma from the inequality $a_{qn+r}\le q\cdot a_n+a_r$ for all integers $q,n,r\ge1$.
\end{exercise} 

\begin{proof}[Proof of the Proposition]
The existence of the average exponent follows from the subadditivity of the  sequence $a_n(\mu):=\int_X \log\|A^n_x\|_{x,f^n(x)} \, d\mu(x)$, a direct consequence of \eqref{eq-subadd}. Therefore $a_n(\mu)/n$ converges to $\inf_{n\ge1} a_n(\mu)/n$ by Fekete's lemma. 
\end{proof}

\subsection{Semicontinuity of the average top exponent}

Given a cocycle $F:V\to V$, we are interested in the dependence of the exponent \emph{on the measure} $\mu$, i.e., on the map
 $$
   \mu\in\Prob(f) \longmapsto \lambda^+(F,\mu)
 $$
and its continuity properties.

To define these, we \emph{choose a topology} on $\Prob(f)$. We consider, as usual, the classical weak * topology defined by the linear forms $\mu\mapsto \mu(\phi)$ for each $\phi\in C_c(X)$. It is a metrizable topology related to the Wasserstein distance of optimal transport and, when $X$ is compact, it turns $\Prob(f)$ into a compact set.  We often call it \emph{weak topology}.

The key point is that the functions appearing in \eqref{eq-def-top-exp}, i.e.,
 $$
    x \longmapsto \frac1n\log\|F_x^n\|_{x,f^n(x)} \qquad (n\ge1)
 $$
are continuous.

\begin{exercise}
Check this continuity.
\end{exercise}

Hence, by definition of the weak topology the functions
 \begin{equation}\label{eq-top-exp-n}
   \mu \longmapsto \frac1n\int \log\|F_x^n\|_{x,f^n(x)}\, d\mu \qquad (n\ge1)
  \end{equation}
are continuous. This fact is enough, applying Fekete lemma, to see that the constant of the relevant subadditive process, as an infimum of continuous functions, is upper-semicontinuous, as stated below.

\begin{proposition}\label{prop-average-exp-usc}
The average top Lyapunov exponent 
 $
    \mu\in\Prob(f)\longmapsto \lambda^+(A,\mu)$ is upper semicontinuous wrt the weak  topology. 
\end{proposition}

\begin{remark}
Given a fixed endomorphism $F:V\to V$ of a Riemannian vector bundle, the map $\mu\mapsto \lambda^+(F,\mu)$ is not necessarily lower semicontinuous.
\end{remark}

\begin{proof}[Proof of the remark]
A counterexample to continuity is provided by the cocycle: $X=\{0,1\}^\ZZ$, $f:(x_n)\mapsto (x_{n+1})$, $\mu_t$ is the Bernoulli measure $(t\delta_0+(1-t)\delta_1)^{\ZZ}$ for $0\le t\le 1$ and
 $$
    A(x)=\mat{ 0 & -1\\ 1 & 0} \text{ if $x_0=0$},\quad \mat{2 & 0 \\ 0 & 1/2} \text{ otherwise.}
  $$
We then define $F:V\to V$ by $F_x(v):=A(x).v$.
It is easily checked that $\lambda^+(F,\mu_t)=\log 2$ if $t=0$ but vanishes for all $0<t\le 1$.
\end{proof}

\subsection{Pointwise top exponent}
Recall Kingman's ergodic theorem \cite[Thm 5.2, p.36]{Krengel-book} which studies subadditive processes.

\begin{theorem}[Kingman]\label{thm-Kingman}
Let $\vf_n:X\to\RR$ be a subadditive process for some probabilistic dynamics $(X,\mu,f)$ such that $\vf_n\in L^1(\mu)$ for each $n\ge1$. 

Then the following limit exists $\mu$-a.e.
 $$
   \overline{\vf}(x):=\lim_n \frac1n\vf_n(x).
 $$
Moreover, this limit is $f$-invariant modulo $\mu$.  If the constant $\gamma$ of the process \eqref{eq-const-subadd} is finite, then  $\overline{\vf}$ is integrable with $\int_X \overline{\vf}\, d\mu=\gamma$. Additionally, the convergence above occurs also in $L^1(\mu)$.
\end{theorem}

Since $X$ is compact and $F$ is continuous, $\log\|F_x\|_{x,f(x)}$ is bounded, hence Kingman's theorem immediately yields the following fact:

\begin{theorem}\label{thm-pw-exponent}
Let $p:V\to X$ be a smooth vector bundle with $X$ a compact manifold and $\cF:V\to V$ a bundle automorphism.

The pointwise Lyapunov exponent $\lambda^+(f,x)$ exists for $\mu$-a.e. $x$,  is $f$-invariant, and satisfies
 $$
 \lambda^+(f,\mu)=\int_M \lambda^+(f,x) \,d\mu(x).
 $$
\end{theorem}

\begin{exercise}
State and prove versions of Propositions \ref{prop-average-exp} and \ref{prop-average-exp-usc} and Theorem~\ref{thm-pw-exponent} for the \emph{bottom} exponents $\lambda^-(f,\cdot):=-\lambda^+(f^{-1},\cdot)$.
\end{exercise}

\begin{remark}
An illustration of the subtlety of the ergodic theory of subadditive processes (vs. additive ones) is the following phenomenon discovered by P. Walters \cite{Walters-nonuniform}: 
There is a smooth diffeomorphism of a compact manifold which is uniquely ergodic, i.e., for any continuous real function $a\in C^0(M,(0,\infty))$, the limit
 $$
   \lim_{n\to\infty}\frac1n\log|a(f^{n-1}x)a(f^{n-2}x)\dots a(x)|
 $$
exists for \emph{all} $x\in M$ and is constant, equal to the average of $A$ wrt to the unique invariant probability measure, but nevertheless there is some $A\in C^0(M,\GL(\RR^2))$ such that the following limit fails to exist for some $x$:
 $$
   \lim_{n\to\infty}\frac1n\log\|A(f^{n-1}x)A(f^{n-2}x)\dots A(x)\|.
 $$
\end{remark}

\subsection{Multiplicative ergodic theorem}

The Lyapunov exponents can be given a more precice geometric interpretation as  rates of growth of vectors under iteration of a vector bundle morphism. This is provided by the \emph{multiplicative ergodic theorem} (see, eg, \cite[p.50]{Krengel-book}). 
For this statement, we define the angle between two subspaces of the same fiber as
 $$
   \angle(A,B):=\min\{\arccos(a\cdot b)\ge0: (a,b)\in A\times B, \|a\|=\|b\|=1\} \in[0,\pi/2].
 $$

\begin{theorem}[Oseledets]\label{thm-Oseledets}
Let $p:V\to X$ be a measured finite-dimensional vector bundle  and  let $F:V\to V$ be a bundle automorphism  over a measure-preserving automorphism $(X,\mu,f)$. Assume the integrability condition:
 $$
 \log\|F_x\|_{x,fx}\in L^1(\mu).
 $$

Then there exist a set $X'\subset X$ of full $\mu$-measure, a measurable function $r:X'\to\mathbb N^*$, and, for each pair $(x,i)$, where $x\in X'$ and $i\in\{1,2,\dots,r(x)\}$, a subspace $V^i_x$ of $V_x$ and a number $\lambda^i_x$, such that:
\begin{itemize}
\item $V^i_x$ and $\lambda^i_x$ depend measurably on $x$;
\item $r(f(x))=r(x)$, $\lambda^i_{f(x)}=\lambda^i_x$, and $V^i_{f(x)}=F_x(V^i_x)$;
\item $V_x=V^1_x\oplus\dots\oplus V^{r(x)}_x$ with $i<j\implies\lambda^i_x>\lambda^j_x$;
\color{black}
   \item $\lim_{|n|\to\infty} \frac1n\log\|F^n_x.v\|=\lambda_x^i$ for all $v\in V_x^i\setminus0$;
   \item  $\lim_{|n|\to\infty} \frac1n\log\angle(V^i_{f^nx},V^j_{f^nx})=0$.
\end{itemize}
\end{theorem}

\subsection*{Some notations}
When necessary, one writes $r(F,x),\lambda^i(F,x)$, or $E^i(F,x)$.
When $\mu$ is ergodic, then one denotes by $r(F,\mu),\lambda^i(F,\mu)$ the $\mu$-ae constant values of $r(F,x),\lambda^i(F,x)$. 

The \new{Lyapunov exponents repeated according to multiplicity} are $L^1_x\ge L^2_x\ge\dots\ge L^d_x$ such that, for each $1\le i\le r(x)$,
 $$
    \#\{j:L^j_x=\lambda^i_x,\} = \dim V^i_x.
 $$
The \new{average Lyapunov exponents} of $\mu\in\Prob(f)$ are:
 $$
   \forall 1\le j\le d\quad L^j(F,\mu) = \int_X L^j(F,x)\, d\mu .
 $$
These objects all depend on the cocycle $F$. But, when $F$ is clearly understood, one may omit it from the notations.

The points in $X$ where the functions above are defined and have the stated properties  are called \new{Oseledets-regular points}. We write $X_\reg^F$ for the set of these points.
It is easy to check that this set is measurable and invariant. By Oseledets theorem, $X_\reg$ has full $\mu$-measure.

The \new{Lyapunov spectrum} at $x\in X_\reg$ is $\{\lambda^i(x):1\le i\le r(x)\}$.

The extremal exponents are  the top and bottom Lyapunov exponents: 
 $$
  \text{for $\mu$-a.e. $x\in M$ } \lambda^+(x)=\lambda^1_x \text{ and }\lambda^-(x)=\lambda^{r(x)}_x.
 $$

\smallbreak

\begin{remark}\label{rem-LE-inv}
The exponents are antisymmetric in the following sense. Given any $\mu\in\Prob(f)$, for $\mu$-a.e. $x\in X$, we have
 \begin{itemize}
  \item[$\circ$] $r(F^{-1},x)=r(F,x)$;
  \item[$\circ$] $\lambda^i(F^{-1},x)=-\lambda^{r+1-i}(F,x)$ ($r:=r(\mu)$).
 \end{itemize}
 In particular, the top exponent of $(F^{-1},\mu)$ is minus the bottom exponent of $(F,\mu)$.
\end{remark}

The following classical fact will allow us to restrict to hyperbolic measures of saddle type (i.e., with both positive and negative exponents a.e.). Moreover, its proof will be a nice exercise.

\begin{lemma}\label{lem-all-negative-exp}
Let $f$ be a $C^1$ diffeomorphism of a closed manifold and let $\mu\in\Proberg(f)$. If all Lyapunov exponents of $\mu$ are negative, then $\mu$ must be a sink.\footnote{More precisely, $\mu$ must be the unique invariant probability measure supported by an attracting hyperbolic periodic orbit.} Similarly, if all its exponents are positive, then it must be a source.
\end{lemma}

The proof of this lemma will require the following important result. The Pliss Lemma allows one to deduce uniform estimates from simple averages by restricting oneself to a subset of positive measure.

\begin{theorem}[Pliss Lemma]\label{thm-Pliss}
Let $a_1,\dots,a_N$ be a finite sequence of real numbers with average $\alpha:=\frac1N(a_1+\dots+a_N)$. Let $A:=\max(a_1,\dots,a_N)$.

We have, for any $\beta<\alpha$, 
 $$
    \#\{1\le k\le N: \forall 0\le\ell\le k\; (a_\ell+a_{\ell+1}+\dots+a_k) \ge \beta\cdot\alpha\} \geq \frac{\alpha-\beta}{A-\beta} \cdot N.
 $$     
\end{theorem}

\begin{remark}
Pliss Lemma implies the lower bound:
 $$
   m:= \#\{1\le k\le N: a_k \ge \beta\cdot\alpha\} \geq \frac{\alpha-\beta}{A-\beta} \cdot N\;,
 $$
which is obvious and optimal: consider the case where $a=(\beta-\eps,\dots,\beta-\eps,A,\dots,A)$ for some arbitrarily small $\eps>0$ and notice that
 $$
    (N-m)(\beta-\eps)+m\cdot A \ge N \alpha\;,
 $$
so $m(A-\beta+\eps)\ge N(\alpha-\beta+\eps)$, which yields the Pliss bound above as $\eps\to0$.
\end{remark}

\begin{proof}[Proof of Pliss Lemma]
Say that an integer interval in $\{1,\dots,N\}$ is bad if $(\# I)^{-1}\sum_{i\in I} a_i<\beta$. Otherwise say that the interval is good.  Say that an integer $1\le k\le N$ is Pliss if for all $0\le\ell\le k$, the interval $\{k-\ell,\dots,k\}$ is good. We have to show that there are at least $N\cdot (\beta-\alpha)/(A-\alpha)$  Pliss integers.

To this end, notice that any non-Pliss integer $1\le k\le N$ is contained in a bad interval $\{\ell(k),\dots,k\}$ for some $1\le\ell(k)\le k$ which we can assume to be maximal. We claim that such bad intervals are disjoint or nested. Otherwise, there would be integers $k,k'$ such that $\ell(k)<\ell(k')\le k<k'$. Note that $\{k+1,\dots,k'\}$ is good by maximality of $\ell(k')$. This implies that $\{\ell(k'),\dots,k\}$ is bad. But this would contradict the maximality of $\ell(k)$.

Thus the non-Pliss integers are contained in a disjoint union $B$ of bad intervals. The average $(\#B)^{-1}\sum_{i\in B}a_i$ is therefore less than $\beta$. Thus
 $$
    N\cdot \alpha = \sum_{i=1}^N a_i \le (\#B)\cdot\beta + (N-\#B)\cdot A
     = (N-(N-\#B))\cdot\beta + (N-\#B)\cdot A\;,
 $$ 
hence
 $$
    N\cdot (\alpha-\beta) \le (N-\#B)\cdot (A-\beta).
 $$
Therefore the number of Pliss integers is at least $N-\#B\ge N(\alpha-\beta)/(A-\beta)$, proving the claim.
\end{proof}

\begin{exercise}\label{exo-all-negative-exp}
Let $f$ be a $C^1$ diffeomorphism of a closed manifold and let $\mu\in\Proberg(f)$ with only negative Lyapunov exponents.
 Show that there is $\gamma>0$ such that for any $\eps>0$, there are $N\ge1$ and $\rho>0$ satisfying $\mu\{x:\Lip(f^N|B(x,\rho)<e^{-\gamma N}\}>1-\eps$. Using Pliss lemma show that there is a set of positive measure of points $x$ such that
  $$
     \lim_{n\to\infty} \diam(f^nB(x,\rho)) = 0.
  $$
Deduce Lemma~\ref{lem-all-negative-exp} from this.
\end{exercise}

\section{Projective dynamics and Lyapunov exponents}
\label{secProjDynLyap}

The Oseledets theorem expresses the Lyapunov exponents as averages of \emph{measurable functions}. Such functions are usually not continuous and, indeed, the Lyapunov exponents do not depend continuously on the measure (see Exercise~\ref{exo-all-negative-exp}). To analyze these discontinuities, one introduces the Grassmanian bundle, whose fibers are the set of vector subspaces of the original vector bundle.
There, the exponents can be expressed as \emph{averages of a continuous function}.

Since we are focusing on surface dynamics, the Grassmanian bundle simplifies to the projectivization of the original bundle. Thus we specialize a bit Definition~\ref{def-projectivization} and simplify the notations:

\begin{definition}\label{def-projective}
Let $(p:V\to X,f,A)$ be a linear cocycle with $\dim V=2$.
The \new{projectivization of the linear cocycle}  is defined by:
 \begin{itemize}
  \item[--] $\hV$, the bundle over $X$ with fibers $\hV_x:=\{\RR.v:v\in V_x\setminus 0\}$;
  \item[--] $\hf:\hV\to\hV$, the map $(x,E)\mapsto(f(x),A_x(E))$;
  \item[--] $\hpi:\hV\to X$, the projection: $(x,E)\mapsto x$;
  \item[--] $\hvf:\hV\to\RR$, the function $(x,E)\mapsto \log|A_x|_E|_{x,f(x)}$, where $|L|_{x,y}$ is the Jacobian of the linear map $L:V_x\to V_y$ with respect to the Euclidean structures $\|\cdot\|_x$ and $\|\cdot\|_y$.
 \end{itemize}
\end{definition}

An (invariant) lift of $\mu\in\Prob(f)$ is $\hmu\in\Prob(\hf)$ such that $\hpi_*(\hmu)=\mu$. 

\medbreak

The measurable splittings defined by Oseledets theorem induce measurable sections, as defined below:

\begin{definition}
Let $(\hpi,\hV,A,\hf)$ be as above. Let $X_\#$ be the set of points in $X_\reg$ with $r(x)=\dim V$. Then the Oseledets sections are the measurable maps
 $$
    \Gamma^i:X_\#\to\hV, \quad x\longmapsto E_x^i.
 $$
When $\dim V=2$, we set $X_\hyp:=\{x\in X_\reg:\lambda^1(x)>0>\lambda^2(x)\}$ and we define the \new{unstable section} to be $\Gamma^+:=\Gamma^1|X_\hyp$
and the \new{stable section} to be $\Gamma^-:=\Gamma^2|X_\hyp$.
\end{definition}

\begin{proposition}\label{prop-lift-projective}
Let $(\hpi,\hV,A,\hf)$ be as above. Assume that the cocycle is invertible.
Let $\mu\in\Prob(f)$. Then the following are true:

\smallbreak\noindent
(i) $\mu$ always has some \new{lift}, i.e., there exists $\hmu\in\Prob(\hf)$ with $\hpi_*(\hmu)=\mu$.

\smallbreak\noindent
(ii) If $\mu$ is ergodic, then a.e. ergodic component of a lift is an ergodic lift. In particular, $\mu$ admits an ergodic lift.

\smallbreak\noindent
(iii) If $\hmu$ is any ergodic lift of $\mu$ (so $\mu$ must be ergodic), then
 $$
    \hmu(\hvf) = \lambda(x,v) \text{  for $\hmu$-a.e. }\hv=(x,v)\in V.
 $$

\smallbreak\noindent
(iv) If $\mu$ is ergodic, then
 \begin{equation}\label{eq-Exponents-Erg-Lifts}
    \{\hmu(\hvf):\hmu\in\Proberg(\hf)\text{ s.t. }\hpi_*(\hmu)=\mu\}=
        \{\lambda^i(\mu):1\le i\le r(\mu)\}.
  \end{equation}

\smallbreak\noindent
(v) In general,  given the ergodic decomposition $\mu=\int_{\Proberg(f)} \nu\, dP$,  there is a disintegration $\hP = \int \hP_\nu\, d\overline{P}$ such that
 $$
   \hnu = \int_{\Proberg(f)} \left(\int_{\Proberg(\hf)\cap \hpi_*^{-1}(\nu)} \hnu\, d\hP_\nu\right)\, dP(\nu).
 $$
(vi) For any lift $\hmu$ of $\mu$, $\hmu\left(\hM_\#\setminus\bigcup_{i=1}^d \hGamma^i(X_\#)\right)=0$.

\end{proposition}

The above is true for a vector bundle of any finite dimension. 
For simplicity, we prove item (iv) under the assumption that $r(\mu)\le 2$. For the general case, see \cite[Thm 6.1]{Viana-LE-book}.

\begin{proof}
Let $\mu\in\Prob(f)$.

We first prove (i). Using the definition of a vector bundle, one can build a measurable section $\Gamma:X\to\hV$ (not necessarily continuous or invariant). The image measure $\hmu_0:=\Gamma_*(\mu)$ on $\hV$ satisfies: $\hpi_*(\hmu_0)=\mu$. As in the classical Krylov-Bogoliubov argument, any accumulation point of the sequence of averages 
 $$
    \frac1n\sum_{k=0}^{n-1} \hf^k_*(\hmu_0)
 $$
produces a lift $\hmu$ of $\mu$.
Indeed, since $\hf_*,\hpi_*:\Prob(\hV)\to\Prob(X)$ are continuous and $\Prob(\hf)$ is a closed subset of the compact $\Prob(\hV)$, the above measures admit at least one accumulation point $\hmu$ and this point belongs to $\Prob(\hf)$. This proves item (i).
\smallbreak

We turn to item (ii): by applying (i) to $\mu\in\Proberg(f)$, we get some not necessarily ergodic lift $\hmu\in\Prob(f)$. Write its ergodic decomposition $\hmu=\int_{\Proberg(\hf)} \hnu\, d\hP$. Taking the image by $\hpi_*$ which is continuous, we get
 $$
    \mu=\int_{\Proberg(\hf)} \hpi_*(\hnu)\, d\hP.
 $$
Since $\mu$ is ergodic, it is an extreme point of $\Prob(f)$. Hence $\hpi_*(\hnu)=\mu$ a.e. Thus, a.e. ergodic component $\hnu$ is an ergodic lift of $\mu$, proving item (ii).

\smallbreak

To prove item (iii), we rely on the following identity: 
 \begin{equation}\label{eq-basic-hvf}
    \forall (x,v)\in V\setminus 0\; \forall n\ge1 \quad \sum_{k=0}^{n-1} \hvf(\hf^k(x,v)) = \log\|F^n.v\|_{f^nx}-\log\|v\|_x.
  \end{equation}
This and the Birkhoff theorem imply that for $\hmu$-a.e. $\hv=(x,\RR.v)\in\hV$
 $$
    \hmu(\hvf) = \lim_{n\to\infty} \frac1n\sum_{k=0}^{n-1} \hvf(\hf^k(\hv))=\lim_{n\to\infty} \frac1n\log\|F^n_x.v\|_{f^nx} = \lambda(x,v).
 $$
Item (iii) is proved.

\smallbreak

We turn to item (iv). First, if  $\hmu\in\Proberg(\hf)$ is some ergodic lift of $\mu$, then one has item (iii), i.e., $\hmu(\hvf)=\lambda(x,v)$ for $\hmu$-a.e. $\hv$, that is, $\hmu(\hvf)$ belongs to the spectrum of $(F,\mu)$. 

Conversely, let us see that $\lambda^1(\mu)$ is an average $\hmu(\hvf)$ for some ergodic lift $\hmu$ of $\mu$.
Pick some $\mu$-generic $x\in X_\reg$ and fix some $v\in E^1_x\setminus 0$. Using the weak topology on $\Prob(\hV)$, take any accumulation point
 $$
     \hmu_\infty = \lim_j \frac1{n_j} \sum_{k=0}^{n_j-1} F^k_*\delta_{(x,v)}.
 $$
Since $\hvf$ is continuous, it holds that
 $$
   \hmu_\infty(\hvf)=\lim_j \frac1{n_j} \sum_{k=0}^{n_j-1} \hvf(\hv^k\hv) =\lambda(x,v)=\lambda^1(\mu).
 $$ 
Since $\hmu_\infty$ might fail to be ergodic, we consider its ergodic decomposition  $\hmu_\infty = \int \hnu\, d\hP$. We must have
 $$
     \lambda^1(\mu) = \int \hnu(\hvf)\, d\hP\;.
 $$
Therefore the inequality $\hnu(\hvf)\le\lambda^1(\mu)$ must be an equality for $\hP$-a.e. $\hnu$: a.e. ergodic component $\hmu$  of $\hmu_\infty$ is an ergodic lift $\hmu$ with $\hmu(\hvf)=\lambda^1(\mu)$.

The previous argument applied to $F^{-1}$ gives an ergodic lift $\hmu$ realizing the bottom exponent: $\lambda^{r(\mu)}(\mu)=\hmu(\hvf)$.

For $r(\mu)\le 2$, this is enough to prove item (iv). 

\smallbreak

To prove item (v), let $\hmu=\int \hnu\, d\hP$ be the ergodic decomposition of $\hmu$. Since $\hpi_*:\Prob(\hf)\to\Prob(f)$ is continuous and affine, we have
 $$
    \mu= \hpi_*(\hmu) = \int_{\Proberg(\hf)} \hpi_*(\hnu)\, d\hP
     = \int_{\Proberg(f)} \nu \, d\hpi_*(\hP)\;.
 $$
Since the ergodic decomposition $\mu=\int \nu\, dP$ is unique, we have $\hpi_*(\hP)=P$.

Consider the Rokhlin disintegration of $\hP$ corresponding to the map $\hpi:\Proberg(\hf)\to\Proberg(f)$:
 $$
   \hP=\int_{\Proberg(f)} \hP_\nu\, d\hpi_*(\hP)
       \text{ with } \hP_\nu(\Proberg(\hf)\cap\hpi_*^{-1}(\nu))=1 \text{ $\hP$-a.e.}
 $$
Hence, the ergodic decomposition of $\hmu$ can be rewritten as follows:
 $$
   \hmu = \int_{\Proberg(f)} \left( \int_{\Proberg(\hf)\cap\hpi_*^{-1}(\nu)} \,d\hP_\nu(\hnu)\right) \, dP,
 $$
which proves item (v).

\smallbreak

Item (vi) is an immediate consequence of (v).
\end{proof}

\section{Reading the simple Lyapunov spectrum}

We consider a linear invertible cocycle $A$ of a vector bundle $p:V\to X$.
Recall that $X_\reg$ is the set of Oseledets regular points. The following is the set of points with \new{simple Lyapunov spectrum}:
 $$
   X_\#:=\{x\in X_\reg:r(x)=\dim V\}\;.
 $$
We also consider 
 $$
    X_=:=\{x\in X_\reg:r(x)=1\}=\{x\in X_\reg:\lambda^+(F,x)=\lambda^-(F,x)\}.
 $$
On $X_=$, the Oseledets theorem reduces to the following equality:
 $$
    \lim_{|n|\to\infty} \frac1n\log\|F^n_x.v\|=\lambda^+(F,x)=\lambda^-(F,x)\text{ for all }v\in V_x\setminus0.
 $$

Over $X_\#$, the situation is more interesting. To describe it, it is convenient to use the projective dynamics as follows. We give the general statement though we are only going to consider the case where $\dim V=2$ and $X_\reg=X_=\sqcup X_\#$ is a partition.

Recall that $\hV^i$ ($1\le i\le\dim V$) is the measurable sub-bundle defined by Oseledets theorem.

\begin{proposition}\label{prop-lift-simple}
Let $\hV^i_\#$ be the restriction of the bundle $\hV^i$ to $X_\#$.
Let $\mu\in\Proberg(f)$ with $\mu(X_\#)=1$. Then the following are true:
 \begin{enumerate}
  \item $\mu$ admits exactly $d=\dim V$ ergodic lifts $\hmu^1,\dots,\hmu^d$
 $$
   \hmu^i := \int_{X_\#} \delta_{(x,V^i_x)} \, d\mu(x) \quad (i=1,\dots,d).
 $$
 \item For each $i=1,\dots,d$, $\hmu^i(\hvf)=\lambda^i(A,\mu)$.
 \item The ergodic components of any invariant lift are, a.s., ergodic lifts.
\end{enumerate}  
\end{proposition}

\begin{proof}
We first prove (1), starting with showing that each $\hmu^i$, $1\le i\le r(\mu)=d$, is an ergodic lift.

We recall Oseledets theorem. Since $\mu(X_\#)=1$, for $\mu$-a.e. $x\in X$, $E^i_x$ is one-dimensional and satisfies $Df_x(E^i_x)=E^i_{fx}$. Therefore, there is a unique measurable section $\Gamma^i:X_\#\to\hV$, $x\mapsto E^i_x$. In particular, $\Gamma^i\circ f =F\circ\Gamma^i$, so $\hmu^i:=\Gamma^i_*(\mu)\in\Proberg(\hf)$ (defined in the above statement) is an ergodic lift. Obviously, $\hmu^i(\hvf)=\int_{X_\#} \hvf(E^i_x)\, d\mu(x)=\lambda^i(\mu)$.

Conversely, let $\hmu$ be an arbitrary ergodic lift of $\mu$. We are going to show that it must be $\hmu^i$ for some $1\le i\le r(\mu)$, which in particular implies  that $\hmu(\hvf)=\lambda^i(\mu)$.

Let $(x,\RR.v)$ be some $\hmu$-typical point. We can assume that $\|v\|_x=1$. We  write $v = \sum_{i=1}^d \alpha_i v^i_x$ where $\alpha_i\in\RR$ and $v^i_x$ a unit vector of $E^i_x$ (unique up to sign). Let $I$ be the smallest integer such that $\alpha^I\ne 0$. By ergodicity, this integer is a.e.-constant. 

If $I=d$, then $\hmu=\hmu^d$ and the claim is proved. Otherwise, note that $\lambda^I>\lambda^{I+1}>\dots>\lambda^d$ for all $i>I$. Using the Euclidean distance on the unit sphere in $V_x$, we obtain
 $$
    d\left(\frac{Df^n.v}{\|Df^n.v\|}, E^I_{f^nx}\right)
    =d\left(\hf^n(v), \hf^n(E^I_x)\right)
    \leq C(x) e^{n(\lambda^I-\lambda^{I+1})} \to 0.
 $$
By invariance of $\hmu$, this implies the distance must be zero for each $n$. Thus, $\hv=E^I_x$, i.e., $\hmu=\hmu^I$ and  (1) is proved.
\medbreak

For each $1\le i\le d$, the Oseledets theorem gives $A_x(V^i_x)=V^i_{f(x)}$, which implies the invariance: $\hf_*(\hmu^i)=\hmu^i$. Since each space $E^i_x$ is one-dimensional, we have $\log|Df^n_x|=\sum_{k=0}^{n-1}\log|Df|_{E^i_{f^kx}}|$. Since $\mu$ is invariant and $\log\|Df|_E\|$ is bounded, the following holds:
 $$\begin{aligned}
   \hmu^i(\hvf) &= \int_{X_\#} \log | Df|_{E^i_x}|\, d\mu 
   = \int_{X_\#} \frac1n\log | Df^n|_{E^i_x}|\, d\mu\\
   &= \int_{X_\#} \lim_{n\to\infty} \frac1n\log | Df^n|_{E^i_x}|\, d\mu
   = \int_{X_\#} L^i_x\, d\mu = \lambda^i(A,\mu).
 \end{aligned}$$
This proves (2).
\medbreak

Let $\hmu=\int \hmu_\xi\, d\xi$ be the ergodic decomposition of a lift $\hmu\in\Prob(\hf)$ of $\mu\in\Proberg(f)$. Since $\hpi_*$ is a continuous and linear map, we get
 $$
   \hmu= \hpi_*(\hmu) = \int \hpi_*(\hmu_\xi)\, d\xi\;.
 $$
Hence by ergodicity of $\hmu$, for a.e. $\xi$, $\hpi_*(\mu_\xi)=\mu$, the ergodic components of the lift are ergodic lifts of $\hmu$. This proves (3).
\end{proof}

\begin{remark}
Without the assumption of simple Lyapunov spectrum (i.e., $\mu(X_\#)=1$), the Lyapunov exponents of $\mu\in\Proberg(f)$ are still  the average dilations wrt all the ergodic lifts
 $$
  \left\{ \hmu(\hvf): \hmu\in\Proberg(\hf) \text{ s.t. } \pi_*\hmu=\mu \right\} \; .
 $$
\end{remark}

\section{Continuity properties of the top Lyapunov exponents}

\subsection{The setting}
We are going to apply the preceding results to the continuity of the top Lyapunov exponent. 
Given $\nu_1,\nu_2,\dots\in\Proberg(f)$ weakly converging to some $\mu\in\Prob(f)$, we consider the average top exponents
 \begin{equation}\label{eq-lsc-defect}
      \lambda^+(f,\mu) - \limsup_{k\to\infty} \lambda^+(f,\nu_k) \ge 0
  \end{equation}
(the inequality is Proposition~\ref{prop-average-exp-usc}). By ergodicity, the average top Lyapunov exponent of $\nu_k$ coincides with its top exponent a.e.,
 $$
    \lambda^+(f,\nu_k)=\lambda^1(f,x) \text{ for $\nu_k$-a.e. }x\in M,
 $$
while, for not necessarily ergodic limit measures, we have
 $$
    \lambda^+(f,\mu) = \int \lambda^+(f,x)\, d\mu.
 $$
In this section, we express \eqref{eq-lsc-defect} in terms of suitable projective lifts of $\nu_k,\mu$, restricting, when necessary, to surface diffeomorphisms.

\subsection{Top Lyapunov exponents of $\nu_k$'s and their limit}
In our applications, each ergodic $\nu_k$ will have positive entropy for the surface diffeomorphism $f$, hence have, almost everywhere, exactly one positive and one negative exponent. In particular, there is no harm assuming that
 \begin{equation}\label{eq-nuk-simple-LE}
  \text{each $\nu_k$, $k\ge1$, has simple Lyapunov spectrum almost everywhere}.
\end{equation}
Thus, Proposition~\ref{prop-lift-simple} gives the existence of the lift $\hnu_k^1$ of $\nu_k$ to $\hGamma^1$. Moreover,
 $$
   \forall k\ge1\quad  \lambda^+(f,\nu_k) = \int \vf \, d\hnu_k^1.
 $$
By compactness, perhaps after passing to a subsequence, the measures $\hnu_k^1$ converge to some $\hmu\in\Prob(\hf)$. By definition of the weak convergence and the continuity of $\vf:\hM\to\RR$, we have
 $$
    \lim_{k\to\infty} \lambda^+(f,\nu_k) = \lim_{k\to\infty} \hnu_k(\vf)
     = \hmu(\vf).
 $$
Obviously, $\hmu$ is a lift of $\mu$. However, $\hmu$ does not need to coincide with the lift $\hmu^1$. Indeed, such a lift may fail to exist as $\dim(E^1_x)$ may be greater than $1$ on a set of positive $\mu$-measure.

\subsection{Decomposition into hyperbolic and neutral parts}
Recall that $M_\#:=\{x\in M_\reg:r(f,x)=\dim M\}$ is the set of (Oseledets regular) points with simple Lyapunov spectrum. Since possibly $\mu(M_\#)<1$, we cannot apply Proposition~\ref{prop-lift-projective} directly. We consider the decomposition (distinct from the neutral decomposition to be discussed later):
 $$
    \mu = t\cdot\mu_\# + (1-t)\cdot\mu_*, \text{ where }
      \mu:=\mu|_{M_\#},\;  \mu_* := \mu|_{M_\reg\setminus M_\#},
      \text{ and } 0\le t\le 1.
 $$
where we use the following notation for any invariant measurable subset $X$, we let
 $$
   \mu|_X:=\frac{\mu(\cdot\cap X)}{\mu(X)} \text{ if }\mu(X)>0 \text{ and } 
   \mu|_X := \mu \text{ otherwise.}
 $$
We are going to control separately the two components $\hmu_\#,\hmu_*$.

\subsection{Control of the neutral part in dimension $2$} $ $
When $\dim M=2$, we control $\mu_*$ by using the fact that 
 $M\setminus M_\#$ coincides with the set $M_=$ of points for which all Lyapunov exponents are equal. Thus the following lemma applies and shows in particular that, on surfaces,  the exponents and the entropy of $\mu_*$ must vanish, and, consequently,
\begin{equation}\label{eq-entropy-limit}
 h(f,\mu) = t\cdot h(f,\mu_\#) \text{ and } \lambda^+(f,\mu) = t\cdot\lambda^+(f,\mu_{\#}).
\end{equation}

\begin{lemma}\label{lem-equal-zero}
Let $f$ be a diffeomorphism of a closed manifold $M$ and let $\nu_k\in\Proberg(f)$ $(k\ge1)$ weakly converge to some $\mu$. Assume that all $\nu_k$ are non atomic.

Then $L^1(f,x)=\dots=L^d(f,x)=0$ for $\mu$-a.e. $x\in M_=$. In particular, $h(f,\mu|M_=)=0$. 
\end{lemma}

\begin{proof}
Up to replacing $f$ by $f^{-1}$, we can assume by contradiction that $L^1(f,x)=\dots=L^d(f,x)<0$ for a set of points $x\in M_=$ with positive $\mu$-measure. This implies that the ergodic decomposition of $\mu$ contains (with positive measure) elements which have only negative exponents, hence are sinks by Lemma~\ref{lem-all-negative-exp}. Since sinks have disjoint open basins, there can be at most countably many of them.  Thus, there is a sink with a basin $B$ such that $\mu(B)>0$. 

By weak convergence, $\nu_k(B)>0$ for all large $k$. Since $\nu_k$ is ergodic, this implies that $\nu_k$-a.e. orbit visits $B$ and therefore converges to the sink. The pointwise ergodic theorem shows that $\nu_k$ must be itself a sink, a contradiction to the non-atomic assumption.
\end{proof}

\subsection{Control of the simple spectrum part}
It is convenient to use the following form of the ergodic decomposition.
For $\mu$-a.e. $x\in M$ we denote by $\mu_x$ the \new{empirical measure} $\lim_{n\to\infty}\frac1n\sum_{k=0}^{n-1}\delta_{f^kx}$. It is known that, for $\mu$-a.e. $x\in M$, the following are true:
 \begin{enumerate}
  \item $\mu_x$ is ergodic;
  \item $\lambda^+(f,x)=\lambda^+(f,\mu_x)$ and $L^j(f,x)=L^j(f,\mu_x)$ for all $1\le j\le \dim M$;
  \item the ergodic decomposition of $\mu$ can be written as
$\mu  = \int_{\bar M} \mu_{\bar x} \, d\bar\mu$ where $(\bar M,\bar\mu)$ is the quotient of $(M,\mu)$ by the map $x\mapsto\mu_x$ and $\mu_{\bar x}=\mu_x$.
\end{enumerate}

We can now apply Proposition~\ref{prop-lift-projective} to $\hmu_\#$ as a lift of $\mu_\#$. It yields a measurable function $a:M_\#\to[0,1]$ such that $\hmu|_{\hpi^{-1}(M_\#)} = \int_{M_\#} \left[(1-a(x))\hmu_x^1 + a(x)\hmu_x^2\right] \, d\mu_\#(x)$.

\subsection{Control of the whole lift, the top exponent, and the entropy}
We can now put together the pieces in the case of a surface diffeomorphism (hence $M_\reg\setminus M_\#=M_=$). We have
 $$
    \hmu = t\cdot \int_{M_\#} \left[(1-a(x))\hmu_x^1 + a(x)\hmu_x^2\right] \, d\mu_\#(x) + (1-t)\cdot \hmu|_{\hpi^{-1}(M_=)}.
 $$

Thus, we obtain:

\begin{proposition}\label{prop-disc-LE-loss-mass}
Let $f$ be a diffeomorphism of a closed surface.
Let $\mu_1,\mu_2,\dots\in\Proberg(f)$ be atomless and converge weakly to some $\mu\in\Prob(f)$ with ergodic decomposition $\mu=\int_{\bar M} \mu_{\bar x}\, d\bar x$. Then
 $$
    \hmu:=\lim_k \hmu^1 = \int_{M_\#} \left((1-\alpha(\bar x))\hmu_x^1+\alpha(\bar x)\hmu_x^2)\right)\, dx
    + \int_{M_=} \hmu_x^0\, d\mu(x)\;,
 $$
where $\hmu_x^1,\hmu_x^2$ are the two ergodic lifts of the hyperbolic measures $\mu_x$ for $x\in M_\#$, and $\hmu_x^0$ are the lifts of the measures $\mu_x$ for $x\in M_=$ {\rm(}i.e., for which both exponents are zero{\rm)}.
\end{proposition}

We note the following consequences for the exponents and entropies of a sequence of measures and its limit:
 
\begin{corollary}
In the setting of the above proposition, let $\nu_k\to\mu$ in the weak topology of $\Prob(f)$. Perhaps for a subsequence, we have the following equality:
 $$
    \lim_k \lambda^1(f,\mu_k) = \int_{M_\#} \left((1-\alpha(x))\lambda^1(f,\mu_x)+\alpha(x)\lambda^2(f,\mu_x)\right)\, d\mu(x)\;.
  $$
In particular, the defect of lower continuity equals
 $$
   \lambda^1(f,\mu)-\lim_k \lambda^1(f,\mu) = \int_{M_\#} \alpha(x)\left(\lambda^1(f,\mu_x)-\lambda^2(f,\mu_x)\right)\, d\mu(x)\;.
 $$
\end{corollary}

\chapter{Neutral decomposition}\label{chapter-neutral}

Recall our general strategy: given a sequence of ergodic measures $\nu_k\in\Proberg(f)$ weakly converging to some $\mu\in\Prob(f)$, 
we want to use discontinuities of the top exponent of measures to derive a non-trivial bound on their entropies.
This chapter accomplishes the linear part of this strategy, deducing from any discontinuity of the exponent
 \begin{equation}\label{eqDiscLyap}
    \lambda^+(\mu)-\limsup_k \lambda^+(\nu_k) \qquad (\text{nonnegative by upper semicontinuity}),
 \end{equation}
the existence of a \emph{neutral part} of the measure $\mu$ whose non-expanding properties will be used in later chapters to bound the entropy.

More precisely, we consider the unstable lifts $\hnu_k^+\in\Proberg(\hf)$ converging to some $\hmu\in\Prob(\hf)$. As we have observed in the previous chapter, since $\hvf$ is continuous on $\hM$, we get
 $$
   \lim_k \lambda^+(\nu_k) = \lim_k \hnu_k^+(\hvf) = \hmu(\hvf).
 $$
Hence, any discontinuity of the exponent means that $\hmu$ differs from the unstable lift $\tmu^+$ of $\mu$. In Theorem~\ref{thm-neutral-decomposition-projective}, this will lead to a {\bf neutral decomposition}:
 $$
    \hmu = (1-\beta)\hmu_0 + \beta \hmu_1 \quad (0\le\beta\le 1,\; \hmu_0,\hmu_1\in\Prob(\hf)),
 $$
where $\hmu_0$ has some non-expanding properties at the level of the differential $Df$.

The neutral part $(1-\beta)\hmu_0$ is obtained by taking a limit of the empirical measure defined by $\nu_k$-typical orbits restricted to certain segments. Those are the {\bf neutral orbit segments} during which the unstable direction does not expand going forward.

We will later show that the lack of expansion within neutral orbit segments prevents them from contributing to the entropy, which yields a bound of the following form:
 $$
    \limsup_k h(\nu_k) \le \beta\cdot h(\mu_1) \qquad (\mu_1:=\hpi_*(\hmu_1)).
 $$

We will define and build an abstract neutral decomposition in the setting of an additive cocycle:
 \begin{itemize}
  \item[--] $T:X\to X$ is a homeomorphism of a compact metric space;
  \item[--] $\vf:X\to\RR$ is a continuous function;
  \item[--] $\nu_k\in\Proberg(T)$ weakly converge to $\mu\in\Prob(T)$.
\end{itemize}
We will apply this construction to the projective dynamics of a smooth surface diffeomorphisms by taking:
 \begin{itemize}
  \item[--] $T:X\to X$, the projective extension of the surface diffeomorphism;
  \item[--] $\vf$, the logarithm of its dilation;
  \item[--] $\nu_k$,  the unstable lift of a hyperbolic measure in $\Proberg(f)$.
 \end{itemize}
Our study of Lyapunov exponents in the previous chapter will allow us to relate the term $(1-\beta)\hmu_0$ of the neutral decomposition to the discontinuity \eqref{eqDiscLyap}.

\section{Neutral blocks}
The first step is to define the \emph{neutral segments} in the orbit of a point $x\in X$. They correspond to a controlled, weak expansion starting at some time $a$ and going forward \emph{any number of iterates} until some time $b$. As we shall see in future chapters, this will be sufficient to prevent significant entropy creation for the nonlinear dynamics in our application to surface diffeomorphisms. 

The structure of the neutral orbit segments will allow the definition of a neutral set which we will use in the next section to construct a neutral decomposition of the measure.

\begin{definition}\label{def-abstract-neutral-block}
Let $\alpha>0$ and $L\ge1$.
Given $x\in X$ and integers $a,b\in\ZZ$ with $a<b$,  a finite \new{integer interval}
 $$
   \ii{a,b}:=\{a,a+1,\dots,b-1\}
 $$
is called an $\alpha$-\NEW{neutral block} for $x$ if 
  $$
    \forall a< i\le b\quad S_a^i\vf(x):=\vf(T^a(x))+\vf(T^{a+1}(x))+\dots+\vf(T^{i-1}(x))\le \alpha\cdot (i-a).
  $$
The integer interval $\ii{a,b}$  is an \emph{$(\alpha,L)$-neutral block} for $x$ if  it is $\alpha$-neutral with length $b-a\ge L$.

We also say that $(T^j(x))_{j\in\ii{a,b}}$ is an $(\alpha,L)$-\new{neutral orbit segment}.
\end{definition}

Let us note that all neutral blocks are finite by definition.\footnote{It is not clear what it would mean for an infinite interval such as $\ii{-\infty,b}$ to be $\alpha$-neutral.} 
We also note that  given $x\in X$ which is not periodic, one can identify the block $\ii{a,b}$ with the orbit segment $(T^j(x))_{a\le j<b}$.

\medbreak

The following terminology will be useful. We say that two integer intervals $\ii{a,b}$ and $\ii{c,d}$:
 \begin{itemize}
  \item[--] are \emph{nested} if one is included in the other;
  \item[--] \emph{overlap} if they intersect but are not nested.
 \end{itemize}

\begin{lemma}\label{lem-neutral-nested}
Given $x\in X$, $\alpha>0$,  the union of two overlapping $\alpha$-neutral blocks  is itself $\alpha$-neutral.
\end{lemma}

\begin{proof}
Let $\ii{a,b}$ and $\ii{c,d}$ be two $\alpha$-neutral blocks with $a\le c <b\le d$. We need to show that their union  $\ii{a,d}$ is $\alpha$-neutral, i.e., that for any $a<j\le d$, one has $S_a^j\vf(x)\le \alpha\cdot(j-a)$. 

If $j\le b$, the assumption that $\ii{a,b}$ is $\alpha$-neutral is enough to conclude. If $b<j\le d$, we also use that $\ii{c,d}$ is $\alpha$-neutral:
 $$
   S_a^j\vf(x) = S_a^{b}\vf(x)+S_b^j\vf(x)
    \leq \alpha\cdot (b-a)+\alpha (j-b) = \alpha\cdot(j-a).
 $$
\end{proof}

For $x\in X$, let $N_x(\alpha,L)\subset\ZZ$ be the union of all $(\alpha,L)$-neutral blocks for $x$.

\begin{proposition}
Given  $\alpha>0$ and $L\ge 1$, let $x\in X$ satisfy
 $$
   \liminf_{n\to\infty} \frac1n S_{-n}^0\vf(x) > \alpha \text{ and }
   \liminf_{n\to\infty} \frac1n S_0^n\vf(x) > \alpha.
 $$ 
Then $N_x(\alpha,L)$ is a disjoint union of $(\alpha,L)$-neutral blocks that are maximal for the inclusion. Moreover, it is equivariant, i.e.,  $\forall x\in X$ $N_{T(x)}(\alpha,L)=N_x(\alpha,L)-1$. 
\end{proposition}

\begin{proof}
Let $a\in\ZZ$. By assumption, there is $N$ such that for any $c\le a-N$, $S^a_c\vf(x) > \alpha (c-a)$. Thus any $\alpha$-neutral block $\ii{c,d}$, containing $a$, satisfies $c>a-N$. Likewise, perhaps after increasing $N$, $d<a+N$. Thus, all  $(\alpha,L)$-neutral blocks in the orbit of $x$ containing $a$ are subsets of $[a-N,a+N]$. Thus, any $(\alpha,L)$-neutral block is contained in a maximal one by Lemma~\ref{lem-neutral-nested}.

The equivariance follows immediately from the definition.
\end{proof}

\begin{definition}
The set of $(\alpha,L)$-\new{neutral points} is
 $$
   \Neutral(\alpha,L):=\{x\in X: 0\in N_x(\alpha,L)\} \subset X.
 $$
\end{definition}

These sets are almost invariant when $L$ is large.

\begin{lemma}\label{lem-almost-inv}
The set $\Neutral(\alpha,L)$ is measurable. Moreover, for any $x\in X$,
\begin{equation}\label{eq-Neutral-near-inv}
   \limsup_{n\to\infty} \frac1n\#\{0\le k<n: x\in (T^{-k}\Neutral(\alpha,L))\diffsym(T^{-k-1}\Neutral(\alpha,L))\} \leq 2/L.
\end{equation}
\end{lemma}

\begin{proof}
The set $\Neutral(\alpha,L)$ is measurable since it is equal to
 $$
 \bigcup_{\ell\ge L} \{x\in X: \ii{-\ell,1} \text{ is $\alpha-$neutral for }x\}.
 $$
Using the equivariance, we get, for any $k\in\ZZ$,
 $$
   \forall i\in\ZZ\; x\in T^{k+i}\Neutral(\alpha,L)\iff i\in N_x(\alpha,L).
 $$
Therefore, 
 $$
    x\in (T^{-k}\Neutral(\alpha,L))\diffsym(T^{-k-1}(\alpha,L))\iff
    k\in \underbrace{N_x(\alpha,L)\diffsym(N_x(\alpha,L)-1)}_{\partial N_x(\alpha,L)}.
 $$
The elements of $\partial N_x(\alpha,L)$ are endpoints of maximal $(\alpha,L)$-neutral blocks. Hence $\partial N_x(\alpha,L)\cap\ii{0,n}$ has at most $2+2n/L$ elements. 
This proves \eqref{eq-Neutral-near-inv}.
\end{proof}

\section{Neutral decomposition}

We define and discuss the neutral decomposition of a sequence of measures.
By the Birkhoff ergodic theorem, one can reconstruct an ergodic measure from its typical orbits. More precisely, if $\nu\in\Proberg(T)$, then  the Birkhoff ergodic theorem gives, for $\nu$-a.e. $x\in X$,
 $$
    \lim_{n\to\infty} \frac1n\sum_{k=0}^{n-1} \delta_{T^kx}  = \mu\;.
 $$
We will obtain the neutral decomposition by splitting the above sums into neutral blocks and the remainder. We will then let $\alpha\to0$ (to capture all expansion) and $L\to\infty$ (to get invariance).

\medbreak

Recall that a sequence $\nu_k\in\Proberg(T)\to\mu\in\Prob(T)$ has been fixed together with a continuous function $\vf:X\to\RR$ satisfying
 \begin{equation}\label{eq-positive-avg}
    \liminf_k \int_X \vf\, d\nu_k>0.
 \end{equation}
We will consider $\alpha>0$ arbitrarily small, and in particular smaller than the above limit.
\medbreak

We consider \new{subprobability Borel measures} on $X$, i.e., positive Borel measures with total mass at most $1$. The set of such measures is compact for the weak (star) topology which is metrizable (for the same reasons as the set of probability measures).

\paragraph{Some notations.} For a subprobability measure $\nu$ and an integrable, nonnegative function $g\in L^1(\nu)$, we write $g\cdot \nu$ for the subprobability measure\footnote{Often the measure  $g\cdot\nu$ is denoted by $g\,d\nu$.} defined by $(g\cdot \nu)(E)=\int_E g\, d\nu$ for all measurable set $E\subset X$.

When $g$ is a measurable function, either nonnegative or integrable, we also write $\nu(g)$ for $\int g\, d\nu$.

\begin{definition}\label{def-neutral-dec}
A $\vf$-\NEW{neutral decomposition} of the sequence $(\nu_k)$ is:
$\mu = m_0 + m_1$ where $m_0,m_1$ are both $T$-invariant, Borel subprobability measures satisfying, for some sequence $k_i\to\infty$,
 $$
    \mu^0_{\alpha,L,k_i} := 1_{\Neutral(\alpha,L)}\cdot \nu_{k_i}\hskip-0.7cm\stackrel{\tiny\begin{array}{c}(\alpha,L)\to(0,\infty)\\i\to\infty\end{array}}\longrightarrow\hskip-0.7cm m_0
    \text{ and }
    \mu^1_{\alpha,L,k_i} := (1-1_{\Neutral(\alpha,L)})\cdot \nu_{k_i} \hskip-0.7cm\stackrel{\tiny\begin{array}{c}(\alpha,L)\to(0,\infty)\\i\to\infty\end{array}}\longrightarrow\hskip-0.7cm  m_1.
 $$
More explicitly, for any pair of neighborhoods $\cV_0,\cV_1$ of $m_0,m_1$, there are $(\alpha_1,L_1)$ such that for all $0<\alpha<\alpha_1$ and all $L\ge L_1$, for all $i$ larger than some $i_1(\cV_0,\cV_1,\alpha,L)$,
  $$
     \mu^0_{\alpha,L,k_i} := 1_{\Neutral(\alpha,L)}\cdot \nu_{k_i} \in \cV_0
     \text{ and }
     \mu^1_{\alpha,L,k_i} := (1-1_{\Neutral(\alpha,L)})\cdot \nu_{k_i} \in \cV_1.
  $$
\end{definition}

We call $m_1$ the \emph{positive component} and we let $\beta:=m_1(X)$ and $\mu_1:=\beta^{-1}m_1$ (if $\beta=0$, we set $m_1:=\mu$). We call $m_0$ the \emph{neutral component} and we let $\mu_0:=(1-\beta)^{-1}\cdot m_1$ (if $\beta=1$, we set $\mu_0:=\mu$). In any case, we write:
 \begin{equation}\label{eq-neutral-decomp}
   \mu = (1-\beta) \mu_0 + \beta\cdot \mu_1
   \qquad
   (0\le\beta\le 1,\; \mu_0,\mu_1\in\Prob(T)).
  \end{equation}

Under the positivity assumption \eqref{eq-positive-avg}, the neutral decomposition always exists, as stated in the theorem below:

\begin{theorem}\label{thm-neutral-dec}
Given a sequence $\nu_k\in\Proberg(T)\to\mu$ and a continuous function $\vf:X\to\RR$ satisfying \eqref{eq-positive-avg}, there exists a $\vf$-neutral decomposition $\mu=m_0+m_1$ such that
 \begin{itemize}
 \item[(A)] $m_0(\varphi)=0$;
 \item[(B)] for $m_1$-a.e. $x\in X$, $\liminf_n \frac1nS_n\varphi(x)>0$.
\end{itemize}
\end{theorem}

On the one hand, $\beta>0$ ensures that $m_1\ne0$ so that item (B) is always relevant. On the other hand, it may happen that $\beta=1$ so that $m_0=0$ and then item (A) is trivial.

We can rewrite the conclusion of the theorem as follows:

\begin{corollary}
In the setting of the theorem above,  we also define $\mu_0,\mu_1\in\Prob(T)$ and $0<\beta\le 1$ such that $m_0=(1-\beta)\mu_0$ and $m_1=\beta\mu_1$, so that the following holds:
 $$
   \mu=(1-\beta)\mu_0+\beta\mu_1\text{ with }\mu_0,\mu_1\in\Prob(T).
 $$
and
 \begin{itemize}
 \item[(A)] if $\beta<1$, then $\mu_0(\varphi)=0$;
 \item[(B)] for $\mu_1$-a.e. $x\in X$, we have $\liminf_n \frac1nS_n\varphi(x)>0$.
\end{itemize}
\end{corollary}

\begin{remark}
We do not claim uniqueness of the neutral decomposition even under the assumptions (A) and (B), see Exercise~\ref{exo-nonunique-neutral}.
\end{remark}

\begin{remark}
The neutral decomposition in our work \cite{BCS-3} with Crovisier and Sarig on SRB measures is slightly more general than the above which closely follows \cite{BCS-2}. Indeed, \cite{BCS-3} does not start with invariant measures (such as the $\nu_k$'s above) but has to work with longer and longer \emph{orbit segments}.
 One could do the same here by considering long enough orbit segments of typical points for the $\nu_k$ (see Remark~\ref{rem-orbit-seg}) but we want to keep these lectures as simple as possible.
\end{remark}

\subsection*{Ergodicity of the limit measure}
The above is interesting only when the limit measure $\mu\in\Prob(T)$ \emph{fails to be ergodic}. Indeed, if $\mu$ is ergodic then the neutral decomposition in \eqref{eq-neutral-decomp} must be trivial with $\beta=1$, and hence $\mu=\mu_1$. To see this, note that if $\mu$ is ergodic, then it is an extremal point in $\Prob(T)$, hence the decomposition must be trivial, i.e.,  $\beta=0$, $\beta=1$, or $\mu_0=\mu_1=\mu$.  Condition \eqref{eq-positive-avg} excludes $\beta=0$. Item (A) shows that $\mu_0=\mu$ imply $\beta=1$. Thus $\beta=1$.

We will later apply the above to surface diffeomorphisms by considering projective lifts. In that setting ergodicity of the limit of the lifts is not implied by the ergodicity of the limiting measure on the surface. See \S\ref{sec-applying-neutral}.

\subsection*{Ergodic components and neutral decompositions}
To get some intuition about the non-ergodic behavior of $\mu$, one can consider a measure with a finite ergodic decomposition $\mu=\sum_{i=1}^I \alpha^i\mu^i$. Then $\nu_k$-typical orbits will essentially be made of arbitrarily long intervals, each one nearly $\mu^i$-typical (say of ``type $i$'') in such a way that their union occupy a fraction close to $\alpha^i$ of $\nu_k$-typical orbits. These measures $\mu^i$ may have different averages $\mu^i(\varphi)$, even with opposite signs. The way the intervals of different types alternate can lead to different neutral blocks.

This shows that sequences with the same limit measure can have distinct neutral decomposition. Let us see an example.

\subsection*{Distinct decompositions for the same limit measure}

Let us show by an example that the neutral decomposition depends not only on the limit measure $\mu$ but on both sequences $\nu_k$ \emph{and} $k_i\to\infty$.

\smallbreak

We take $T$ the shift on $X:=\{-2,+1,+2\}^\ZZ$ and $\mu:=\frac13(\delta_{-2}+\delta_1+\delta_{+2})$ where $\delta_s$ is the Dirac measure concentrated on the sequence with the constant value $s$. Let $\vf:X\to\RR$ with $\vf(x):=x_0$.
Let
 $$\begin{aligned}
   \nu_{2k}\in\Proberg(T)\;\; &\text{be supported}\text{ on the periodic orbit}\\
   &\text{obtained by repeating }(\underbrace{-2,\dots,-2}_{k\; \rm times},\underbrace{+2,\dots,+2}_{k\; \rm times},\underbrace{+1,\dots,+1}_{k\; \rm times})\;,\\
 \nu_{2k+1}\in\Proberg(T)\;\; &\text{ be supported}\text{ on the periodic orbit}\\  
  &\text{obtained by repeating }(\underbrace{-2,\dots,-2}_{k\; \rm times},\underbrace{+1,\dots,+1}_{k\; \rm times},\underbrace{+2,\dots,+2}_{k\; \rm times}) \;.  
 \end{aligned}$$
The exercise below invites you to check that this sequence $\nu_k$ which weakly converges to $\mu$ admits at least two neutral decompositions.
 
\begin{exercise}\label{exo-nonunique-neutral}
Check that $\nu_k\to\mu$ in the weak topology. Prove that the following are distinct neutral decompositions, obtained by two distinct subsequences $\nu_{k_i}$:
 $$
   (1)\;  \mu = \frac{2}{3} \frac{\delta_{-2}+\delta_{+2}}2
      + \frac{1}{3} \delta_{+1}\; \quad
   (2) \; \mu = \frac{5}{6} \frac{2\delta_{-2}+2\delta_{+2}+\delta_{+1}}{5} 
      + \frac{1}{6} \delta_{+2}.   
 $$
\end{exercise}

Though there is no serious difficulty in building the decomposition, its uniqueness properties are not so clear. Let us state two problems in this direction.

\begin{problem}
Is there another decomposition with with better properties such as a smaller $\beta$ (leading to a stronger entropy bound in our application)? 

Is therea decomposition invariant by time-reversal, i.e., replacing $T$ by $T^{-1}$ (leading to a stronger semicontinuity result about the Hausdorff dimension)?
\end{problem}

\section{Construction of a neutral decomposition}

In this section, we build a $\vf$-neutral decomposition for a given sequence $\nu_k\to\mu$ and function $\vf:X\to\RR$, proving Theorem~\ref{thm-neutral-dec}. 

\medbreak
Recall that $\Neutral(\alpha,L)\subset X$ is the measurable set of points in $X$ that are contained in some $\alpha$-neutral block of length at least $L$. In particular,
\begin{equation}\label{eqNeutralMono} 
  \forall 0<\alpha\le\alpha'\;\forall L\ge L'\quad
  \Neutral(\alpha,L)\subset \Neutral(\alpha',L').
\end{equation}

\medbreak

For this proof it is convenient to say that a subsequence $k_i\to\infty$ is \emph{admissible} if for any $\alpha>0,L\ge1$ the following weak limit exists:
 $$
   m_{0}^{\alpha,L} := \lim_{i} \left(1_{\Neutral(\alpha,L)}\cdot \nu_{k_i}\right)
 $$
Note that the above weak limit always belong to the set $\SProb(X)$ of Borel sub-probabilities on $X$.
Note also that this asks for the convergence of uncountably many sequences $(\alpha_j,L_j)\to(0,\infty)$ but the monotonicity \eqref{eqNeutralMono} will save the day.

\medbreak

The first step of the proof of the theorem is to build an admissible sequence.
\begin{lemma}\label{lem-adm-seq}
There exists an admissible sequence $k_i\to\infty$.
\end{lemma}

\begin{proof}
In this proof, we use the notation $\mu_{\alpha,L,\ell}:=1_{\Neutral(\alpha,L)}\cdot\nu_{\ell}$.

We rely on the classical Cantor's diagonal argument in compact metric spaces such as $[0,1]$ or $\SProb(M)$. Let us first remark that since $L$ takes only countably many values, it is enough to find, for each $L\in\NN$, a subsequence $(k_i)$ so that $(\mu_{\alpha,L,k_i})$ converges (note that $k_i$ depends on $L$). Indeed, the diagonal Cantor argument will then conclude.

Thus we fix $L\in\NN^*$.
We let $\cA$ be a dense countable subset of $(0,1)$. An application of Cantor's diagonal argument yields a strictly increasing sequence of integers $\ell_i$ such that:
 $$
    \forall \alpha\in \cA \quad
    \text{the limit }m_0^{\alpha,L}:=\lim_i \mu_{\alpha,L,\ell_i}
    \text{ exists}.
 $$
 
We extend the function $\alpha\mapsto m_0^\alpha$ from $\cA$ to $(0,1)$.
Let $\mathcal U$ be a countable dense subset of the continuous nonnegative functions on $X$.
We fix $u\in\mathcal U$ and consider the function 
 $$
   \alpha\in(0,1)\longmapsto \tm_0^{\alpha,L}(u):=\sup\{m_0^{\beta,L}(u):\beta\in \cA\cap(0,\alpha]\}.
 $$
Since $\beta\mapsto m_0^{\beta,L}$ is non-decreasing, for $\alpha\in\cA$, we have $\tm_0^{\alpha,L}(u)=m_0^{\alpha,L}(u)$ for all $u\in\cU$.

The function $\alpha\mapsto\tm_0^{\alpha,L}(u)$ being monotonic, its set of discontinuities $\Delta_{L,u}$ is countable and so is  $\Delta_L:=\bigcup_{u\in\mathcal U}\Delta_{L,u}$. 
Therefore Cantor's diagonal argument can be applied again, extracting $(k_i)$ from $(\ell_i)$ such that $m_0^{\alpha,L}(u):=\lim_i \mu_{\alpha,L,k_i}(u)$ does exist for all $\alpha\in \cA\cup\Delta$ and all $u\in\cU$.

Now, let  $0<\alpha_0<1$ be arbitrary (no longer contained in a countable set). We will show the existence of $\lim_i \mu_{\alpha_0,L,k_i}$. If $\alpha_0\in\cA\cup\Delta$, this is obvious. So we can assume that $\alpha_0\notin\cA\cup\Delta$. In particular, $\alpha=\alpha_0$ is a point of continuity of $\alpha\mapsto \tm^\alpha_0(u)$ for each $u\in\cU$. Fix some $u\in\cU$ and $\eps>0$. There exist $\alpha_-<\alpha_0<\alpha_+$ such that
 $$
    \tm_0^{\alpha_0,L}(u)-\eps < \tm_0^{\alpha_-,L}(u) \le
    \tm_0^{\alpha_+,L}(u) < \tm_0^{\alpha_0,L}(u)+\eps.
 $$
Without lost of generality, we can assume that $\alpha_-,\alpha_+\in\cA$ so that
$m_0^{\alpha_-,L}(u) = \tm_0^{\alpha_-,L}(u)$ and $m_0^{\alpha_+,L}(u) = \tm_0^{\alpha_+,L}(u)$. Thus, for some $i_0$,
 $$
  \forall i\ge i_0\quad \tm_0^{\alpha_0,L}(u)-\eps <  \mu_{\alpha_-,L,k_i}(u) \le
     \mu_{\alpha_+,L,k_i}(u) < \tm_0^{\alpha_0,L}(u)+\eps\;,
 $$
and since each $\alpha\mapsto \mu_{\alpha,L,k_i}(u)$ is nondecreasing, we obtain:
 $$
  \forall i\ge i_0\quad \tm_0^{\alpha_0,L}(u)-\eps < \mu_{\alpha_0,L,k_i}(u) < \tm_0^{\alpha_0,L}(u)+\eps\;.
 $$
Since $\eps>0$ is arbitrary, this means that $\lim_i \mu_{\alpha_0,L,k_i}(u)$ exists for each $u\in\cU$. This is easily seen to imply the weak  convergence of $\lim_i \mu_{\alpha_0,L,k_i}$ in $\SProb(M)$, as claimed.
\end{proof}

\medbreak

\begin{remark}\label{rem-orbit-seg}
Note that $m_0^{\alpha,L}$ can be written as a ``\emph{partial quasi-empirical measure}''. Indeed, from the definition of the weak (star) topology, if $u\in C(X)$, then by Birkhoff pointwise ergodic theorem, for any $\nu_{k_i}$-generic point $x_i$, we have
 $$\begin{aligned}
   m_0^{\alpha,L}(u) &= \lim_{i} \nu_{k_i} (1_{\Neutral(\alpha,L)}u)=
   \lim_{n\to\infty} \frac1n\sum_{m=0}^{n} 1_{\Neutral(\alpha,L)}(T^mx_i)u(T^mx_i)\\
   &=\lim_i \lim_{n\to\infty} \left(\frac1n\sum_{m=0}^{n} 1_{\Neutral(\alpha,L)}(T^mx_i)\delta_{T^mx_i}\right)\cdot u
 \end{aligned}$$
\end{remark}

\begin{proposition}\label{prop-m0}
Fix some admissible subsequence $k_i\to\infty$. For each $\alpha>0$ and $L\ge1$, let $m_0^{\alpha,L}$ be the corresponding limit in $\SProb(M)$. Then the following weak limit exists:
 $$
   m_0:=\lim_{(\alpha,L)\to(0,\infty)} m_0^{\alpha,L}\;.
 $$
It is a $T$-invariant Borel sub-probability measure.
\end{proposition}

\begin{proof}
Let $u$ be some non-negative continuous function on $X$. The monotony properties of $\Neutral(\alpha,L)$ imply that
 $$
   \forall \alpha_1\ge\alpha_2\; \forall L_1\le L_2\quad 
    m_0^{\alpha_2,L_2}(u) \le m_0^{\alpha_1,L_1}(u).
 $$
Therefore, 
 $$
   \lim_{(\alpha,L)\to(0,\infty)} m_0^{\alpha,L}(u) 
    = \inf_{\alpha>0,L\ge 1} m_0^{\alpha,L} (u)
 $$
(including the existence of the limit). Since the above holds for any nonnegative $u$, this implies the existence of the weak limit $m_0:=\lim_{(\alpha,L)\to(0,\infty)} m_0^{\alpha,L}$.

\medbreak

Let us check the invariance of this limit $m_0$. 
For any $u\in C(X)$, we have $u\circ T-u\in C(X)$ and, by definition of weak convergence,
 $$
    m_0^{\alpha,L}(u\circ T-u) = \lim_i \nu_{k_i}(1_{\Neutral(\alpha,L)}(u\circ T-u))
 $$
By Lemma~\ref{lem-almost-inv} and the ergodic theorem,
  $$
     |\nu_{k_i}(1_{\Neutral(\alpha,L)}(u\circ T-u))| \le \sup|u| \cdot \frac{2}L.
  $$
Finally,
 $$
   m_0(u\circ T-u)  = \lim_{\alpha,L} m_0^{\alpha,L}(u\circ T-u) = 0.
 $$
This shows that $(T_*(m_0^{\alpha,L})-m_0^{\alpha,L})(u)=0$ for all $u\in C(X)$. Thus, $T_*(m_0^{\alpha,L})=m_0^{\alpha,L}$, i.e., the measure is $T$-invariant.
\end{proof}

\section{Properties of the neutral decomposition}
To prove the theorem, we apply Lemma~\ref{lem-adm-seq} to get an admissible sequence, and Proposition~\ref{prop-m0} to get  corresponding subprobability measures $m_0,m_1$. It remains to show that this decomposition has $m_1\ne0$ and satisfies the claims (A) and (B).
To simplify notations, we replace $(\nu_k)_{k\ge1}$ by the subsequence $(\nu_{k_i})_{i\ge1}$ defined by the admissible sequence.

\subsection*{Proof of (A) and $m_1\ne0$}
We assume $m_0\ne0$.
(A) is a simple consequence of the boundedness of $\vf$ and the maximality of the neutral blocks. By the Birkhoff theorem, $m_0^{\alpha,L}(\vf)$ is an average over $x$ of the limits:
 $$
    \lim_{n\to\infty} \frac1n\sum_{k=0}^{n-1} 1_{\Neutral(\alpha,L)}(T^kx)\vf(T^kx)
 $$
In particular, it is an average of $\vf(T^kx)$ for $k\in N_x(\alpha,L)$ (i.e., $T^{-k}x\in\Neutral(\alpha,L)$). To estimate $m_0^{\alpha,L}(\vf)$ it is therefore enough to estimate its average over neutral blocks.

Now, if $\ii{a,b}$ is an $(\alpha,L)$-neutral block for $x$, then
 $$
     \frac{1}{b-a}\sum_{k=a}^{b-1} \vf(T^kx) \le \alpha\;,
 $$
whereas the maximality of neutral blocks implies $\sum_{k=a}^b\vf(T^kx) > (b-a+1)\alpha$ so
 $$
     \frac{1}{b-a}\sum_{k=a}^{b-1} \vf(T^kx) > \alpha + \frac{\alpha-\vf(T^{b}x)}{b-a}.
 $$
Hence
 $$
    \alpha-\frac{\alpha+\|\vf\|_{\sup}}{L} \le m_0^{\alpha,L}(\vf)\le \alpha
 $$
Thus, in any limit where $L\to\infty$, $m_0(\vf)=\alpha$ and $m_0$ is obtained as $\alpha\to0$, proving (A).

Therefore $\mu(\vf)=m_0(\vf)+m_1(\vf)=m_1(\vf)$, but this is positive by assumption~\ref{eq-positive-avg}. Thus $m_1\ne0$ or $\beta\ne 0$.

\subsection*{Proof of (B)}
We proceed by contradiction, assuming
 \begin{equation}\label{eq-gamma-m1}
     \gamma := \frac12 m_1\left\{x\in X:\liminf_{n\to\infty}(1/n)S_n\vf(x)\le0\right\} > 0.
  \end{equation}
We are going to use Pliss lemma (Theorem~\ref{thm-Pliss}) to deduce from the above assumption the existence of many $(\alpha_0,L_0)$-neutral blocks disjoint from all $(\alpha_1,L_1)$-neutral blocks for $0<\alpha_1\ll\alpha_0\ll 1$ and $L_1\gg L_0\gg 1$. This will contradict the construction of $m_0$ as the limit of $m_0^{\alpha,L}$.

This convergence gives $\alpha_0>0$ and $L_0\ge1$ such that
 $$
   \forall 0<\alpha\le\alpha_0\;\forall L\ge L_0\quad \left| m_0^{\alpha,L}(X) - m_0(X) \right| < \frac{\gamma}{100}.
 $$
Note that we can assume that $\alpha_0$ does not belong to the countable set of values $\alpha$ such that $\mu\{x:\exists n\ge1$  s.t. $(1/n)S_n\vf(x)=\alpha\}>0$.
For $K\ge1$, define
 $$\begin{aligned}
   &W_0(K) := \{ x\in X: \exists 0\le a<K\; \ii{-a,0} \text{ is ($\alpha_0,L_0$)-neutral}\}.
  \end{aligned}$$
By the choice of $\alpha_0$, $\mu(\partial W_0(K))=0$ for any integer $K\ge1$. Thus, by construction of $m_1$ as a weak limit,
 $$
   \forall K\ge1\;\lim_{\alpha,L} \lim_k \nu_k(W_0(K)\setminus\Neutral(\alpha,L))
   = m_1(W_0(K)).
 $$
By the ergodic theorem, for $\mu$-a.e. $x\in X$, 
 $$
    \lim_n(1/n)(\vf(T^{-1}x)+v(T^{-2}x)+\dots+\vf(T^{-n}x))=\liminf_n(1/n)S_n\vf(x),
 $$
and this is nonpositive on a set of $m_1$-measure $2\gamma>0$ by \eqref{eq-gamma-m1}. For such an $x$, Pliss lemma implies that there are infinitely many integers $a\ge1$ such that $\ii{-a,0}$ is $\alpha_0$-neutral.\footnote{This is where we use $\alpha_0>0$.}  
Therefore, for $K_0$ large enough (in particular $K_0\ge L_0$), $m_1(W_0(K_0))>\gamma$. Thus, there are numbers $0<\alpha_1<\alpha_0$ and $L_1>200K_0/\gamma$ such that
 $$
    \lim_k \nu_k(W_0(K_0)\setminus\Neutral(\alpha_1,L_1))=:\delta > \gamma.
 $$

For each $x\in X$, we introduce the following subsets of $\ZZ$:
 $$\begin{aligned}
    &N_i(x):=\{n\in\ZZ:T^n(x)\in\Neutral(\alpha_i,L_i)\} \quad (i=0,1) \text{ and }\\
    &\Delta(x) := \{n\in\ZZ:T^n(x)\in W_0(K)\setminus\Neutral(\alpha_1,L_1)\}.
  \end{aligned}$$
Since $\nu_k$ is ergodic, the following densities exist for $\nu_k$-a.e. $x\in X$:
 $$\begin{aligned}
    &\dens(N_i(x)):=\limsup_{n\to\infty}\frac{\#(N_i(x)\cap[0,n))}{n} = \nu_{k}(\Neutral(\alpha_i,L_i)) \stackrel{k\to\infty}\longrightarrow m_0^{\alpha_i,L_i}(X)\\
    &\dens(\Delta(x)) = \nu_k(W_0(K)\setminus\Neutral(\alpha_1,L_1)) \stackrel{k\to\infty}\longrightarrow \delta > \gamma.
  \end{aligned}$$
Thus, for $k$ large enough, for $\nu_k$-a.e. $x\in X$, we have
 $$\begin{aligned}
    &\dens(N_i(x)) \in(m_0(X)-\gamma/100,m_0(X)+\gamma/100)\text{ for }i=0,1;\\
    &\dens(\Delta(x))  > \gamma.
  \end{aligned}$$
From now on, we omit the dependence on $x$.
 The estimates above imply that
 $$
    \dens(N_0) < \dens(N_1) + \gamma/50
 $$
We are going to show that this contradicts $N_1\cup\Delta\subset N_0$ by finding a large subset $\Delta'\subset\Delta$ that is disjoint from $N_1$. 

By construction, every $j\in\Delta$ is associated to an $(\alpha_0,L_0)$-neutral block $I(j):=\ii{j-a,j}$ with length $L_0\le j\le K_0$ (we do not claim that $I(j)$ is a maximal neutral block). We consider
 $$
   \Delta' := \bigcup\{I(j):j\in\Delta\text{ s.t. }I(j)\cap N_1\ne\emptyset\}
 $$
Obviously $\Delta\setminus\Delta'$ is disjoint from $N_1$.
To bound the density of $\Delta'$ in the integers, we are going to show that $\Delta'$  is contained in the union of the \emph{$K_0-1$-boundaries}  of the maximal $(\alpha_1,L_1)$-neutral blocks. By the $c$-boundary of an integer interval $\ii{a,b}$ for $c\in\NN$, we mean the subset $\ii{a-c,a+c}\cup\ii{b-c,b+c}$ of $\ZZ$.

To this end, let $(\ii{m_r,M_r})_{r\in\ZZ}$ be the family of the disjoint maximal $(\alpha_1,L_1)$-neutral blocks so that $N_1=\bigsqcup_{r\in\ZZ} \ii{m_r,M_r}$. We claim that
 \begin{equation}\label{eq-contrad1}
   \Delta' \subset \bigcup_{r\in\ZZ} \ii{M_r-K_0+1,M_r+K_0-1}\;.
  \end{equation}
This will imply that
 $$ 
   \dens(\Delta') \le (2K_0/L_1)\dens(N_1) < \gamma/100
 $$
so that we get our contradiction: 
 $$
    \dens(N_0) \ge \dens(\Delta\cup N_1) > m_0(X) + \gamma/2.
 $$
Let us prove the claim \eqref{eq-contrad1}. Consider $j\in\Delta$ such that $I(j)=\ii{j-a,j}$ intersects $A_1$, i.e., some maximal $(\alpha_1,L_1)$-neutral block $\ii{b,c}$. Since $j\notin N_1$,  $c< j$ so $j-a>c-K_0$. Obviously, $c>j-K_0$. Thus, $c-K_0<j-a<j<c+K_0$. 

\medbreak

This concludes the proof of (B) and of Theorem~\ref{thm-neutral-dec}.\qed

\section{Neutral decomposition for the projectivization of surface diffeomorphisms}
\label{sec-applying-neutral}

We state the consequences of the neutral decomposition for the projective dynamics of a surface diffeomorphism and make some comments.

\subsection*{Statement of the neutral decomposition}

The neutral decomposition (Theorem~\ref{thm-neutral-dec}) yields the following:

\begin{theorem}\label{thm-neutral-decomposition-projective}
Let $f\in\Diff^1(M)$ be a $C^1$ diffeomorphism of a closed surface. Let $\nu_k\in\Proberg(f)$ be ergodic, aperiodic\footnote{A measure is aperiodic if it gives zero measure to the set of all periodic points.} measures with
 $$
   \lambda^+(\nu_k)\ge\chi>0 \qquad (\text{for some }\chi>0).
 $$
In particular, the unstable lifts $\hnu_k^+\in\Proberg(\hf)$ are well-defined.
Recall the dilation function $\hvf:\hM\to\RR$, $(x,E)\mapsto\log\|Df_x|E\|$.

Assume that the unstable lifts $\hnu_k^+$ weakly converge to some $\hmu\in\Prob(\hf)$. Then there are invariant probability measures $\hmu_0,\hmu_1\in\Prob(\hf)$ and a number $0<\beta\le 1$ such that
 $$
    \hmu = (1-\beta)\hmu_0 + \beta \hmu_1
 $$
with the following properties ($\mu_i:=\hpi_*(\hmu_i),\mu:=\hpi_*(\hmu)$):
 \begin{enumerate}
   \item $\lim_k \lambda^+(\nu_k) = \beta\cdot\hmu_1(\hvf) = \beta\cdot \lambda^+(\mu_1)$;
   \item $\hmu_0(\hvf) = 0$ (if $\beta\ne1$);
   \item for $\mu_1$-a.e. $x\in M$, $\lambda^+(x)>0$. 
   \item for $\mu_1$-a.e. $x\in M$, $\lambda^-(x)\le0$.
 \end{enumerate}
\end{theorem}

\begin{proof}
Let $f,\nu_k,\hnu_k^+,\hmu$ be as above. We apply Theorem~\ref{thm-neutral-dec} with $T=\hf$, $X=\hM$, $\vf=\hvf$, and the measures $\hnu_k^+$. Note that the assumption in \eqref{eq-positive-avg} is satisfied, since
 $$
   \liminf_{k\to\infty} \hnu_k^+(\hvf) \ge \chi>0.
 $$
The theorem gives:
 $$
    \hmu = (1-\beta)\hmu_0 + \beta \hmu_1 \text{ with }\hmu_0,\hmu_1\in\Prob(\hf) \text{ and }0<\beta\le1,
 $$
with the two items: (A) $\hmu_0(\hvf)=0$ (if $\beta\ne1$) and (B)  $\hmu_1$-a.e. $\lim_n(1/n)S_n\hvf>0$. We are going to prove items (2), (3), (4), and (1), in that order.

\medbreak

We note that item (2) is a restatement of item (A).

\medbreak
We consider item (3). First,
we claim that 
 \begin{equation}\label{eq-lambda-nuk}
   \forall k\ge1\quad \lambda^+(\nu_k)=\hnu_k^+(\hvf).
 \end{equation}
This will follow from Proposition~\ref{prop-lift-simple} as soon as we see that the top exponent of the ergodic measure $\nu_k$ is simple.
By \eqref{eq-positive-avg}, $\nu_k$ has at least one positive exponent. If it had two (counting with multiplicity), it would be a sink, contradicting the aperiodicity assumption. This concludes the proof of  \eqref{eq-lambda-nuk}.

\medbreak

We then check that $\lambda^+(\mu_1)=\hmu_1(\hvf)$. Note that $\hmu_1$ is not necessarily ergodic. By item (B), $\mu_1$-a.e. there exists some $E\in T_xM$ satisfying
 $$
    \lim_n \frac1n\sum_{k=0}^{n-1}\log\|Df|Df^k(E)\| = \lim_n \frac1n\log\|Df^n_x|E\| >0.
 $$
This says that some exponent and therefore the top exponent is positive at $x$. We have proven item (3).

\medbreak

As previously, one cannot have two positive exponents (counting multiplicities) on a set of positive measure. Indeed, this would imply the existence of a source in the ergodic decomposition of $\mu_1$ and therefore of $\mu$, but this would force $\nu_k$ to be atomic, a contradiction. Thus, 
 $$
   \text{for $\mu_1$-a.e. $x\in M$} \quad \lambda^+(x)>0\ge\lambda^-(x).
 $$
We have proven item (4). 

The above inequalities imply that $\lambda^+(x)$ is a simple exponent at  $\mu_1$-a.e. $x\in M$. In particular, the space $E^1_x$ is defined 
and satisfies
 $$
   \lambda^+(x) = \lim_{n\to\infty} \frac{1}{n} S_n\hvf(x,E_x^1) > 0.
 $$
Hence,
 \begin{equation}\label{eq-lambda-mu1}
   \lambda^+(\mu_1) = \int_M \lambda^+(x)\, d\mu_1 = \hmu^1(\hvf).
 \end{equation}
We can now prove item (1):
 $$\begin{aligned}
   \lim_k \lambda^+(\nu_k) &\stackrel{(a)}= \lim_k \hnu_k^+(\hvf) \stackrel{(b)}= \hmu(\hvf) \stackrel{(c)}= (1-\beta)\hmu_0(\hvf)+\beta\cdot\hmu_1(\hvf)\\  &\stackrel{(d)}= \beta\cdot\hmu_1(\hvf) \stackrel{(e)}= \beta\cdot\lambda^+(\mu_1)\;,
 \end{aligned}$$
where (a) is \eqref{eq-lambda-nuk}, (b) is weak convergence, (c) the neutral decomposition, (d) is item (A), and (e) is \eqref{eq-lambda-mu1}.
\end{proof}

\subsection*{Some special cases and comments}
It is interesting to examine some special cases to make sense of the above decomposition. We wish also to point out some open questions.

\medbreak
\subsubsection*{When the limit $\mu$ is ergodic} In this case, the \emph{projection on the surface} of the neutral decomposition must be trivial: $\mu_0=\mu_1=\mu$. However, the decomposition in $\hM$:
 $$
   \hmu = (1-\beta)\cdot \hmu_0 + \beta \hmu_1
 $$
is not necessarily trivial.
We can compute it as follows. The ergodic $\mu:=\hpi_*(\hmu)$ has a single, simple positive exponent and a single, nonpositive one. Hence Proposition~\ref{prop-lift-simple} implies that $\hmu,\hmu_0,\hmu_1$ which are lifts of $\mu$, must be convex combinations of $\hmu^+,\hmu^-$, the well-defined unstable and stable lifts of the ergodic measure $\mu$.

Item (3) of the above theorem implies that $\hmu_1\ll\hmu^+$, hence $\hmu_1=\hmu^+$ by the ergodicity of $\hmu^+$. Item (2) therefore implies that
 $$
   \hmu_0=\frac{\lambda^+(\mu)\hmu^- - \lambda^-(\mu)\hmu^+}{\lambda^+(\mu) - \lambda^-(\mu)}\;.
  $$
Finally, writing $\hmu=\alpha\cdot\hmu^+ + (1-\alpha)\cdot\hmu^-$, we get:
 $$
   1-\beta = \left(1+\frac{|\lambda^-(\mu)|}{\lambda^+(\mu)}\right) (1-\alpha)\;.
 $$
This reflects exactly the fact that neutral blocks correspond to \emph{contraction plus cancellation}.

\subsubsection*{Time asymmetry} 
Given a sequence of measures $\nu_k\to\mu$, $\nu_k\in\Proberg(T)$, and some continuous function $\vf$, replacing $T$ by its inverse $T^{-1}$ does not leave the set of neutral orbit segments unchanged. This can already be seen in the ergodic example above. Using the notation introduced there, observe that $\alpha$ is invariant under time reversal, but $\beta$ is not, because of the factor
$1+|\lambda^-(\mu)|/\lambda^+(\mu)$ which becomes $1+\lambda^+(\mu)/|\lambda^-(\mu)|$. \footnote{This asymmetry is already present in the definition \ref{def-abstract-neutral-block} of neutral orbit segments and in their positions along the integers.} 

\begin{problem}
Is there a general argument that would get the best of both bounds in the above example, i.e., $1+\max\left(|\lambda^-(\mu)|/\lambda^+(\mu),\lambda^+(\mu)/|\lambda^-(\mu)|\right)$? Is there a way to get a decomposition which would be invariant under replacing $f$ by its inverse?
\end{problem}

We note that the last question is related to the problem of maximizing the Hausdorff dimension among hyperbolic measures.

\part{Entropy theory}

\begin{center}
{\Large Introduction}
\end{center}
\bigbreak

This second part is devoted to the necessary background on entropy theory. 
\medbreak

Chapter~\ref{chap-entropy} reviews the key results in the topological setting (using Bowen-Dinaburg dynamical balls) and the relations with exponents and unstable manifolds for hyperbolic ergodic measures. 

Chapter~\ref{chap-yomdin} reviews the basics of Yomdin theory and adapts these results to the projective extension.

\chapter{Background on entropy}\label{chap-entropy}

This chapter reviews the basics of entropy theory. We consider mainly continuous self-maps of compact metric spaces, before specializing to the smooth setting in order to relate entropy, exponents and unstable manifolds.

\medbreak

Until further notice, $f:M\to M$ is a self-map of a compact metric space $(M,d)$.

\section{Topological and Kolmogorov-Sinai entropies using dynamical balls}

As in many recent works (but not all) we will rely on Bowen's approach to topological entropy. This approach can be likened to a box dimension using dynamical distances. We note that this approach (also pioneered by Dinaburg) is different from the earlier and more general one of Kolmogorov and Sinai using measurable partitions.

\begin{definition}\label{def-htop-Bowen}
The $(\eps,n)$-\new{dynamical ball} with center $x\in M$ and radius $\eps>0$ is:
 \begin{equation}\label{eqDefBn}
     B_f(x,\eps,n) := \{ y\in M:\forall 0\le k<n\; d(f^ky,f^kx)<\eps\}.
  \end{equation}
The covering number
 $$
   r_f(Y,\eps,n) := \min\{\# C: Y\subset\bigcup_{x\in C} B_f(x,\eps,n)\}.
 $$
The \new{topological entropy} of a non-necessarily invariant subset $Y\subset X$ is\footnote{When $Y$ is not compact, Bowen's original definition of its topological entropy is: $\sup_{K\Subset Y} h_\top(f,K)$, where $K$ ranges over the compact subsets of $Y$. Hence it can be strictly smaller, e.g., for some discrete infinite sets.} 
 $$
   h_\top(f,Y) := \lim_{\eps\to0} h_\top(f,Y,\eps) \text{ where }
   h_\top(f,Y,\eps) := \limsup_{n\to\infty} \frac1n\log r_f(Y,\eps,n).
 $$
We write $h_\top(f):=h_\top(f,M)$.
\end{definition}

\begin{definition}\label{def-hKS-Katok}
Given $\mu\in\Proberg(f)$, its \new{Kolmogorov-Sinai entropy} is:
 $$
   h(f,\mu) := \lim_{\eps\to0} h(f,\mu,\eps) \text{ where }
   h(f,\mu,\eps) := \limsup_{n\to\infty}\frac1n\log \min\{r_f(Y,\eps,n):\mu(Y)>1/2\}.
 $$

Given a not necessarily ergodic $\mu\in\Prob(f)$ with ergodic decomposition $\mu=\int \mu_\xi d\xi$, its Kolmogorov-Sinai entropy is:\footnote{This formula is well-defined since $\mu\mapsto h(f,\mu)$ is a measurable non-negative function.}
 $$
   h(f,\mu,\eps) := \int h(f,\mu_\xi,\eps)\, d\xi  \text{ and }
   h(f,\mu) := \int h(f,\mu_\xi)\, d\xi.
 $$
\end{definition}

Observe that, even for a non-ergodic measure, $h(f,\mu)=\lim_\eps h(f,\mu,\eps)$ by monotonicity.

\begin{remark}
Kolmogorov and Sinai introduced first a more general notion of entropy for endomorphisms of probability spaces (see, e.g., \cite{CFS-book}). The above formulation, adapted to compact metric spaces, is due to Katok \cite{Katok-periodic}.
\end{remark}

It is a key feature of entropy that it is obtained by taking two successive limits. A key consequence is that entropy can be a highly irregular function over the set of invariant probability measures. We refer to Downarowicz's theory of entropy structures \cite{Downarowicz-ES} and Boyle symbolic extension theory \cite{Boyle-Downarowicz}. 

\section{The variational principle}

The basic link between topological and Kolmogorov-Sinai entropy is provided by the following fundamental results:

\begin{theorem}[Variational principle]
Let $f:M\to M$ be a continuous map on a compact metric space.
Then the \NEW{variational principle for entropy} holds:
 $$
   h_\top(f) = \sup\{h(f,\mu):\mu\in\Prob(f)\} = \sup\{h(f,\mu):\mu\in\Proberg(f)\}.
 $$
\end{theorem}

This celebrated result is due to Dinaburg \cite{Dinaburg} and Goodman \cite{Goodman}. Misiurewicz has given an elegant and powerful proof, see  \cite{Misiurewicz-VP}.

\begin{remark}
The above equality $\sup\{h(f,\mu):\mu\in\Prob(f)\} = \sup\{h(f,\mu):\mu\in\Proberg(f)\}$ is a consequence of the following formula for the entropy of $\mu$ in terms of its ergodic decomposition $\mu=\int \mu_\xi\, d\xi$:
 $$
    h(f,\mu) = \int h(f,\mu_\xi)\, d\xi\;.
 $$
Such a formula holds not only in the setting of the above theorem, but for any endomorphism of a standard probability space (i.e., a separable complete metric space equipped with the Borel $\sigma$-field). See \cite{Rokhlin-Lectures}.
\end{remark}

The variational principle puts the focus on the measures achieving the supremum:

\begin{definition}
Given $f:X\to X$ a Borel self-map of a Polish space, an  ergodic \NEW{measure maximizing the entropy (or MME for short)} is any $m\in\Proberg(f)$ such that
 $$
    h(f,m) = \sup\{h(f,\mu):\mu\in\Prob(f)\} = \sup\{h(f,\mu):\mu\in\Proberg(f)\}.
 $$
\end{definition}

\begin{exercise}
Show that $\mu\in\Prob(f)$ is an MME if and only a.e. of its ergodic components are MME. Deduce that the existence of an ergodic MME is equivalent to the existence of a not necessarily ergodic one.
\end{exercise}

The MMEs may fail to exist, including for diffeomorphisms of arbitrarily large, but finite smoothness: see \cite{Misiurewicz-NoMax} in dimension $4$ and \cite{Buzzi-NoMax} in dimension $2$. However, we have the following deep result:

\begin{theorem}[Newhouse]\label{thm-Newhouse}
Let $f:M\to M$ be a map on a closed manifold.
If $f$ is $C^\infty$ smooth, then there exists at least one measure maximizing the entropy. In fact, the map $\mu\mapsto h(f,\mu)$ is upper semicontinuous over the nonempty compact set $\Prob(f)$.
\end{theorem}

It is striking that the existence of MME is obtained in this theorem without any  ``explicitly dynamical'' assumptions. We will see that the proof is more geometrico-analytic. See below.

\section{Tail entropy}

The following quantity $h^*_\top(f)$ vanishes exactly when the convergence $h(f,\mu,\eps)\stackrel{\eps\to0}\longrightarrow h(f,\mu)$ is uniform with respect to $\mu\in\Prob(f)$. It was introduced using another formalism in \cite{Misiurewicz-CTE} under the name of \new{conditional topological entropy}. It turns out to play a key role in the theory of symbolic extension entropy \cite{Boyle-Downarowicz}.

\begin{definition}[Misiurewicz]
The \NEW{tail entropy} of a map $f$ of a compact metric space $M$ is:
 $$\begin{aligned}
    &h^*_\top(f) := \lim_{\eps\to 0} h^*_\top(f,\eps) \text{ where}\\
    &h^*_\top(f,\eps) := \lim_{\delta\to0} \limsup_{n\to\infty} \frac1n\log \max_{x\in M} r_f(B_f(x,\eps,n),\delta,n).
 \end{aligned}$$
\end{definition}

The vanishing of this tail entropy, called \new{asymptotic entropy-expansiveness}, generalizes expansivity:

\begin{exercise}
Let $f:X\to X$ be a homeomorphism of a compact metric space. Assume that $f$ is \new{expansive}, i.e.,  there is $\eps_0>0$ such that, for all $x,y\in X$, $\sup_{n\in\ZZ} d(f^nx,f^ny)<\eps_0 \implies x=y$. Show that it implies $h^*_\top(f)=0$.
\end{exercise}

The key property of the tail entropy is that it controls the speed of convergence of entropy at the scale $\eps\to0$. The following upper bound is easy:

\begin{exercise}
Let $f:M\to M$ be a continuous map of a compact metric space. Check that
 $$
   | h(f,\mu) - h(f,\mu,\eps)| \le h^*_\top(f,\eps)
 $$
Deduce that $h^*_\top(f)=0$ implies the uniform convergence as announced.
\end{exercise}

The converse is more difficult but true. It is a consequence of a variational principle of Downarowicz (see \cite{Burguet-VP} for a direct proof).

We will see later that $C^\infty$ smoothness implies that the tail entropy vanishes. This gives in particular Newhouse theorem:

\begin{exercise}
Let $f:M\to M$ be a continuous self-map of a compact metric space. Show that if $h^*_\top(f)=0$, then the entropy is upper semicontinuous over $\Prob(f)$ with the weak topology. Deduce that there exists an MME.
\end{exercise}

\section{Shannon-McMillan-Breiman (SMB) ergodic theorem for entropy}
Tne following ergodic theorem states a fundamental ``equidistribution of information'' property. 

\medbreak
We will use the (non-dynamical) notion of \new{mean Shannon entropy}: given a partition\footnote{We always assume that partitions only contain measurable subsets.} $P$ of $M$, we set:
 $$
    H_\mu(P):=-\sum_{A\in P} \mu(A)\log\mu(A).
 $$
We also use the \new{iterates of a partition} $P$ under $f:M\to M$,
 $$
    P^n:=\bigvee_{k=0}^{n-1} f^{-k}P:=\{A_0\cap\dots\cap f^{-n+1}A_{n-1}:A_0,\dots,A_{n-1}\in P\}\setminus\{\emptyset\}
 $$ 
and $P(x)$ for the element of $P$ containing a point $x\in M$.
To state the SMB theorem, we need to define the \new{entropy w.r.t. a partition} $P$,
 $$
    h(f,\mu,P):=\lim_{n\to\infty} \frac1n H_\mu(P^n).
 $$

\begin{remark}
The SMB does not require a topological setting. We also note that the Kolmogorov-Sinai entropy can be computed (and historically was defined) as: $h(f,\mu)=\sup_{P:H_\mu(P)<\infty} h(f,P)$. 

In particular, $h(f,P,\mu)\le h(f,\mu)$ with equality if $P$ is \new{generating under the map}: for all $x,y$ in some set of full measure, $x\ne y\implies\exists n\in\ZZ\; P(f^nx)\ne P(f^ny)$. 

However, we will not use these important facts.
\end{remark}

\begin{theorem}\label{thm-SMB}
Let $f:M\to M$ be a homeomorphism of a compact metric space and let $P$ be a partition of $M$ such that $H_\mu(P)<\infty$. Then the following limit exists $\mu$-a.e. and in $L^1(\mu)$:
 $$
    h(f,P,x) := \lim_{n\to\infty} -\frac1n\log\mu(P^n(x)).
 $$
Moreover, the limiting function is measurable and satisfies $h(f,P,f(x))=h(f,P,x)$ and 
 $$
   \int h(f,P,x)\, d\mu = h(f,P,\mu).
 $$

In particular for ergodic systems, $h(f,P,x)=h(f,P,\mu)$ for $\mu$-a.e. $x\in M$.
\end{theorem}

\begin{exercise}
Show that if $P$ is additionally generating under $f$ and $\mu$ is ergodic, then, for $\mu$-a.e. $x\in M$,
 $$
    \lim_{n\to\infty} -\frac1n\log\mu(P^n(x)) = h(f,\mu).
 $$
\end{exercise}

We refer to \cite[(13.4)]{DGS} for the proof of the above statement (in particular, with no ergodicity assumption).

\begin{remark}
It can be shown that, $\mu$-a.e., $h(f,x)=h(f,\mu_x)$ where $\mu_x:=\lim_{n\to\infty}\frac1n\sum_{k=0}^{n-1}\delta_{f^kx}$, the empirical measure.
\end{remark}

\begin{remark}
Brin and Katok have proved a version of the SMB theorem where $P^n(x)$ is replaced by the dynamical ball of order $n$, center $x$, for a small radius (see \cite{KH-book}).
\end{remark}

\section{Entropy, Lyapunov exponents and unstable manifolds}

We turn to smooth dynamics and to geometric properties related to entropy.

\subsection{Bounding entropy by exponents}

The fundamental \NEW{Ruelle-Mar\-gu\-lis inequality} is due to Margulis for volume measures and to Ruelle in general. It provides a fundamental link between entropy and a weak form of hyperbolicity studied especially by Pesin (see below).

\begin{theorem}[Ruelle-Margulis]\label{thm-Ruelle-Margulis}
Let $f\in\Diff^1(M)$ with $M$ a compact manifold and let $\mu\in\Proberg(f)$. Then\footnote{Recall that the multiplicity of $\lambda^i(f,\mu)$ is $m^i(f,\mu)=\dim E^i_x$ a.e. for an ergodic $\mu$.}
 $$
    h(f,\mu) \leq \sum_{i=1}^d m^i(f,\mu)\cdot\max(\lambda^i(f,\mu),0).
 $$
\end{theorem}

\begin{exercise}
A predecessor of Ruelle-Margulis inequality is the following  inequality due to Kushnirenko. Let $T:X\to X$ be a self-map with Lipschitz constant $L$ on a compact metric space $X$ with upper box dimension $B$. Then prove that $h_\top(T) \le B\cdot \log^+ L$.
\end{exercise}

In particular, positive entropy implies the existence of nonvanishing Lyapunov exponents.

\begin{corollary}\label{coro-Ruelle-Pesin}
Let $f\in\Diff^1(M)$ with $M$ a compact surface and let $\mu\in\Proberg(f)$.
If $h(f,\mu)>0$, then $\mu$ has both $\lambda^+(f,\mu)\ge h(f,\mu)$ and $\lambda^-(f,\mu)\le -h(f,\mu)$.
\end{corollary}

\begin{exercise}
Show by example that for diffeomorphisms of $3$-dimensional manifolds, one can have an invariant ergodic probability measure with both positive entropy and some zero exponent.
\end{exercise}

\begin{remark}
Ledrappier-Young theory allows one to \emph{express exactly} entropy of a diffeomorphism in terms of dimension and Lyapunov exponents.
\end{remark}

\section{Hyperbolicity in the sense of Pesin}

In the 1970s, building on Oseledets theorem, Pesin \cite{Pesin-RMS} established a theory of hyperbolic-like dynamics in a much more general setting than the classical Anosov-Smale theory.

\begin{definition}
Given $\chi>0$, a measure $\mu\in\Prob(f)$ is said to be $\chi$-\NEW{hyperbolic (in the sense of Pesin)} if, for $\mu$-a.e. $x\in M$, 
 $$
     \forall 1\le i\le r(f,x) \quad  |\lambda^i(f,x)|>\chi.
 $$
We say that the measure $\mu$ is \new{hyperbolic (in the sense of Pesin)}  if it is $\chi$-hyperbolic for some $\chi>0$.
 The measure has \new{saddle type} if it has both positive and negative exponents almost everywhere, i.e., for $\mu$-a.e. $x\in M$, $\lambda^1(f,x)>0>\lambda^{r(x)}(f,x)$.
\end{definition}

\begin{example}
If $f$ is a surface diffeomorphism, we can deduce from Ruelle-Margulis inequality (more precisely, Corollary~\ref{coro-Ruelle-Pesin}) that any $\mu\in\Proberg(f)$ with positive entropy is $\chi$-hyperbolic of saddle-type for any $0<\chi<h(f,\mu)$.
\end{example}

Given $\mu\in\Prob(f)$, we may apply Oseledets theorem \ref{thm-Oseledets}. For $\mu$-a.e. $x\in M$, the \NEW{unstable space} and \NEW{stable space} are:
 $$
     E^u_x := \bigoplus_{\tiny\begin{array}{c} 1\leq i \leq r(f,x) \\ \lambda^i(f,x) >0\end{array}} E^i_x
     \quad\text{ and }\quad
     E^s_x := \bigoplus_{\tiny\begin{array}{c} 1\leq i \leq r(f,x) \\ \lambda^i(f,x) <0\end{array}} E^i_x \; .
 $$ 
Hyperbolicity is equivalent to $T_xM = E^s_x\oplus E^u_x$.

\begin{definition}\label{def-WuWs}
For any given $x\in M$, the \NEW{unstable Pesin manifold} and \NEW{stable Pesin manifold} are the following subsets:
 $$\begin{aligned}
     &W^u(f,x) := \{ y\in M \mid \limsup_{n\to\infty} \frac1n\log d(f^{-n}y,f^{-n}x) < 0\}
     \text{ and }\\
     &W^s(f,x) := \{ y\in M \mid \limsup_{n\to\infty} \frac1n\log d(f^ny,f^nx) < 0\}.
 \end{aligned}$$
\end{definition}

Note that $W^u(f,x) = W^s(f^{-1},x)$.

\medbreak
The following result of Pesin ensures that these sets are often  smooth.
 
\begin{theorem}[Pesin stable manifold]
Let $f\in\Diff^r(M)$ with $1<r\le\infty$ and $M$ a closed manifold.
Let $\mu\in\Prob(f)$ be hyperbolic in the sense of Pesin. 
 
Then, for $\mu$-a.e. $x\in M$, $W^s(f,x)$ is a $C^r$ immersion of some Euclidean space with $T_x W^s(f,x) = E^s_x$. 
Moreover, $x\mapsto W^s(f,x)$ is measurable in the compact-open $C^1$-topology.
\end{theorem}

Let us note that in this (nonuniform) hyperbolic setting these Pesin manifolds are as smooth as the map (even though $x\mapsto W^s(f,x)$ is not even continuous).

\medbreak

Obviously, symmetric statements hold for the unstable Pesin manifolds.

\begin{remark}
It is not difficult to see that the sets $W^s(f,x)$, $W^u(f,y)$, can fail to be immersed manifolds at some points $x\in M$. A more interesting fact is that
one cannot allow $r=1$ in the above theorem, see \cite{BoCrSh-2013}. 
\end{remark}

\section{Unstable disintegration}

We explain what the ``disintegration along the unstable direction'' of some hyperbolic measure of a smooth diffeomorphism is. This will allow us to reduce the computation of entropy to the analysis of the dynamics of good subsets of (smooth) curves (see Corollary~\ref{cor-entropy-from-unstable}).

\medbreak

We first introduce some terminology and concepts from \emph{Stone-Rokhlin disintegration} of measures on ``reasonable'' spaces.

\subsection{Countably generated partitions and disintegrations}
The proper setting is that of a standard probability space $(M,\cM,\mu)$, i.e., one that is isomorphic mod zero to a Polish space with a Borel probability measure such that $\cM$ is $\mu$-complete.

\smallbreak

Let $\xi$ be a partition of $M$. Recall that, given $x\in M$, we denote by $\xi(x)$ the element of $\xi$ containing $x$. We will need the following classical type of partitions, often simply called ``measurable partitions'' (see \cite{Rokhlin-Lectures}).

\begin{definition}
The partition $\xi$ is a \new{countably generated measurable partition} of $(M,\cM,\mu)$ if it is a partition of $M$ such that there are countably many measurable subsets $E_1,E_2,\dots$ and a zero measure subset $N$ satisfying
 $$
    \forall x\in M\; \exists \eps\in\{-1,+1\}^{\NN^*}\quad
       \xi(x)\setminus N = \bigcap_{n\ge1} E_n^{\eps_n}
       \qquad (E_n^{-1}:=M\setminus E_n,\; E_n^{+1}:=E_n).
 $$
\end{definition}

The motivation for this definition is the following result (see, e.g., \cite{Viana-Oliveira-book} and also \cite{Rokhlin-Lectures}).

\begin{theorem}[Rokhlin]
Let $(M,\cM,\mu)$ be a standard probability space with a countably generated partition $\xi$. Then there exists a family of probability measures $\mu^x$ on $(M,\cM)$ satisfying, for all $x,y$ belonging to some subset of full measure,
 \begin{itemize}
  \item[--] $\mu^x=\mu^y$ if $\xi(x)=\xi(y)$;
  \item[--] $\mu^x(\xi(x))=1$;
  \item[--] for any $\phi\in L^1(\mu)$, $x\mapsto \mu^x(\phi):=\int\phi\,d\mu^x$ is a.e. well-defined\footnote{More precisely, one picks one representative $\overline{\phi}$ mod zero of $\phi\in L^1(\mu)$ and then check that $x\mapsto\mu^x(\overline{\phi})$ is a.e. unique.} and integrable with
   $$
      \int \phi\, d\mu = \int \left(\int \phi\, d\mu^x\right) \, d\mu(x).
   $$
 \end{itemize}
Moreover, the family is unique: if $(\tmu^x)_{x\in M}$ is another family with the same properties, then for $\mu$-a.e. $x\in M$, $\tmu^x = \mu^x$.
\end{theorem}

The family $\{\mu^x\}_{x\in M}$ is called the \new{disintegration of the measure} $\mu$ w.r.t. to $\xi$.

\subsection{Unstable disintegration}
It is well-known that the Pesin laminations define measurable partitions  (e.g., $\{W^s(f,x):x\in M\}$) which are not countably generated in positive entropy. One therefore ``approximates'' them in the following sense.

\begin{definition}\label{def-unstable-disint}
Given $\mu\in\Prob(f)$, a measurable partition $\xi^u$ is \new{subordinate to the unstable Pesin lamination} (modulo $\mu$) if, for $\mu$-a.e. $x\in M$, $\xi^u(x)$ is a compact neighborhood of $x$ in the intrinsic topology of the Pesin manifold $W^u(x)$.
\end{definition}

This can always be done:

\begin{theorem}[Ledrappier-Young]
Let $f\in\Diff^r(M)$ with $r>1$ and $M$ a closed manifold.
For any $\mu\in\Proberg(f)$, there exists a countably generated measurable partition $\xi^u$ subordinate to the unstable Pesin lamination.
\end{theorem}

Note that the subordinate partition above is not unique (e.g., it is not difficult to refine one).

\begin{definition}\label{def-u-disint}
In the above setting, we define the \new{unstable disintegration of the measure} $\mu$ as the Rokhlin disintegration $(\mu^u_x)_{x\in M}$ with respect to a  partition $\xi^u$ subordinate to the unstable Pesin lamination.
\end{definition}

Since $\xi^u$ is not unique, neither is the unstable disintegration. However, any two unstable disintegrations are closely related to one another:

\begin{theorem}[Ledrappier-Young]
Let $f\in\Diff^r(M)$ with $r>1$ on a closed manifold.
For any $\mu\in\Proberg(f)$, if $(\mu^u_x)_{x\in M}$ and $(\nu^u_x)_{x\in M}$ are two unstable disintegrations of $\mu$ defined by two partitions $\xi^u,\eta^u$ subordinate to the unstable lamination, then, for $\mu$-a.e. $x\in M$, $\mu^u_x$ and $\nu^u_x$ restricted to $\eta^u(x)\cap\xi^u(x)$ are proportional.
\end{theorem}

\subsection{Unstable entropy}
Following Ledrappier-Young, we are going to compute the entropy in terms of the dynamics along the unstable lamination. 

We first define the \new{unstable dynamical balls} for all $x\in M$ such that $W^u(f,x)$ is an immersed manifold. Given $v,w\in W^u(f,x)$, $d^u(v,w)$ is the geodesic distance defined by the restriction to $W^u(f,x)$ of the  Riemannian structure of $M$. The unstable balls are:
 \begin{equation}\label{eqDefBun}
   B^u_f(x,\eps,n) := \{y\in W^u(f,x): \forall 0\le k<n\; d^u(f^ky,f^kx)<\eps\}
  \qquad (\eps>0,\; n\ge1).
  \end{equation}
Notice that the distance $d^u(f^ky,f^kx)$ is defined in $W^u(f,f^kx)$. This is fine since, for $\mu$-a.e. $x\in M$, the Pesin manifolds $W^u(f,f^kx)$ are all smooth.

We note that $B^u_f(x,\eps,n)\subset B_f(x,\eps,n)\cap W^u(f,x)$, but this inclusion may be strict as the unstable manifold usually accumulate on itself.

\begin{theorem}[Ledrappier-Young]\label{thm-LY-u-entropy}
Let $f\in\Diff^r(M)$ with $r>1$ and $M$ a compact manifold.
For any hyperbolic $\mu\in\Proberg(f)$ with an unstable disintegration $(\mu^u_x)_{x\in M}$, for $\mu$-a.e. $x\in M$, we have
 $$
    h(f,\mu) = \lim_{\eps\to0} \limsup_{n\to\infty} -\frac1n\log\mu^u_x(B^u_f(x,\eps,n))
    = \lim_{\eps\to0} \liminf_{n\to\infty} -\frac1n\log\mu^u_x(B^u_f(x,\eps,n)).
 $$
\end{theorem}

This is \cite[section (7.2)]{LY2}. Note that in the case where $\mu$ is not ergodic, the limits above may fail to be constant a.e. In fact, they are a.e. equal to the entropy $h(f,\mu_x)$, where $\mu_x$ is the empirical measure at $x$.

\medbreak

We define the \new{unstable covering numbers} for some\footnote{Recall that $\Prob(M)$ is the set of not necessarily invariant Borel probability measures on $M$.} $\nu\in\Prob(M)$ and $0<\gamma<1$ as:
 $$
   N^u_{f,\gamma}(\nu,\eps,n) := \min\left\{\#C:\nu\left(\bigcup_{x\in C} B_f^u(x,\eps,n)\right)>\gamma\right\}\; .
   $$

\begin{corollary}\label{cor-entropy-from-unstable}
Let $f\in\Diff^r(M)$ with $r>1$ and $M$ a compact manifold. Pick any $0<\gamma<1$.

For any hyperbolic $\mu\in\Proberg(f)$ with an unstable disintegration $(\mu^u_x)_{x\in M}$, for $\mu$-a.e. $x\in M$,
 $$
    h(f,\mu) = \lim_{\eps\to0} \liminf_{n\to\infty} \frac1n\log N^u_{f,\gamma}(\mu^u_x,\eps,n) \;.
 $$
\end{corollary}

\begin{exercise}
Deduce the above corollary from Theorem~\ref{thm-LY-u-entropy}.
\end{exercise}

\section{Entropy in the projective bundle}\label{secEntropyProj}

We now combine the projective extension $(\hM,\hf)$ from Section \ref{secProjDynLyap} with the entropy theory presented in this chapter. We will see that the the projection $\hpi:\hM\to M$ preserves the entropies. We will also introduce the notion of double scales so that we can work at different scales  on $M$ and $\hM$.

\subsection{Lifting measures and entropy}
Since $\hM$ is compact, every $\mu\in\Prob(f)$ admits an invariant lift $\hmu\in\Prob(\hf)$ (we already saw a detailed analysis in Proposition~\ref{prop-lift-projective}). Now, note that, for every $x\in M$, the restriction $\hf:\hpi^{-1}(x)\to\hpi^{-1}(fx)$ is a projective homeomorphism between projective spaces.

In particular, the fiber entropy vanishes:
 \begin{equation}\label{eqZeroFib}
     \forall x\in M\quad h_\top(\hf,\hpi^{-1}(x)) = 0.
  \end{equation}
To simplify things, we restrict ourselves to the case where $\dim M=2$. Then the fibers $\hpi^{-1}(x)$ are circles with lengths bounded by some $C>0$. But any composition of circle homeomorphisms has zero entropy. Indeed, fix $\eps>0$ and, for each $x\in M$, let $P_x$ be a partition of the circle $\hpi^{-1}(x)$ into at most $\lceil C/\epsilon\rceil$ arcs with diameter less than $\epsilon$. Note that
 $$
    \#\bigvee_{0\leq k<n} \hf^{-k}(P_{f^kx}) \leq n \cdot C/\epsilon+1.
 $$
It follows that $r_f(\epsilon,n,\hpi^{-1}(x))$ grows linearly, hence the claim \eqref{eqZeroFib}.

\begin{exercise}
Prove \eqref{eqZeroFib} in any dimension. Hint: $n$ hyperplanes can cut $\RR^d$ into at most $\mathcal O(n^d)$ regions.
\end{exercise}

\medbreak

Using formulas of Bowen and Ledrappier-Walters \cite{Ledrappier-Walters-RVP} we obtain

\begin{lemma}
Let $f$ be a $C^1$ diffeomorphism of a closed manifold $M$. Let $\hf:\hM\to\hM$ be its projective extension. We have $\hpi_*(\Prob(\hf))=\Prob(f)$ and
 \begin{equation}\label{eq-h-M-hm}
   \forall\hmu\in\Prob(\hf)\quad h(\hf,\hmu) = h(f,\hpi_*\hmu)
   \text{ and }
   h_\top(\hf) = h_\top(f).
 \end{equation}
\end{lemma}

\subsection{Double scales}

We will need to choose different scales on the surface and on the projective fibers as follows.

\begin{definition}\label{defBdouble}
The \new{doubly scaled dynamical ball} with center $\hx\in\hM$, radiuses $\eps,\heps>0$, and order $n\ge1$, is
 \begin{equation}\label{eqDefBdouble}
      B_{\hf}(\hx,\eps,\heps,n) := \{ \hy\in\hM \mid \forall 0\le k<n\; d(f^k\hpi(\hy),f^k\hpi(\hx))<\eps,\; d(\hf^k\hy,\hf^k\hx)<\heps\}\;.
  \end{equation}
One can then define doubly scaled covering numbers:
 \begin{itemize}
  \item $r_{\hf}(K,\eps,\heps,n)$ is the minimum number of doubly scaled dynamical balls the union of which contains some set $K\subset \hM$;
  \item $r_{\hf,\gamma}(\nu,\eps,\heps,n)$ is the minimum number of doubly scaled dynamical balls the union of which has $\nu$-measure at least $\gamma$ where $\nu\in\Prob(\hM)$.
\end{itemize}

\end{definition}

One can use the above doubly scaled covering numbers to compute entropy. For instance,

\begin{lemma}\label{lem-h-double}
For any $\heps>0$, we have
 $$
      h_\top(\hf) = \lim_{\eps\to0} \limsup_{n\to\infty} \frac1n\log r_{\hf}(\hM,\eps,\heps,n) = \lim_{\eps\to0} \liminf_{n\to\infty} \frac1n\log r_{\hf}(\hM,\eps,\heps,n).
 $$
Moreover, for any $\hmu\in\Proberg(\hf)$, for any $0<\gamma<1$, it holds that
 $$
   h(\hf,\hmu) = \lim_{\eps\to0} \limsup_{n\to\infty} \frac1n\log r_{\hf,\gamma}(\hmu,\eps,\heps,n).
 $$
\end{lemma}

\begin{exercise}
Prove the above lemma.
\end{exercise}

\chapter{Yomdin theory}\label{chap-yomdin}

In this chapter, we introduce the notions and techniques that we need from Yomdin's theory. These ideas are very relevant to our goals in two respects:

First, Yomdin theory implies that, thanks to the $C^\infty$ smoothness, the tail entropy vanishes (Theorem~\ref{thm-tail-zero}), so we can estimate the entropies of the measures $\nu_k$, $k\ge1$, at a \emph{fixed} scale, independent of $k$;

Second, the reparametrization results (Theorem~\ref{thm-Yomdin} and its corollary) allow controlling the splitting of $C^r$ curves under the dynamics. More precisely, they imply that the number of reparametrizations needed to control the smooth size of the iterations of a curve does not grow much faster than the number of dynamical balls needed to cover that curve.

\section{Motivation and key ideas of Yomdin theory}

The key insight of Yomdin theory may be the observation that high smoothness forces low local complexity. The first applications were:
 \begin{itemize}
  \item[--] the proof of the Shub entropy conjecture for $C^\infty$ self-maps of compact manifolds \cite{Yomdin-VolumeGrowth1};
  \item[--] the existence of measures maximizing the entropy \cite{Newhouse-Continuity}.
 \end{itemize}
The latter application can be deduced from the following result from \cite{Buzzi-SIM}. We write:
 $$
    \lambda(f) := \lim_{n\to\infty} \frac1n\log \| Df^n \|_{\sup} \leq \log\Lip(f).
 $$
Here $\|Df^n\|_{\sup}$ is the supremum over $x\in M$ of the operator norm $\|Df^n\|$ induced from the norms $\|\cdot\|_x$ and $\|\cdot\|_{fx}$ defined by the Riemannian structure of $M$. Notice that $\lambda(f)$ does not depend on the choice of the Riemannian structure on the compact manifold $M$.

\begin{theorem}[Buzzi]\label{thm-tail-zero}
Let $f$ be a $C^r$ self-map of a compact manifold with $1<r\le\infty$, then its tail entropy satisfies:
 $$
   h^*_\top(f) \leq \frac{\dim M}{r} \lambda(f).
 $$
In particular, $h^*_\top(f)=0$ if  $f$ is $C^\infty$.
\end{theorem}

This gives a direct proof (i.e., not using Pesin theory) of Newhouse's Theorem~\ref{thm-Newhouse} \cite{Newhouse-Continuity}.

\begin{remark}
According to \cite{BoyleFiebigFiebig}, the vanishing of the tail entropy of $f:M\to M$ is  equivalent to the existence of a principal symbolic extension $\pi:\Sigma\to M$, i.e., $(\Sigma,\sigma)$ is a compact subshift, $\pi$ is continuous, onto and commutes ($\pi\circ\sigma=f\circ\pi$) and for every $\nu\in\Prob(\sigma)$, $h(f,\pi_*(\nu))=h(\sigma,\nu)$. See \cite{Boyle-Downarowicz}.
\end{remark}

A bit more technically, a key idea of Yomdin theory is that one can define the entropy by using \emph{reparametrizations}, i.e., maps with unit size from the differentiable point of view, instead of unit balls.

\section{Yomdin reparametrizations of curves in $\hM$}
\label{sec-Yomdin-reparam}

Throughout this section $r\ge1$ is some integer, $f:M\to M$ is a $C^r$ map, and $0<\eps,\heps<1$ are some numbers.

\subsection{Passing to the projective extension}
We note that $\hf$ is a $C^{r-1}$~diffeomorphism of the closed manifold $\hM$. The two are related by explicit estimates, such as the following.

\begin{exercise}
If $f:M\to M$ is a $C^1$ diffeomorphism of a closed manifold and $\hf:\hM\to\hM$ its projective extension, show that $\lambda(\hf) \le \lambda(f)+\lambda(f^{-1})$.
\end{exercise}

\subsection{Defining the $C^r$ size of disks or maps}
To define the $C^r$ size of a parametrized curve, we fix some finite $C^r$ atlas $\cA$ of $M$, i.e., a finite set of $C^r$~diffeomorphisms $\psi_i:U_i\to V_i$, $i=1,\dots,I$ where $U_i$ are open sets of $M$ with $\bigcup_i U_i=M$, and  $V_i$ are open sets of $\RR^{\dim M}$. 
We require additionally that each diffeomorphism $\psi_i$  can be extended to a $C^r$ diffeomorphism between larger open sets $\widetilde{U_i}\to\widetilde{V_i}$, where $\overline{U_i}\subset\widetilde{U_i}$ and $\overline{V_i}\subset\widetilde{V_i}$. This ensures that the following quantity $\|\sigma\|_{C^r}$ is always finite.

\begin{definition}\label{def-Cr-size}
For any integer $r\ge1$, the $C^r$ \new{size of a parametrized $k$-disk} $\sigma:[0,1]^k\to M$ is
 $$
   \|\sigma\|_{C^r} := \sup_{t\in[0,1]^k} \max_{\psi\in\cA}  \max_{1\le s\le r} \sup \left\| D_t^s(\psi\circ\sigma) \right\|
 $$
where $D^s$ denotes the differential of order $s$ and $\|\cdot\|$ is the max norm of $\RR^{\dim M}$.
\end{definition}

We note that this size is finite, defined only for finite $r$, and does depend on the choice of the finite atlas $\cA$.

\newcommand\cB{\mathcal B}
\begin{exercise}
Let $\|\cdot\|_{C^r,\cA}$, $\|\cdot\|_{C^r,\cB}$, be the $C^r$ sizes defined as above with respect to two finite atlases $\cA$, $\cB$. Show that for $A,B>1$ large enough,
 $$
  A^{-1} \cdot \|\cdot\|_{C^r,\cB} \le \|\cdot\|_{C^r,\cA} \leq B \cdot \|\cdot\|_{C^r,\cB}.
 $$
\end{exercise}

\medbreak

\begin{remark}
The above also applies to the projective bundle $\hM$ since it is also a $C^\infty$ smooth manifold.
\end{remark}

\subsection{Reparametrizations}

In light of Corollary~\ref{cor-entropy-from-unstable}, we only need to consider the complexity of \emph{curves} on $M$ and more precisely their lifts to the projective bundle $\hM$ in restriction to their typical points. This means that we only need the Yomdin theory in dimension 1, which is greatly simplified for the following reasons:
 \begin{enumerate}
  \item semi-algebraic subsets are just intervals,
  \item one can use only affine reparametrizations (instead of Nash functions, i.e., real-analytic semi-algebraic maps), see Definition~\ref{def-admissible-reparam},
    \item the join of two interval partitions is an interval partition with cardinality bounded by the sum of the cardinalities of the initial partitions.
\end{enumerate}  

\medbreak

Throughout this section, we fix some integer $r\ge2$, numbers $\eps,\heps>0$ and also some finite atlases on $M$ and $\hM$ as in Definition \ref{def-Cr-size}. For any $v\ne0$ in some vector space, we write $\RR.v$ for the one-dimensional subspace generated by $v$: $\RR.v:=\{tv:t\in\RR\}$.

\begin{definition}
Given $r\ge1$, a \new{parametrized curve} on $M$ is a $C^r$ map $\sigma:[0,1]\to M$. It is a \new{regular parametrized curve} if, additionally, the speed $\sigma'(t)$ never vanishes. In this case, its canonical lift (to $\hM$) is 
 $$
   \hsigma:[0,1]\to\hM\text{ with }\hsigma(t)=(\sigma(t),\RR.\hsigma'(t)) \text{ for }t\in[0,1].
 $$
\end{definition}

Throughout this section~\ref{sec-Yomdin-reparam}, $\sigma:[0,1]\to M$ is a regular $C^r$ curve. 
We note that if $\sigma$ is $C^r$, $\hsigma$ is $C^{r-1}$.

\begin{definition}
Let $\eps,\heps>0$.

A curve $\sigma$ has $C^r$ \NEW{double size} (at most) $(\eps,\heps)$ if:
 $$
    \| \sigma \|_{C^r} \leq \eps \text{ and } \| \hsigma \|_{C^{r-1}} \leq \heps.
 $$
 
A curve $\sigma$  has \NEW{double diameter} (at most) $(\eps,\heps)$ if:
 $$
    \diam(\sigma([0,1])) \le \eps \text{ and }
    \diam(\hsigma([0,1])) \le \heps.
 $$
We will often omit ``double'' when it is clear from the context.
\end{definition}

The following notion can be understood as a strong, parametrized version of a doubly scaled dynamical ball. 

\begin{definition}\label{def-admissible-reparam} 
Given $(C^r,f,N,\eps,\heps)$, an \NEW{admissible reparametrization} of a curve $\sigma:[0,1]\to M$  {\bf up to time} $n\in\NN^*$ is an affine map $\psi:[0,1]\to[0,1]$ such that, for some integers $n_0=0<n_1<\dots<n_\ell=n$ (the \new{admissible times}), for all $0<j\le \ell$, one has
 \begin{enumerate}
  \item $n_j-n_{j-1}\le N$;
  \item $f^{n_j}\circ\sigma\circ\psi$ has $C^r$ size at most $(\eps,\heps)$.
 \end{enumerate}
We will often write ``$(C^r,f,N,\eps,\heps)$-admissible reparametrization'', or ``$(C^r,f,N,\eps,\heps)$-reparametrization'',  and omit some information (such as $C^r,f,N$) whenever this does not lead to confusion.
\end{definition}

Let us stress the flexibility in the admissible times. It will allow us to concatenate iterates of varying order.

\begin{definition}
A family $\cR$ of \new{reparametrizations over} $T\subset[0,1]$ is a collection of reparametrizations such that
 $$
      T\subset \supp(\cR):=\bigcup_{\psi\in\cR} \psi([0,1]).
 $$
\end{definition}

Let us note that the admissible times will usually be different for each repara\-me\-tri\-zation in the given family.

\subsection{Application to dynamical coverings}
Families of reparametrizations yield dynamical coverings at \emph{any} double scale. Recall the doubly scaled covering numbers $r_{\hf}(\cdot,\eps,\heps,n)$ from Definition~\ref{defBdouble}.

\begin{proposition}\label{prop-param-to-ball}
Fix arbitrarily $\eps_0,\heps_0>0$ and $N\ge1$.
Let $\cR$ be a family of $(N,\eps_0,\heps_0)$-repa\-ra\-me\-tri\-za\-tions of a regular curve $\sigma$ over some $T\subset[0,1]$ up to some time $n$. Then we have the following upper bound for coverings by doubly scaled dynamical balls for any $\eps,\heps>0$:
 $$
    r_{\hf\,}(\hsigma(T),\eps,\heps,n)  \le \frac{2\heps_0\cdot\|D\hf\|^N_{\sup}}{\min(\eps,\heps)} \cdot \#\cR.
 $$
\end{proposition}

\begin{exercise}
Find some $C_r\subset[0,1]$ such that $\{t\in[0,1]:\exists c\in C_r\; |t-c|<r\}=[0,1]$ and $\# C_r\le 2/r$. Now let $r:=\frac{\min(\eps,\heps)}{\heps_0\cdot\|D\hf\|_{\sup}^N}$. 

Show that, in the setting of Proposition~\ref{prop-param-to-ball}, $\bigcup_{\psi\in\cR} \sigma\circ\psi(C_r)$ is a $(\eps,\heps,n)$-covering of $\sigma(T)$. 
Conclude the proof of the proposition.
\end{exercise}

This leads to the following version of a Ledrappier-Young entropy formula:

\begin{proposition}\label{prop-entropy-ureparam}
Let $f\in\Diff^2(M)$ with $M$ a closed surface and let $\mu\in\Proberg(f)$ with $h(f,\mu)>0$. Let $N$ be a positive integer and $\eps,\heps>0$. 
Then for $\mu$-a.e. $x\in M$, the following holds if $\sigma$ is a regular $C^r$ parametrization of the local unstable manifold $W^u_\loc(x)$:

If, for every $n\ge1$,  $\cR_n$ is a family of $(C^r,N,\eps,\heps)$-reparametrizations of $\sigma$ until time $n$ over $T\subset[0,1]$ with $\mu^u_x(\sigma(T))>0$, then
 $$
     h(f,\mu) \le \liminf_{n\to\infty} \frac1n\log\#\cR_n.
 $$ 
\end{proposition}

\begin{exercise}
Prove the above proposition using Corollary~\ref{cor-entropy-from-unstable} and Proposition~\ref{prop-param-to-ball}.
\end{exercise}

\subsection{Iterative construction}
Families of admissible reparametrizations over a relevant set will be built inductively using the following simple fact:

\begin{lemma}[Concatenation of reparametrizations]\label{lem-concat}
Given $S\subset[0,1]$, suppose that for two positive integers $n<n'$, we are given:
 \begin{itemize}
  \item[--] a family $\cR$ of $(C^r,N,\eps,\heps)$-reparametrizations over $S$ up to time $n$;
  \item[--] for each $\psi\in\cR$, a family $\cR_\psi$ of $(C^r,N,\eps,\heps)$-repa\-ra\-me\-tri\-za\-tions over $\psi^{-1}(S)$ up to time $n'$.
 \end{itemize}
Then $\cR':=\bigcup_{\psi\in\cR} \{\psi\circ\phi:\phi\in\cR_\psi\}$ is a family of $(C^r,N,\eps,\heps)$-reparametrizations over $S$ up to time $n+n'$.
\end{lemma}

\begin{exercise}
Check the above using that our reparametrizations are affine contracting maps.
\end{exercise}

\section{Fundamental results of Yomdin theory}

After giving the basic definitions and notations, we state the core estimate of Yomdin theory (Theorem~\ref{thm-Yomdin}) for the image of a parametrized curve $\sigma$ by a smooth map $f$ and apply it to the images of $\sigma$ by iterates $f^n$, $n\ge1$ (Corollary~\ref{cor-Yomdin}).

\subsection{Local image of a small curve}
Yomdin theory shows that, for smooth dynamics, coverings by admissible reparametrizations and coverings by dynamical balls have comparable cardinalities. Thus we will be able to consider smooth curves and use the tangent dynamics in our entropy estimates.

Yomdin fundamental theorem considers the image by a given map $f$  of a small curve $\sigma$ and shows that, in restriction to a small ball, this image is included in a moderate numbers of small curves. This can then be iterated, see below.

\begin{theorem}[Yomdin]\label{thm-Yomdin}
Given an integer $r\ge2$ and a $C^r$ map $f:M\to M$, there are an integer $Y\ge1$  depending only on $r$, and a number $\eps_Y>0$ with the following property. Given
 \begin{itemize}
  \item[--] numbers $0<\eps,\heps<\eps_Y$;
  \item[--] a regular $C^r$ curve with $C^r$ size at most $(\eps,\heps)$;
  \item[--] $T:=\{t\in[0,1]:\hf(\hsigma(t))\in B(\hf(\hx),\eps,\heps)\}$ for some $\hx\in\hsigma([0,1])$,
 \end{itemize}
there exists a family $\cR$ of reparametrizations of $\sigma$ over $T$ such that:
 \begin{enumerate}
  \item for each $\psi\in\cR$, $f\circ\sigma\circ\psi$ is a regular $C^r$ curve with $C^r$ size at most $(\eps,\heps)$;
  \item $\#\cR \le Y \cdot \| D\hf\|_{\sup}^{1/(r-1)}$.
 \end{enumerate}
\end{theorem}

We refer to \cite[theorem 4.13]{BCS-2} for a proof. The statement there is a bit more uniform. Note that the map $f$ above does not need to be “small”; Yomdin's proof starts by zooming in to make all higher differentials small. The first differential cannot be changed by a zoom or a clever choice of chart. However, its effects will be controlled by restricting the analysis to the part of the image inside a small ball.

\medbreak

By applying inductively Yomdin's result one gets that, on an exponential scale, the cardinality of admissible reparametrizations is not much larger than that of dynamical balls:

\begin{corollary}\label{cor-Yomdin}
Given an integer $r\ge2$ and a $C^r$ map $f:M\to M$, there are an integer $Y\ge1$ depending only on $r$, and a number $\eps_Y>0$ with the following property. Given
 \begin{itemize}
  \item[--] numbers $0<\eps,\heps<\eps_Y$;
  \item[--] integers $n,N$ and $T\subset[0,1]$;
  \item[--] a regular $C^r$ curve $\sigma$ with $C^r$ size at most $(\eps,\heps)$,
 \end{itemize}
there exists a family $\cR$ of  reparametrizations over $T$ which are $(C^r,N,\eps,\heps)$-admissible up to time $n$ and with cardinality:
 $$
   \#\cR \leq Y^{\lceil n/N\rceil} \left( \|D\hf^N\|_{\sup}^{\lfloor n/N\rfloor}
     \| D\hf^{n-N\lfloor n/N\rfloor} \|_{\sup}\right)^{\frac1{r-1}} \cdot
     r_{\hf}(\hf(\hsigma(T)),\eps,\heps,n).
 $$
\end{corollary}

\begin{exercise}
Write $n=qN+p$ with $q:=\lfloor n/N\rfloor$ and $0\le p<N$.
Deduce the corollary from the theorem applied to the two maps $\hf^N$ and $\hf^{n-N\lfloor n/N\rfloor}$ using an induction on $q$.
\end{exercise}

\part{Nonlinear estimates}

\begin{center}
{\Large Introduction}
\end{center}
\bigbreak

The third and final part of these lectures is where everything comes together to prove our main result, Theorem~\ref{maintheorem-htop}. Recall that this theorem is an entropy bound for a sequence of measures in terms of the neutral decomposition of Chapter~\ref{chapter-neutral}. 
The proof combines two types of information:
 \begin{enumerate}
  \item the nonexpanding property of the differential along the unstable direction within neutral orbit segments given by Theorem~\ref{thm-neutral-decomposition-projective};
  \item the number of dynamical balls within the remaining, arbitrary orbit segments given by the topological entropy.
 \end{enumerate}

Converting the linear information (1) to an entropy bound is quite delicate. It is performed in Chapter~\ref{chap-reparam-neutral}, the  technical core  of our approach. It is this argument which requires the consideration of the projective dynamics, the study of unstable curves, and the use of double scales.

The information (2) is just the formula for the topological entropy, adapted to the projective bundle, see Section~\ref{secEntropyProj}. 

Chapter~\ref{chap-proof-main} concludes the proof of Theorem~\ref{maintheorem-htop} by applying Ledrappier-Young entropy formula (Theorem~\ref{thm-LY-u-entropy}) to an appropriate unstable curve. The estimates are obtained by subdividing the orbits of typical points into long neutral orbit segments and the rest, using the relevant estimate (1) or (2) above for each orbit segment. 

We stress that the various arguments mix dynamical balls and smooth reparame\-tri\-zations. The translation between these two languages is provided by Yomdin theory (especially Corollary~\ref{cor-Yomdin}).

Finally, Chapter~\ref{chap-beyond} outlines the extra difficulties entailed by the proof of the stronger Theorem~\ref{theoremWithoutErgodicity} of our original paper \cite{BCS-2} and how they are solved.

\chapter{Controlling neutral blocks}\label{chap-reparam-neutral}

This chapter is devoted to the main technical estimate of these lectures: Theorem~\ref{thm-reparam-neutral} which uses neutral blocks to obtain nontrivial entropy bounds, i.e., a bound on the number of admissible reparametrizations needed to cover the iterates of typical points inside an unstable curve and which belong to a given neutral orbit segment.

\medbreak

Let $f:M\to M$ be a $C^r$ diffeomorphism of a compact surface $M$. Recall that $(\hM,\hf,\hpi)$ is the projective extension. 

\section{Statement of the reparametrization lemma}

We consider $\hmu\in\Prob(\hf)$ (not necesarily ergodic).
We specialize Definition \ref{def-abstract-neutral-block} to segments of $\hf$-orbits:

\begin{definition}\label{def-neutral-orbit}
For $\alpha>0$, an $\alpha$-\new{neutral orbit segment} is $(\hx,\hf\hx,\dots,\hf^{\ell-1}\hx)$  where $\hx:=(x,E)\in\hM$, such that
 $$
   \forall 0\le k\le\ell \quad \| Df^k_x|E \| \le e^{\alpha k}.
 $$
\end{definition}

The following variants are convenient: 
 \begin{itemize}
  \item[-] An $\alpha$-\new{neutral block} at $\hx$ is an integer interval $\ii{n,n+\ell}$ such that $(\hf^n\hx,\dots,\hf^{\ell-1}\hx)$ is a $\alpha$-neutral orbit segment. 
  \item[-] An $\alpha,\ell$-\new{neutral vector} is $\hx\in\hM$ such that $\ii{0,\ell}$ is an $\alpha$-neutral block for $\hx$.
\end{itemize}  
Note that this is the same as Definition \ref{def-abstract-neutral-block} for the homeomorphism $\hf$ and the continuous function $\vf(x,E)=\log\|Df_x|E\|$.

\medbreak

We can now state the reparametrization lemma. The following will apply to a set of $(\alpha,\ell)$-neutral vectors tangent to an unstable manifold. We assume the following condition for $\mu:=\hpi_*(\hmu)$,
 \begin{equation}\label{eq-measure-hyp}
   \mu(\{x\in M:\lambda^+(x)\ne\lambda^-(x)\}) > 0.
 \end{equation}
It will be satisfied in our application as $h(f,\mu)>0$.

\begin{theorem}\label{thm-reparam-neutral}
Let $f$ be a $C^r$ surface diffeomorphism ($r\ge2$).
There are functions $\gamma_0:(0,\infty)\to(0,\infty)$ and $N_0:(0,\infty)^2\to\NN^*$ with the following properties. Given $\delta,\eta>0$, $0<\gamma\le\gamma_0(\eta)$, and $N\ge N_0(\eta,\gamma)$, for any $\hmu\in\Prob(\hf)$ satisfying \eqref{eq-measure-hyp}, there are:
 \begin{enumerate}
  \item[(a)] numbers $0<\eps,\heps\le\delta$;
  \item[(b)] an open set $\hU_0\subset\hM$ with $\hmu(\hU_0)>1-\gamma^2$ and $\hmu(\partial\hU_0)=0$;
  \item[(c)] a positive integer $n_0\ge1$;
\end{enumerate} 
such that, for any regular curve with $C^r$ size at most $(\eps,\heps)$ and any $n\ge n_0$, the following conditions hold:
 \begin{enumerate}
  \item there is a family $\cR_n$ of reparametrizations of $\sigma$ over $T$ defined as:
   \begin{equation}\label{eq-T-neutral}
      \begin{aligned}
       \hsigma^{-1}\biggl\{ \hx \mid (\hx,\hf\hx,\dots,\hf^{n-1}\hx)& \text{ is $\frac{\eta}{10}$-neutral and }\\
       & \#\{0\le k<n:\hf^k\hx\notin\hU_0\}<\gamma\cdot n\biggr\};
   \end{aligned}\end{equation}
  \item $\cR_n$ is $(C^r,N,\eps,\heps)$-admissible up to time $n$;
  \item $\#\cR_n \leq \exp\left(\frac{\lambda(\hf)}{r-1}+\eta\right)$.
 \end{enumerate}
\end{theorem}

The rest of this chapter is devoted to the proof of this theorem. 

\section{Strategy of proof}

By Definition~\ref{def-neutral-orbit}, the iterates of a neutral tangent vector $\hx\in\hM$ stay small throughout the whole of the neutral block (for instance the iterates first become very small then grow back to unit size). The same will be true of the length of the iterates of a small unstable curve \emph{as long as the whole curve stays small and approximately straight}. This will give the required bound on the diameter of the iterates of the curve $\sigma$ on the surface $M$. 

However, this argument is incomplete: as we noted, it requires that the iterates of the curve stay approximately straight. We would like to see that $\sigma$ having a  small length implies that the diameter of $\hsigma$ does not become large. However, this implication does not hold in general, but because the iterates of $\sigma$ are unstable curves, we can make the following argument.

Since the map is not uniformly hyperbolic, it is not true that an unstable curve with small length is always approximately straight. However we are considering a measure $\nu_k$ which is weakly close to some limit measure $\mu$. This limit measure can be decomposed into a zero entropy part and a hyperbolic part in the sense of Pesin. By Lusin theorem the hyperbolic part will contain a uniform subset with large measure on which small length will imply small diameter in $\hM$.
By the general theory, zero entropy part will be controlled by a small number of dynamical balls.

This leads to using two scales $0<\eps\ll\heps$; their separation comes from the (arbitrarily bad) modulus of continuity of the Oseledets spaces produced by Lusin theorem. Here $\heps>0$ just needs to be small given the map $Df$, whereas $\eps>0$ is chosen so small that for some large $N\ge1$, with high probability with respect to the hyperbolic part of $\mu$,
  $$
     d(x,y)<\eps \implies d(E^+_{f^Nx},E^+_{f^Nx})<\heps.
  $$
\medbreak

Finally, the necessity to mix the control by dynamical balls (for the zero entropy part) and with the control by the tangent dynamics (for the hyperbolic part) will require Yomdin theory and high smoothness. 

\medbreak

We implement this strategy as follows:
 \begin{itemize}
  \item[--] we split the limiting measure $\hmu\in\Prob(\hf)$ into a hyperbolic and a non-hyperbolic part;
  \item[--] we prove by induction that, as long as the diameter of the lifted curve $\hf^k(\hsigma)$ stays small, the length of its trace on $M$ is given by the derivative at any point (Lemma~\ref{lem-ctrl-len});
  \item[--] we show that near the hyperbolic part, a very small diameter of $\sigma$ on the surface $M$ implies that the $\hM$-diameter of the image $\hf(\hsigma)$ stays small (Lemma~\ref{lem-W-neigh});
  \item[--] we observe that the non-hyperbolic part of $\hmu$ has only zero Lyapunov exponents and therefore zero entropy. It follows that the corresponding orbit segments can be covered by a small number of dynamical balls (Lemma~\ref{lem-zero-entropy});
  \item[--] we decompose the orbit into segments controlled by (1) the hyperbolic part, (2) the zero entropy part, (3) the rest;
  \item[--]  the number of such decomposition is not too large (Lemma~\ref{lem-type}) and therefore one can deal with typical orbits having a given decomposition;
  \item[--] we split by induction a parametrization of a unstable curve, restricting to typical points and keeping the $C^r$ size small. 
 \end{itemize}
\medbreak

\section{List of parameters}

Let $\hmu\in\Prob(\hf)$. We will need to fix a number of parameters $\eta,\delta,\gamma,N,\eps,\heps,n_*$, open subsets $\hU_\#$ and $\hW_0$, and pay attention to their mutual dependencies to avoid any contradiction. 

We proceed as follows (we omit the dependence on $f$):
 \begin{itemize}
  \item[--] we fix $\eta,\delta>0$ arbitrarily small;
  \item[--] we let $Y,\eps_Y\ge1$ be given by Yomdin Corollary~\ref{cor-Yomdin} (they only depend on $r$ and $f$);
 \item[--]  we fix $\gamma_0:=\gamma_0(\eta)>0$ so small that
   $$
     \exp(\eta/10\gamma_0)> \max(3Y,\|D\hf\|_{\sup})\text{ and } H(4\gamma_0)<\eta/10\;;
   $$
 \item[--] we choose $0<\gamma\le\gamma_0$;
 \item[--] we fix $N_0:=N_0(\eta,\gamma)>\max(1/\gamma,10/\eta)$ such that
  $$
     \forall N\ge N_0\quad \frac1N\log\|D\hf^N\|_{\sup}<\lambda(\hf) +\eta/10\;;
  $$
 \item[--] finally, we pick some $N\ge N_0$.
\end{itemize}
From now on, these parameters are fixed. In Lemma \ref{lem-W-neigh} we will fix
 \begin{itemize}
  \item[--] a small  $0<\heps<\delta$, depending on $N,\eta,\delta, \gamma$;
  \item[--] a small $0<\eps<\delta$, depending on $\heps,N,\eta,\delta,\gamma$;
  \item[--] an open set $\hU_0$ with $\hmu_\#(\hU_0)>1-\gamma^2$ and $\hmu(\partial\hU_0)=0$ depending on $\eps,\heps,N,\eta,\delta, \gamma$, on which the stable and unstable section will control the lift.
  \end{itemize}
  Finally we will also fix
\begin{itemize}
  \item[--]  a large multiple  $n_*\ge1$ of $N$ and an open set $\hW_0$ with $\hmu_0(\hW_0)>1-\gamma^2$ and $\hmu(\partial\hW_0)=0$, with small complexity, see \eqref{eq-hW0} below.
 \end{itemize}

\section{Length from derivative at a single point}

The images of a $C^r$-small regular curve $\sigma$ can be controlled by the $Df^N|_{T_x\sigma}$ at a single point $x\in\sigma$. We denote the length of a curve $\sigma:[0,1]\to M$ by:
 $$
    |\sigma| := \int_0^1 \|\sigma'(t)\|\, dt.
 $$ 

\begin{lemma}[Control of length]\label{lem-ctrl-len}
Given $f\in\Diff^r(M)$ ($1<r<\infty$), $\eta>0$ and $N\ge1$, there is a number $\eps_*:=\eps_*(\eta,N)>0$ satisfying:

For any $0<\eps,\heps\le\eps_*$, any $n\ge1$, if $\sigma$ is a regular curve with a reparametrization $\psi$ which is $(C^r,N,\eps,\heps)$-admissible up to time $n$, then
 $$
    \forall(x,E)\in\hsigma([0,1])\quad
    |f^n\circ\sigma\circ\psi| < e^{\eta n/10} \| Df^n_x|E \|\cdot  |\sigma\circ\psi|.
 $$
\end{lemma}

\begin{remark}
In order to control the length of a curve along a whole neutral block, we need to see that the contraction of a tangent vector implies a \emph{similar contraction} of the curve as long as the iterates of this curve stay small. One may think of the case when a neutral block comprises many contracting iterates followed by expanding ones: just establishing non-growth during the contracting iterates would not allow the necessary control of the rest of the neutral block.
\end{remark}

\begin{proof}
Since $\hM$ is compact, $D\hf$ is uniformly compact: there is $\eps_*>0$ such that
 $$
     \forall\hx,\hy\in\hM\; d(\hx,\hy)<\eps_*\implies\forall 0\le k\le N\;
         \log\|Df|\hf^k(\hx)\| < \log\|Df|\hf^k(\hy)\| + \frac{\eta}{10}.
  $$
Let $0<\eps,\heps\le\eps_*$ and let $\psi$ be a $(\eps,\heps,N)$-admissible reparametrization up to time $n$ of some regular $C^r$ curve $\sigma$. Let $n_0=0<n_1<\dots<n_\ell=n$ be the admissible times. Let $\hx=(x,E)\in\hsigma([0,1])$. We have, for each $0\le i<\ell$,
 $$
    \forall\hy\in \hf^{n_i}\circ\hsigma\circ\psi([0,1]) \quad
    \|Df^{n_{i+1}-n_i}|\hf^{n_i}(\hy)\| < e^{(n_{i+1}-n_i)\eta/10} \|Df^{n_{i+1}-n_i}|\hf^{n_i}(\hx)\|.
  $$
The claimed inequality follows by an induction on $i$.
\end{proof}

\section{Hyperbolic structures}

In this section, we exploit the hyperbolic part $\hmu_\#$ of the measure $\hmu$ using measurability of the stable and unstable sections.

\subsection{Hyperbolic decomposition}
We split the measure $\hmu$ into a hyperperbolic part $\hmu_\#$ and a non-hyperbolic part $\hmu_0$.

Recall the sets $M_\reg\subset M$ of Oseledets regular points and the subset $M_\#$ of those with simple Lyapunov spectrum
 $$
   M_\# := \{ x\in M \mid \lambda^+(x)>\lambda^-(x) \}.
 $$
We set $\hM_\#:=\hpi^{-1}(M_\#)$ and define the decomposition
 $$
   \hmu = a\hmu_\# + (1-a)\hmu_0\;,
 $$
where $a:=\hmu(\hM_\#)$, $\hmu_\#:=\hmu|_{\hM_\#}$, and $\hmu_0:=\hmu|_{\hM\setminus \hM_\#}$ (recall that $\hmu|_X:=\hmu(\cdot\cap X)/\hmu(X)$ if $\hmu(X)>0$ and $\hmu|_X=\hmu$ otherwise). By assumption \eqref{eq-measure-hyp}, $a>0$.

\smallbreak

We also denote by $\mu,\mu_\#,\mu_0$ the projections under $\hpi$ of $\hmu,\hmu_\#,\hmu_0$.

\smallbreak

Since the ergodic decomposition of $\mu$  contains no sink or source, there cannot be a set of positive measure on which the two exponents (repeated according to mulitiplicities) are both positive or both negative (see Lemma~\ref{lem-equal-zero}). In particular $\lambda^+(x)=\lambda^-(x)=0$ for a.e. $x\notin M_\#$.

\smallbreak

Note that, even though we call $\hmu_\#$ the hyperbolic part, we only require the existence a.e. of two distinct exponents, we do not exclude that one of them vanishes.

\subsection{Stable and unstable sections}

By Oseledets Theorem~\ref{thm-Oseledets}, there are two measurable (partial) sections $M_\#\to\hM$, $x\mapsto(x,E^+_x)$ and $x\mapsto(x,E^-_x)$. By Lusin theorem, the functions $x\mapsto E^*_{f^kx}$, $*=\pm$ and $1\le k\le N$, are all continuous on some compact set $K_\#\subset M_\#$ with $\mu_\#(K_\#)>1-\gamma^2$. Let 
 $$
    \hK_+:=\{(x,E^+_x):x\in K_\#\},\;
    \hK_-:=\{(x,E^-_x):x\in K_\#\}, \text{ and }
    \hK_\pm:=\hK_-\sqcup\hK_+.
 $$
These are compact subsets of $\hM$. By Proposition~\ref{prop-lift-projective}, $\hmu_\#(\hK_\pm)=1$.

\subsection{First $N$ iterates starting near $\hK_\pm$}

The set $\hK_\pm$ is the disjoint union of the graphs of two continuous functions. Hence, near $\hK_\pm$, a tight control for the distance on the surface (up to a very small $\eps$) will give some (less precise) control on the projective lift (up to a small $\heps$), as stated in the lemma below:

\begin{lemma}\label{lem-W-neigh}
Given $N\ge1$, $\delta>0$, let $K_\#,\hK_+,\hK_-$ be as above. Set
 $$
   \heps:=\frac{1}{10}\min(\delta,\eps_Y,\eps_*,d(\hK_+,\hK_-))>0,
 $$
where $\eps_Y$ and $\eps_*$ are given by Corollary~\ref{cor-Yomdin} and Lemma~\ref{lem-ctrl-len}.

Then, there is $0<\eps(N,\heps,K_\#)<\delta$ such that, for any small enough  open neighborhood $\hW_\#$ of $\hK_\pm$ with $\hmu(\partial W_\#)=0$, we have, 
 $$
    \forall 0\le k\le N\quad \diam(\hf^k(\hsigma\cap \hW_\#)) < \heps
 $$
where $\sigma$ is any regular curve $\sigma$ with $C^r$ size at most $(\eps,\heps)$.
\end{lemma}

\begin{proof}
By the continuity of $x\mapsto E^\pm_x$ over $K_\#$, there is $0<\eps<\heps$ such that, for all $x,y\in K_\#$ with $d(x,y)\le\eps$, letting $\hx^+:=(x,E^+_x)$ and $\hx^-:=(x,E^-_x)$, one has
 $$
    \forall 0\le k\le N\quad
       d(\hf^k(\hx^+),\hf^k(\hy^+))<\heps/2 \text{ and }
       d(\hf^k(\hx^-),\hf^k(\hy^-))<\heps/2.
  $$
Since $\heps<d(\hK_+,\hK_-)$, for all $\hx,\hy\in\hK_\pm$ with $d(\hx,\hy)\le\heps$, it must be that
 $$
    (\hx=\hx^+ \text{ and } \hy=\hy^+) \text{ or }
    (\hx=\hx^- \text{ and } \hy=\hy^-).
 $$ 
Hence:
 $$
   \left.\begin{array}{c}
        \hx,\hy\in\hK_\pm\\
        d(\hx,\hy)\le\heps \text{ and }
        d(x,y)\le\eps 
        \end{array}\right\} \implies
    \forall 0\le k\le N\quad f(\hf^k\hx,\hf^k\hy)<\heps.
 $$
Using the compactness of $\hM$, we find $\delta_{*}>0$ such that
 $$
   \left.\begin{array}{c}
        \hx,\hy\in B(\hK_\pm,\delta_{*})\\
        d(\hx,\hy)\le\heps \text{ and }
        d(x,y)\le\eps 
        \end{array}\right\} \implies
    \forall 0\le k\le N\quad f(\hf^k\hx,\hf^k\hy)<\heps,
 $$
where $B(E,r):=\{x:d(E,x)<r\}$. We conclude by letting $\hW_\#:=B(\hK_\pm,\delta_{**})$ for some $0<\delta_{**}\le\delta_*$ such that $\hmu(\partial B(\hK,\delta_{**}))=0$.
\end{proof}

\section{Control of the first iterates near $\hmu_0$}

Since all Lyapunov exponents vanish $\mu_0$-a.e., Ruelle-Margulis inequality implies that the entropy of $\mu_0$ vanishes. Since $\hpi:\hM\to M$ preserves entropy, this is also the case for $\hmu_0$. The following is folklore (but the proof is instructive):

\begin{lemma}\label{lem-zero-entropy}
If $\hnu\in\Prob(\hf)$ satisfies $h(\hf,\hnu)=0$, then for any $\overline{\gamma},\overline{\epsilon},\heps>0$ and large enough integer $n\gg 1$, there is a compact set $\hM_0$ with $\hnu(\hM_0)>1-\overline{\gamma}$ such that
 $$
     \frac1{n}\log r_f(\hM_0,\overline{\eps},\heps/2,n) < \frac{\eta}{10}.
 $$
\end{lemma}

\begin{proof}
The above is obvious from Bowen's formula if $\hnu$ is ergodic. The general case requires a bit of extra work.
Let $P$ a finite partition of $\hM$ into measurable sets,  each contained in doubly scaled balls of size $(\overline{\eps}/2,\heps/4)$.

The Shannon-McMillan-Breiman theorem and more precisely its non-ergodic version (Theorem~\ref{thm-SMB}), shows that, for $\hnu$-a.e. $x\in\hM$, we have
 $$
     \lim_{n\to\infty}  -\frac1n\log \hnu(P^n(x)) = h(x)\;,
 $$
for some $\hf$-invariant, integrable nonnegative function such that $\int h\,d\mu=0$. Obviously $h=0$ a.e., so this implies that there are a set $\hM_0\subset \hM$ and an integer $N_0$ such that $\nu(\hM_0)>1-\overline{\gamma}$ and for all $n\ge N_0$, the following holds:
 $$
     \forall x\in\hM_0 \quad  \hnu(P^n(x)) > e^{-n\eta/10}.
 $$
Hence, the set of elements of $P^n$ that intersect $\hM_0$ has cardinality less than $e^{n\eta/10}$ and joint measure at least that of $\hM_0$. But each element in this set is contained in an $(\overline{\eps},\heps/2,n)$-dynamical ball, so the minimal number of balls covering $\hM_0$ is less than $e^{n\eta/10}$.
\end{proof}

This lemma applied to $\hmu_0$ gives a large multiple integer $n$ of $N$ (in particular $n_*\ge2N$)  and a compact set $\hK_0$ such that 
 $$
    \hmu_0(\hK_0)>1-\gamma^2 \text{ and } \frac1{n}\log r_{\hf}(\hK_0,n,\eps/3,\heps/3) < \frac{\eta}{10}.
 $$
\newcommand\hC{\widehat{C}}
Denote by $\hC_0$ the cover implicit in the definition of $r_{\hf}(\hK_0,n,\eps/3,\heps/3)$. Pick $\eps/3<\rho\le\eps/2$ and $\heps/3<\widehat{\rho}\le\heps/2$ such that, 
 $$
   \hW_0:=\bigcup_{\hx\in\hC_0} B_{\hf}(\hx,n,\rho,\widehat{\rho}) \text{ satisfies } \hmu(\partial\hW_0)=0.
 $$
We set $n_*:=n$ and note, in conclusion, that the open set $\hW_0\subset\hM$ satisfies:
 \begin{equation}\label{eq-hW0}
    \hmu_0(\hW_0)>1-\gamma^2,\; \hmu(\partial\hW_0)=0,\;   \text{ and }
    \frac{1}{n_*}\log r_{\hf}(\hW_0,n_*,\eps/2,\heps/2) < \frac{\eta}{10}.
   \end{equation}

\section{Construction of the reparametrization families}

This section will conclude the proof of the main Theorem~\ref{thm-reparam-neutral} of this chapter. 

\smallbreak

For convenience, we start by summarizing what we have achieved so far. We were given $\delta,\eta>0$, $\gamma>0$ small enough, $N\ge1$ large enough, and $\hmu\in\Prob(\hf)$ satisfying the condition \eqref{eq-measure-hyp}. We have already determined numbers $0<\eps,\heps\le\delta$ as in Claim (a) of the theorem,  a large multiple $n_*$ of $N$ as in \eqref{eq-hW0}. We have also obtained open subsets $\hW_\pm,\hW_0\subset\hM$ with large measure on which  hyperbolicity or zero entropy gave some control. 

We set:
 $$
   \hU_0:=\hW_\#\cup\hW_0.
 $$
Note that $\hmu(\hU_0)\ge a\cdot\hmu_\#(\hW_\#)+(1-a)\cdot\hmu_0(\hW_0)>1-\gamma^2$ and $\hmu(\partial\hU_0)\le\hmu(\partial\hW_\#)+\hmu(\partial\hW_0)=0$, i.e., the set $\hU_0$ satisfies the Claim (b) of the theorem.

We now consider a regular curve $\sigma$ with $C^r$ size at most $(\eps,\heps)$ and some large integer $n\ge n_0$ ($n_0$ to be determined). To prove the theorem, we need to find an admissible reparametrization over the part defined in \eqref{eq-T-neutral}, that is,  tangent vectors which are neutral over the integer interval $\ii{0,n}$ and are $\hmu$-typical in the sense that they belong to some $\hmu$-large open set $\hU_0:=\hW_\#\cup\hW_0$ except for a small fraction of the integer interval.

\subsection{Partitionning orbits}
For each tangent vector $\hx$, we decompose $\ii{0,n}$ into segments starting by a visit to $\hW_\#$ (so the derivative controls the expansion) or a visit to $\hW_0$ (so the entropy vanishes). Recall that we are considering a \new{good orbit segment}, that is, $(\hx,\dots,\hf^{n-1}\hx)$  such that
 \begin{equation}\label{eq-good-orbit}
    \frac1n\#\{0\le k<n:\hf^k\hx\in\hU_0\}>1-\gamma.
  \end{equation}

\begin{lemma}\label{lem-type}
Any good orbit segment with length $n>\gamma^{-1}n_*$ can be decomposed into disjoint orbit sub-segments of the following three kinds:
 \begin{enumerate}
   \item of length $N$ with initial point in $\hW_\#$;
   \item of length $n_*$ with initial point in $\hW_0$;
   \item of length $1$.
 \end{enumerate}
such that the orbit segments of kind (3) altogether contain at most $2\gamma n$ elements.
\end{lemma}

The \emph{kind} of each sub-segment defined above will determine how it is controlled. Below we will define the \emph{type} of a segment which will describe how the segment is split into subsegments and the kinds of each.

\begin{proof}
We define inductively times $n_0=0<\dots<n_\ell=n$, setting $n_0:=0$ and then
  \begin{enumerate}
   \item $n_{i+1}:=n_i+N$ if $\hf^{n_i}\hx\in\hW_\#$ and $n_i+N\le n$;
   \item $n_{i+1}:=n_i+n_*$ if $\hf^{n_i}\hx\in\hW_0\setminus\hW_\#$ and $n_i+n_*\le n$;
   \item $n_{i+1}:=n_i+1$ otherwise.
 \end{enumerate}
Note that this defines a decomposition into the three announced kinds.
We stop when $i+1=\ell$. 
Case (3) only occurs outside $\hU_0:=\hW_0\cup\hW_\#$ or within the integer interval $\ii{n-n_*,n}$. By \eqref{eq-good-orbit}, the first possibility occurs less than $\gamma\cdot n$ times. Thus, case (3) occurs at most $\gamma\cdot n+n_*<2\gamma\cdot n$ times.
\end{proof}

The finite sequence of splitting spots in Lemma~\ref{lem-type}, $(n_i)_{i=0}^\ell$ (in the notations of the above proof) is called the \new{type of the orbit segment}. Since the three lengths $1<N<n_*$ are pairwise distinct, the type uniquely defines which kind (1), (2), or (3) occur and at which positions.
We define
 $$
    \Theta(n) := \{(n_i)_{i=0}^\ell:\forall 0\le i<\ell\; n_{i+1}-n_i\in\{N,n_*,1\},\; n_0=0, \text{ and }n_\ell=n\}.
 $$
Since the sub-segments $\ii{n_i,n_{i+1}}$ are either very long or few (depending on their kind), an easy combinatorial estimate yields:

\begin{lemma}\label{lem-fewtypes}
There is some $N_1\ge1$ such that, for all integers $n\ge N_1$, the number of possible types of good orbit segments of length $n$ is bounded by $\exp (H(4\gamma)n)$.
\end{lemma}

The curious reader may take a look to the proof of the similar Lemma~\ref{lem-combi}.

\subsection{Inductive scheme}
Combining the lower bounds $n_*$ from Lemmas \ref{lem-type}~and~\ref{lem-fewtypes}, we set:
 $$
    \overline{n}_0 := \lceil\max(n_*/\gamma,N_1)\rceil.
 $$
 
 Let $\sigma$ be a regular $C^r$ curve with $C^r$ size at mot $(\eps,\heps)$ and choose $n\ge \overline{n}_0$.  To prove the theorem we need to build a reasonably sized family of admissible reparametrizations up to time $n$ over the set $T\subset[0,1]$ defined in \eqref{eq-T-neutral}. 
 
By lemma~\ref{lem-type}, orbit segments of length $n$ can be split according to their types $\theta=(n_0,\dots,n_\ell)\in\Theta(n)$. Lemma~\ref{lem-fewtypes} and the choice of $\gamma$ give:
 $$
   \forall n\ge N_1\; \#\Theta(n)\le \exp[H(4\gamma)n] < e^{(\eta/10)n}.
 $$
We split the parameters from \eqref{eq-T-neutral} accordingly:  $T=\bigsqcup_{\theta\in \Theta(n)} T_\theta$, where $T_\theta$ is the subset of $t\in T$ with type $\theta$.

We fix some $\theta\in\Theta(n)$. We are going to build, by induction on $i=0,\dots,\ell$, families $\cR^\theta_i$ of $(C^r,\eps,\heps,N)$-reparametrization over $T_\theta$ up to time $n_i$ satisfying
 \begin{equation}\label{eq-reparam-block}
    \# \cR^\theta_i \leq \#\cR^\theta_{i-1} \times \alter{
      \exp\left(\left(\frac{\lambda(\hf)}{r-1}+\frac{7\eta}{10}\right)\Delta_i\right) & \text{ if }\Delta_i:=n_i-n_{i-1}=N \text{ or }n_*,\\
      \exp(\eta/\gamma) & \text{ if }\Delta_i = 1
      }
  \end{equation}
and with lengths of their iterates bounded as follows
 \begin{equation}\label{eq-reparam-contract}
   \forall 0\le i\le\ell\quad |f^{n_i}\circ\sigma\circ\psi| < \eps e^{-\frac{\eta}{10}n_i} \|Df^{n_i}|\hx\|,
  \end{equation}
for any $\hx\in\hsigma\circ\psi([0,1])$.

\medbreak

We say ``admissible reparametrization'' or ``$(C^r,\eps,\heps,N)$-reparametrization'' to mean ``$(C^r,\eps,\heps,N)$-admissi\-ble re\-pa\-ra\-me\-trization''.

\subsection{Inductive step}
Let  $1\le i\le\ell$.
Assume that there is some family $\cR_{i-1}$ of admissible reparametrizations over $T_\theta$ up to time $n_{i-1}$ satisfying \eqref{eq-reparam-block}-\eqref{eq-reparam-contract}. We want to refine this family to make it admissible up to time $n_i$. We are going to apply the concatenation lemma~\ref{lem-concat} to each $\psi\in\cR_{i-1}$ and a (yet to be determined) family $\cR_\psi$ of admissible reparametrizations of the curve $\sigma':=f^{n_i}\circ\sigma\circ\psi$ over $T':=\psi^{-1}(T_\theta)$. The construction of each family will depend on the kind of the subsegment $\ii{n_{i-1},n_i}$.

\medbreak

Let us check that in each of the three cases, $\hf^N\circ \hsigma'(T')$ is contained in the union of a controlled number of doubly scaled $(\eps,\heps)$-balls.

\medbreak

\subsubsection*{Case $n_i-n_{i-1}=N$} 
In this case, $\hf^{n_{i-1}}(\sigma'(T'))\subset\hW_\#$: the subsegment is of kind (1). Using \eqref{eq-reparam-contract} for $i-1$ and Lemma~\ref{lem-ctrl-len} applied to $f^N$, $\sigma'$, and $\psi$, we obtain, for any $\hx'\in\hsigma'([0,1])$,
 $$\begin{aligned}
   |f^N\circ\sigma'\circ\psi| &< e^{\eta N/10} \| Df^N (\hx') \| \cdot |\sigma'\circ\psi| = e^{\eta N/10} \| Df^N (\hx') \| \cdot |f^n\circ\sigma\circ\psi| \\ 
     &< \eps \cdot e^{\eta N/10} e^{-\frac{\eta}{10}n_{i-1}} \cdot \|Df^{n_{i-1}}(\hx)\|\cdot \| Df^N (\hx') \| \\
     & = \eps \cdot e^{\eta N/5} e^{-\frac{\eta}{10}n_i} \cdot \|Df^{n_{i}}(\hx)\|\;,
 \end{aligned}$$
where $\hx:=\hf^{-n_{i-1}}(\hx')$. We have used the equality $\| Df^N (\hx')\|\cdot \|Df^{n_{i-1}}(\hx)\| = \|Df^{n_{i}}(\hx)\|$.

As usually in Yomdin theory, we cancel the factor $e^{\eta N/5}$ by an \new{affine subdivision}: introducing $\lceil e^{\eta N/5} \rceil\le e^{3\eta N/10}$ affine maps $\phi_j:[0,1]\to I_j$, each $e^{-\eta N/5}$ being contracting and such that the union of their images covers $T'$. For every such subdivision, selecting $\hx_j\in \psi(I_j\cap T_\theta)$, we get
 \begin{equation}\label{eq-contract2}
    |f^{n_i}\circ\sigma\circ\psi\circ\phi_j| < \eps e^{-\frac{\eta}{10}n_i} \cdot \|Df^{n_{i}}(\hx_j)\| \le \eps\;,
  \end{equation}
since $(\hx_j,\dots,\hf^{n-1}\hx_j)$ is a $\frac{\eta}{10}$-neutral orbit segment ($\hx_j\in \psi(T_\theta)$). In particular the projection on $M$ of $f^{n_i}\circ\sigma\circ\psi\circ\phi_j([0,1])$ is contained in an $\eps$-ball.

By the inductive assumption, the curve $\sigma'$ is regular with $C^r$-size at most $(\eps,\heps)$, thus Lemma~\ref{lem-W-neigh} gives
 $$
     \diam(\hf^N(\hsigma'(T'_\theta\cap I_j)))\le \diam(\hf^N(\hsigma'(T'_\theta\cap I_j) \cap \hW_\#)) < \heps.
 $$
In particular, $\hf^N\circ\hsigma'(T'\cap I_j)$ is contained in a doubly scaled $(\eps,\heps)$-ball, hence  $\hf^{n_i}\circ\hsigma\circ\psi(T'_\theta)$ is contained in $\lceil e^{\eta N/5}\rceil$ doubly scaled $(\eps,\heps)$-balls, and the claim is proved.

\medbreak

Now, Yomdin Theorem~\ref{thm-Yomdin} applies and yields a family $\cR_{\psi,j}$ of  reparametrizations $\phi$ of $\sigma'$ over $T'_\theta\cap I_j$ such that
 $$
     \#\cR_{\psi,j} \le Y\|D\hf^N\|_{\sup}^{1/(r-1)}\;,
 $$
and each curve $f^N\circ\sigma'\circ\phi$ has $C^r$ size at most $(\eps,\heps)$, so that each reparametrization $\phi$ is $(C^r,N,\eps,\heps)$-admissible.

\medbreak

The concatenation Lemma~\ref{lem-concat} gives a family $\cR_{i}$ of $(N,\eps,\heps)$-admissible repara\-me\-trizations up to time $n_i$ with cardinality
 $$\begin{aligned}
    \#\cR_i &\leq \#\cR_{i-1} \cdot Y \|D\hf^N\|_{\sup}^{1/(r-1)} \times e^{3\eta N/10} \\
     &\leq \#\cR_{i-1} \cdot \exp\left((\lambda(\hf)/(r-1)+7\eta/10)\Delta_i\right).
 \end{aligned}$$

\medbreak
 
Since the reparametrizations are contractions, \eqref{eq-reparam-contract} follows from \eqref{eq-contract2}. The inductive step is complete in this case.

\medbreak

\subsubsection*{Case $n_i-n_{i-1}=n_*$}
As before, fix one $\psi\in\cR_{i-1}$. In this case, $\hsigma'(T'_\theta)\subset\hW_0$, hence the subsegment is of kind (2) and we have
 $$
    r_{\hf}(\hsigma'(T'_\theta),\eps,\heps,n_*) \le e^{\eta n_*/10}.
 $$
We apply Corollary~\ref{cor-Yomdin} of Yomdin theory and obtain a family $\cR_\psi$ of admissible reparametrizations with
 $$\begin{aligned}
   \#\cR_\psi &\le Y^{\lceil n_*/N\rceil} \cdot \| D\hf^N\|_{\sup}^{\lfloor n_*/N\rfloor}/(r-1) \cdot \max_{0\le k<N} \|D\hf^k\|_{\sup}^{1/(r-1)} \\
   &\qquad\qquad\qquad\qquad \times r_{\hf}(\hsigma'(T'_\theta),\eps,\heps,n_*) \\
   &\le \max_{0\le k<N} \|D\hf^k\|_{\sup}^{1/(r-1)} \times \left(2\cdot Y^{1/N} \cdot e^{\lambda(\hf)/(r-1)} \cdot e^{\eta/5}\right)^{n_*} \\
   &\le e^{(\lambda(\hf)/(r-1) +4\eta/10) n_*}.
  \end{aligned}$$
 \medbreak

To ensure \eqref{eq-reparam-contract}, we proceed as in the previous case by applying Lemma~\ref{lem-ctrl-len} and using a bounded affine subdivision. This concludes the inductive step for this second case.

\medbreak
 
\subsubsection*{Case $n_i-n_{i-1}=1$}
The subsegment is of kind (3).
By our inductive assumption, we have the following control on lengths:
 $$
   \diam(\hsigma')<\heps \text{ and }|\sigma'|<\eps \cdot e^{-\eta n_{i-1}/10}\|D_xf^{n_{i-1}}|_E\|\;,
 $$
for any $(x,E)\in\sigma$. To bound the derivatives $D\hf^{n_{i-1}+1}$ and $Df^{n_{i-1}+1}$, we use an affine subdivision into $m:=\lfloor\|D\hf\|_{\sup}\rceil+\lfloor e^{\eta/10}\|Df\|_{\sup}\rceil$ subintervals $I_1,\dots,I_m$ so that
 \begin{itemize}
  \item[$\circ$] $|f\circ \sigma'(I_j)| < \eps \cdot e^{-(\eta/10)n_{i}} \|Df^{n_i}(\hy)\|$, for any $\hy\in\sigma$ with $\hf^{n_i}(\hy)\in\hsigma'(I_j)$;
  \item[$\circ$] $|\hf\circ\hsigma'(I_j)| < \heps$.
 \end{itemize}
We see that each image $\hf\circ\hsigma'(I_k)$ is contained in a doubly scaled $(\eps,\heps)$-ball. Thus we can apply Yomdin theorem \ref{thm-Yomdin} to obtain a family $\cR_j$ of $(C^r,1,\eps,\heps)$-reparame\-tri\-za\-tions of $\sigma'$ with cardinality at most $Y\|D\hf\|_{\sup}^{1/(r-1)}$. We  conclude as before with 
 $$
   \#\left(\bigcup_{j=1}^m \cR_j\right) \leq 3Ye^{\eta/10}\|D\hf\|_{\sup}^{1+1/(r-1)}
     \le e^{\eta/10\gamma}.
 $$

\subsection{Conclusion of the induction}
We start with the trivial reparametrization family: $\cR_0:=\{\Id\}$ which satisfies \eqref{eq-reparam-contract} since $\sigma$ has $C^r$ size at most $(\eps,\heps)$.
The induction yields for $i=\ell$ a family $\cR^\theta_n$ of admissible reparametrizations up to time $n$ such that
 $$\begin{aligned}
   \#\cR^\theta_n &\le \exp\left(\frac{\lambda(\hf)}{r-1}n +\frac{7\eta}{10}n + \frac\eta\gamma\cdot 2\gamma n\right) 
   \le \exp \left(\frac{\lambda(\hf)}{r-1}n +\frac{7\eta}{10}n +  2\eta n\right)\\
   &\le \exp \left(\frac{\lambda(\hf)}{r-1}n + \eta n\right)\,
  \end{aligned}$$
using the bound $2\gamma n$ on the number of orbit subsegments of kind (3) in any given type (see Lemma~\ref{lem-type}).

\medbreak

This proves Theorem~\ref{thm-reparam-neutral}.\qed

\chapter{Proof of the main theorem}\label{chap-proof-main}

This chapter is devoted to the proof of a slightly more precise form of Theorem~\ref{maintheorem-htop} (page \pageref{maintheorem-htop}) in terms of the neutral decomposition in the projective extension. 

\section{Statement}

We can now give a slightly more precise statement, using the projective extension.

\begin{theorem}\label{thm-weak-dec}
Let $f\in\Diff^\infty(M)$ with $M$  a compact surface. Let $\nu_k\in\Proberg(f)$ such that the following limits exist:
 \begin{enumerate}
     \item[(a)] $\lim_{n\to\infty} h(f,\nu_k)>0$;
     \item[(b)] $\hmu:=\lim_{n\to\infty} \hnu_k^+$ in the weak (star) topology of $\Prob(\hf)$;\footnote{Since $h(f,\nu_k)>0$ and $f$ is a surface diffeomorphism, $\nu_k$ is hyperbolic and $\hnu_k^+$ is well-defined.}
  \end{enumerate}
Then there are $\hmu_0,\hmu_1\in\Prob(\hf)$ and $0<\beta\le 1$ such that, writting $\mu_i:=\hpi_*(\hmu_i)$ for $i=0,1$, we have
 \begin{enumerate}
  \item $\hmu=(1-\beta)\hmu_0+\beta\cdot\hmu_1$;
  \item $\lim_{n\to\infty} \lambda^+(f,\nu_k) = \beta \cdot\lambda^+(f,\mu_1)$;
  \item $\lim_{n\to\infty} h(f,\nu_k) \le \beta \cdot h_\top(f)$.
  \end{enumerate}
Moreover, writing $\hvf(x,E):=\log\|Df_x|E\|$, we also have
 \begin{enumerate}
  \item[(4)] $\hmu_1(\hvf) = \lambda^+(\mu_1)$ and $\lambda^+(f,x)>0$ for $\mu_1$-a.e. $x\in M$;
  \item[(5)] if $\beta<1$, then $\hmu_0(\hvf)=0$.
  \end{enumerate}
\end{theorem}

\begin{exercise}
Deduce Theorem~\ref{maintheorem-htop} from the above theorem.
\end{exercise}

\begin{remark}
In comparison with the general theorem \ref{theoremWithoutErgodicity}, the main simplification above is in the key item (3) where the bound is $\beta\cdot h_\top(f)$ instead of the more natural (and powerful) $\beta\cdot h(f,\mu_1)$ obtained in \cite{BCS-2}.
\end{remark}

\section{Strategy and preparation}

Let $\nu_k\in\Proberg(f)$ with $\lim_k h(f,\nu_k)>0$ and let $\hmu$ be the weak limit of the well-defined unstable lifts $\hnu_k^+$. 

To begin with, note that the Ruelle-Margulis inequality $h(f,\nu_k)\le \lambda^+(f,\nu_k)$ implies $\lim_n \lambda^+(f,\nu_k)>0$, hence we can apply the neutral decomposition as in Theorem~\ref{thm-neutral-decomposition-projective}. We obtain
 $$
    \hmu = (1-\beta)\mu_0 + \beta \mu_1 \text{ where } 0\le\beta\le 1
 $$
together with the above items (1), (2), (4), and (5) except for the claim that $\beta\ne0$. But this one follows from the equality
$\lim_k \lambda^+(\nu_k)=\beta\lambda^+(\mu_1)$  and assumption (a). The rest of this chapter is devoted to the proof of item (3), the entropy bound.

To estimate the entropy of the ergodic measures $\nu_k$ for $k$ large, we use the formula of Ledrappier-Young from Proposition~\ref{prop-entropy-ureparam} which reduces this question to the number of admissible reparametrizations  (Definition~\ref{def-admissible-reparam}) needed to cover a subset of positive measure along a typical unstable manifold (the measure being given by the unstable disintegration).

We stress that relevant estimates (such as the choice of the double scale) must hold for \emph{all} large $k$. This uniformity will come from two sources: the vanishing of the tail entropy from Theorem~\ref{thm-tail-zero} (and more generally Yomdin theory in the $C^\infty$ setting) and the convergence of the measures  together with the neutral decomposition of the limit.

We will inductively iterate the diffeomorphism (or rather some fixed, large iterate, to absorb constants), dividing the reparametrizations as needed to keep them admissible.
Our main tools will be:
 \begin{itemize}
  \item[--] the key Theorem~\ref{thm-reparam-neutral} that needs very little splitting within a neutral block;
  \item[--] the cover of $\hM$ by dynamical balls given by the Bowen-Dinaburg entropy formula (or rather some variant adapted to our double scale);
  \item[--] Yomdin theory to convert control by dynamical balls into control by reparametrization at an arbitrarily small price since the map is $C^\infty$ smooth.
 \end{itemize}
More precisely, the neutral decomposition induces long neutral blocks in $\nu_k$-typical orbits, leaving out a proportion $\beta$. Since these blocks are long, their combinatorics does not change much the entropy and one can assume them to be fixed as usual. 
Then we can inductively refine the family of reparametrizations by applying either Theorem~\ref{thm-reparam-neutral} for each neutral block with almost no splitting, 
or a simpler version, Proposition~\ref{prop-reparam-htop} below, with a the splitting just controlled by the topological entropy. We get an entropy bound of the type $\beta\cdot h_\top(f)$, as announced.

\section{Preparations}

We note that for all large $k\ge1$, by the Ruelle-Margulis inequality, we have
 $$
   \nu_k^+(\hvf) = \lambda^+(\nu_k)>\frac12\lim_k h(f,\nu_k)>0.
 $$
Hence, we can apply Theorem~\ref{thm-neutral-decomposition-projective} so as to obtain
 $$
    \hmu = (1-\beta) \cdot \hmu_0 + \beta\cdot\hmu_1 \qquad (0<\beta\le 1,\; \hmu_0,\hmu_1\in\Prob(\hf))\;.
 $$
We need to show that $\lim_k h(f,\nu_k) \le\beta\cdot h_\top(f)$.
We proceed by contradiction. We can assume that there is $\beta'>\beta$ such that, after passing to a subsequence,
 $$
   \forall k\ge1\quad h(f,\nu_k) > \beta' \cdot h_\top(f).
 $$ 

We fix $\delta,\eta>0$ sufficiently small (in particular smaller than the number $\eps_Y$ from Yomdin theorem~\ref{thm-Yomdin}).
Given  $0<\gamma\ll\beta'-\beta$ small enough and $N$ large enough, Theorem~\ref{thm-reparam-neutral} defines the numbers $0<\eps,\heps\le\delta$, the open set $\hU_0$, and a positive integer $n_0$. 

\begin{remark}
A significant simplification in comparison with \cite{BCS-2} is that we can choose the parameters depending only on Theorem~\ref{thm-reparam-neutral}. This is in contrast to  \cite[p.829]{BCS-2} where, in Step 2 of the proof of Theorem D, one needs to simultaneously take into account the constraints coming from the two reparametrization statements, Proposition 5.1 for non neutral blocs and 5.2 for neutral blocks.
\end{remark}

\section{Reparametrizations for arbitrary orbit segments}

To treat the not necessarily neutral orbit segments (between two consecutive neutral orbit segments), we need some analogue of Proposition 5.1 in \cite{BCS-2}.  Since we only want a bound by the topological entropy, rather than by the entropy of a limiting measure ($\mu_1$ from the neutral decomposition), our estimate  is much simpler.
\medbreak

Recall the number $\eps_Y>0$ defined in Yomdin Theorem~\ref{thm-Yomdin}.
\medbreak

\begin{proposition}\label{prop-reparam-htop}
Let $f$ be a $C^r$ surface diffeomorphism (for some $r\ge 2$). Fix $\eta>0$. Then, for any numbers $0<\eps,\heps\le\eps_Y$, there is an integer $N_1\ge1$ such that the following holds for all $N\ge N_1$.

For any regular curve $\sigma$ with $C^r$ size at most $(\eps,\heps)$ and any integer $n\ge N_1$, there is a family $\cR_n$ of reparametrizations of $\sigma$ such that
 \begin{enumerate}
  \item $\cR_n$ is $(C^r,N,\eps,\heps)$-admissible up to time $n$;
  \item $\#\cR_n \leq \exp\left(h_\top(f)+\frac{\lambda(\hf)}{r-1}+\eta\right)$.
 \end{enumerate}
\end{proposition}

\begin{exercise}\label{exoLambda}
Let $g$ be a $C^1$ diffeomorphism of a compact manifold. Let $\lambda(g):=\limsup_{n\to\infty}\frac1n\log\|Dg^n\|_{\sup}$.
Show that for every $\gamma>0$, there is $m_0>0$ such that for all $m\ge m_0$,  $\|Dg^m\|_{\sup}^{1/m} \le  e^{\lambda(g)+\gamma}$.
\end{exercise}

\begin{proof}
Given $\eta>0$, we use Exercise~\ref{exoLambda} to get $N_0$ such that 
 $$
 \forall N\ge N_0 \quad \|D\hf^N\|^{1/N}\le \exp (\lambda(\hf)+\eta/4).
 $$

The Bowen formula for the topological entropy (Lemma~\ref{lem-h-double}) gives $m_1\ge1$ such that
 $$
   \forall n\ge 1\quad r_{\hf}(\hM,\eps,\heps,n) \leq e^{m_1\eta/4}\cdot e^{(h_\top(\hf)+\eta/4)n}.
 $$
We recall that $h_\top(\hf)=h_\top(f)$ by \eqref{eq-h-M-hm}.

Since $0<\eps,\heps<\eps_Y$ we can apply Corollary~\ref{cor-Yomdin} for any $n\ge1$. We obtain that any regular curve $\sigma$ with $C^r$ size $(\eps,\heps)$ admits families $(\mathcal R_n)_{n\ge1}$ of $(C^r,N,\eps,\heps)$-reparametrizations such that, for some $m_2\ge1$, for all $n\ge1$,
  $$\begin{aligned}
     \#\mathcal R_n &\le C_1 \|D\hf^N\|_{\sup}^{(n/N)/(r-1)} \cdot e^{m_1\eta/4} e^{(h_\top(f)+\eta/4)n}\\
     &\le C_1 \cdot e^{n\cdot \lambda(\hf)/(r-1)} \cdot e^{m_1\eta/4} e^{(h_\top(f)+\eta/4)n}\\
     &\le e^{m_2 \eta/2} \exp\left(n\cdot \left(\frac{\lambda(\hf)}{r-1} + h_\top(f) + \eta/2\right)\right).
  \end{aligned}$$
In particular, for all $n\ge m_2$, $\#\cR_n \leq \exp\left(h_\top(f)+\frac{\lambda(\hf)}{r-1}+\eta\right)$. It suffices to take $N_1:=\max(N_0,m_2)$.
\end{proof}

\section{Finding clean neutral blocks}

In view of Definition~\ref{def-neutral-dec} of a neutral decomposition,
we have the weak convergence (perhaps after passing to a subsequence):
 $$
    \hmu_0 = \lim_{(\alpha,L)\to(0,\infty)} m_0^{\alpha,L}\;.
 $$
In particular, there are  numbers $0<\alpha_0\le\eta/10$ and $L_0\gg n_0$ such that, for all $0<\alpha\le\alpha_0$ and $L\ge L_0$, the following holds:
 $$
    \forall k\ge k_0(\alpha,L)\quad 
       \hnu_k^+(\Neutral(\alpha,L)) = (1_{\Neutral(\alpha,L)}\cdot\hnu_k^+)(\hM)
          > m_0(\hM)-\gamma = 1-\beta-\gamma > 1-\beta'.
 $$
We fix $k_3\ge k_0$.

The ergodic theorem gives a set $\hM_3$ with $\hmu(\hM_3)>1/2$ and an integer $N_3$ such that
 $$
  \forall\hx\in\hM_3\; \forall n\ge N_3\quad  \#(\Neutral(\eta/10,L_0)\cap\ii{0,n}) \ge (1-\beta-\gamma)n.
  $$
Say that a maximal $(\alpha_0,L_0)$-neutral block $\ii{a,b}$ of $\hx$ is \emph{clean} if it does not contain too many iterations out of the set $\hU_0$ defined previously, more precisely,
 $$
    \#\{a\le j<b:\hf^j\hx\notin\hU_0\}< \gamma\cdot (b-a).
 $$
Notice that the union $D(\hx,n)$ of unclean maximal $(\alpha_0,L_0)$-neutral blocks in $\ii{0,n}$ satisfies
 $$
       \#D(\hx,n)\cdot\gamma \le \#\{0\le j<n: \hf^j\hx\notin\hU_0\} \le \gamma^2\cdot n.
 $$
Thus, for all $\hx\in\hM_3$ and $n\ge N_3$,  we have $\#D(\hx,n)\le\gamma\cdot n$. We have just proved

\begin{lemma}
Given $\hx\in\hM_3$ and $n\ge N_3$, the union of the clean $(\alpha_0,L_0)$-neutral blocks in $\ii{0,n}$ occupy a fraction at least $1-\beta-2\gamma$.
\end{lemma}

\section{Selecting a good unstable curve}

Recall that $k\ge k_3$ has been fixed together with a good set $\hM_3$ with $\hnu_k^+(\hM_3)>0$. From now on, we omit the index $k$. Because $\hnu^+$ is the image of $\hnu$ by the section $\hGamma^+$, we have
 $$
    \hnu^+(\hM_3) = \hnu^+(\hM_3\cap\hGamma^+) = \hnu(M_3), \text{ where } M_3:=\hpi(\hM_3\cap\hGamma^+).
 $$

Consider the unstable disintegration $(\nu^u_x)_x$ of $\nu$ (see Definition~\ref{def-unstable-disint}). By definition, we have
 $$
    \int_{\hM} \nu^u_x(M_3)\, d\nu(x) = \nu(M_3) > 0.
 $$
Recall that $\nu^u_x(W^u_\loc(x))=1$ for a.e. $x\in M$. Hence, there is a set of positive $\nu$-measure of points $x\in M$ such that $M_3\cap W^u_\loc(x)$ has positive $\nu^u_x$-measure. The Proposition~\ref{prop-entropy-ureparam} yields the following bound on the entropy of $\nu$:
 $$
   h(f,\nu) \le \liminf_{n\to\infty}\frac1n\log\#\cR_n\;,
 $$
if $\cR_n$ is a family of $(C^r,N,\eps,\heps)$-reparametrizations of $\sigma$ up to time $n$ over $T$, where
 \begin{itemize}
  \item[--] $\sigma$ is a regular curve with $C^r$ size at most $(\eps,\heps)$ such that $\nu^u_x(\sigma([0,1])\cap M_3)>0$;
  \item[--] $T=\sigma^{-1}(M_3)$.
 \end{itemize}
 
\bigbreak
 
The rest of the proof is quite similar to the construction of the reparametrizations inside a neutral block. Let us explain the main remaining points.

\section{Decomposition of orbits of typical points}

To each pair $x\in M_3$ and $n\ge n_3$ we associate the sequence $\Sigma(n,\hx):=(a_i,b_i)_{i=1}^m$ of successive maximal $(\alpha_0,L_0)$-neutral blocks in $\ii{0,n}$ for $\hx:=(x,E_x^+)$. 

Note that these neutral blocks correspond to the intersection with $\ii{0,n}$ of all the neutral blocks that are maximal in $\ZZ$ up to two possible exception: (i) an arbitrary but finite initial neutral block containing $0$ (independent of $n\ge1$); (ii) a terminal neutral block containing $n-1$ but intersecting $\ii{0,n}$ over a length smaller than $L_0$. The contribution of both neutral blocks can be safely ignored.

 Since each of neutral blocks has length at least $L_0$ and is contained in $\ii{0,n}$, their number is bounded by $n/L_0$. Moreover, the set of such sequences, realized by some orbit segment or not,
 $$
    \Sigma(n) := \{(a_i,b_i)_{i=1}^m: 0\le a_1\le b_1\le a_2\le\dots b_m\le n, \forall 1\le i\le m\; b_i-a_i\ge L_0, m\ge1\}\;,
 $$
has  cardinality bounded by a small exponential, as stated next.

\begin{lemma}\label{lem-combi}
Given $\eta>0$, for $L_0$ large enough, we have $\#\Sigma(n) \le e^{\eta n/10}$ for all $n\ge1$.
\end{lemma}

\begin{proof}
The integer intervals $\ii{a_i,b_i}$ are pairwise disjoint with lengths $\ge L_0$, hence their number $m$ satisfies $m\le p:=\lfloor n/L_0\rfloor$. Once $m$ is fixed, the number of choices for their positions is bounded by $\binom{n}{2m}\le \binom{n}{2p}$. We have the elementary bound
 $$\begin{aligned}
   \binom{n}{2p}&\le \left(\frac{en}{2p}\right)^{2p}\le \exp \left(2p(\log(n/2p)+1)\right) 
   = \exp \left(\frac{2p}{n}\left(-\log\frac{2p}{n}+1\right)n\right)\\
   &\le \exp \left(\ \frac{n\eta }{20}\right).
 \end{aligned}$$
Indeed, $2p/n\leq 1/L_0\to0$, hence $\frac{2p}{n}\left(-\log\frac{2p}{n}+1\right)\le \eta/20$ for $L_0$ large enough. Finally, using $1+t\le e^t$ for all $t$, we obtain $\#\Sigma(n)\le (n/L_0+1)\cdot e^{n\eta/20} \le e^{n\eta/10}$ for all large $L_0$.
\end{proof}

We do the usual trick: splitting the set of points according to those sequences, i.e., writing $T=\bigsqcup_{S\in\Sigma(n)} T_S$ ,  and consider one subset $T_S:=\{t\in T:\Sigma(n,\hsigma(t))=S\}$ at a time.

\section{Construction of the reparametrizations}

Recall that $\sigma$ is a regular parametrization of a well-chosen unstable manifold and $\sigma(T_S)$, $S\in\Sigma(n)$, is some set of positive measure for some unstable disintegration measure $\mu^u_x$. 

We start with the trivial family $\cR_0:=\{\Id\}$ and proceed by induction.
We write $S=(a_i,b_i)_{i=1}^m$ and set $b_0:=0$ for convenience. Now let $1\le i\le m$ and assume that we have a family of reparametrizations admissible up to time $b_{i-1}$.

Proposition~\ref{prop-reparam-htop} tells us that each reparametrization admissible up to time $b_{i-1}$ needs to be subdivided into at most
  $$
    N_i^1:=e^{\eta n_2/2} \exp\left(h_\top(f)+\frac{\lambda(\hf)}{r-1}+\eta\right)(a_i-b_{i-1})
  $$
reparametrizations to stay admissible up to time $a_i$.
For the neutral block $\ii{a_i,b_i}$ we apply Theorem~\ref{thm-reparam-neutral} and subdivide the reparametrizations up to time $a_i$ into at most
 $$
    N^0_i:=\exp\left(\frac{\lambda(\hf)}{r-1}+\eta\right)(b_i-a_i)
 $$
reparametrizations to keep the admissibility property up to time $b_i$ over $T_\theta$.

Finally we observe that, concatenating the subdivisions as in Lemma~\ref{lem-concat}, we end up subdividing the curve $\sigma$ into the following number of reparametrizations over $T_\theta$:
 $$
    \#\cR_n \le N^1_1\cdot N^0_1\cdot N^1_2\cdot\dots\cdot N^0_m\cdot N^1_{m+1}
      \le \exp\left(\beta\cdot h_\top(f)+\frac{\lambda(\hf)}{r-1}+3\eta\right)n.
 $$
The alternating non-neutral and neutral blocks are assumed to end with a (possibly empty) non-neutral block $\ii{b_m,a_{m+1}}$ where $a_{m+1}:=n$.

Thus,
 $$
   \limsup_k h(f,\nu_k) \le (\beta+3\eta)\cdot h_\top(f)
 $$
Since $\beta+3\eta<\beta'$, we obtain the desired contradiction and Theorem~\ref{thm-weak-dec} is proved. \qed

\chapter{Additional steps for Theorem B}\label{chap-beyond}

In this short chapter, we indicate how to strengthen Theorem~\ref{thm-weak-dec} into Theorem~\ref{theoremWithoutErgodicity}. This is only intended as a preparation to the reading of our paper \cite{BCS-2}, and we do not go into any details.

\section{The goal}
Both these theorems yields a decomposition
 $$
   \mu = (1-\beta)\mu_0 + \beta\cdot\mu_1\qquad (0<\beta\le 1, \mu_0,\mu_1\in\Prob(f))
 $$
We want to strengthen the entropy bound
 $$
   \limsup_{k\to\infty} h(f,\nu_k) \le \beta\cdot h_\top(f)
 $$
into
 $$
   \limsup_{k\to\infty} h(f,\nu_k) \le \beta\cdot h(f,\mu_1).
 $$
There are two stages to obtain this.

\section{Bounding by the entropy of an ergodic measure}

In the first stage, let us assume that the measure $\mu_1$ is ergodic (for instance if $\mu$ itself is). In this rather simple case, it is enough to use Katok's entropy formula: for all large $n\ge1$, there is a union of at most $\exp(h(f,\mu_1)n)$ dynamical balls with $\mu_1$-measure arbitrarily close to $1$. Since these balls are open sets, this is still true for the measures $(1-1_{\Neutral(\alpha,L)})\cdot \nu_k$ for $k$ large. 
Note here that it is necessary to first fix the iterate $n$ before deciding how large $k$ must be. Hence one has to split long non-neutral blocks into a concatenation of blocks with that fixed length $n$. One can adapt in this way the argument explained in the previous chapter as Proposition~\ref{prop-reparam-htop}.

\section{Bounding by the entropy of an arbitrary measure}

In the second stage, we remove the artificial assumption that the measure $\mu_1$ is ergodic (there is no reason why it should be!). Katok's entropy formula may then fail. It is replaced by the following: for all large $n\ge1$, there is a union of at most $\exp(\overline{h}(f,\mu_1)n)$ dynamical balls with $\mu_1$-measure arbitrarily close to $1$.

The new quantity that appears here:  $\overline{h}(f,\mu_1)$, is the \emph{essential supremum} of the entropies of the ergodic components of $\mu_1$. This is unavoidable. In fact, this quantity is what Katok's entropy formula yields when applied to a non-ergodic measure. Of course, it is usually much larger than the usual entropy, which is the \emph{average} of these entropies.

To solve this problem, we split the ergodic decomposition $\mu=\int \mu_\xi\, d\xi$ according to the value of $h(f,\mu_\xi)$. Thus,
 $$
   \mu = a_1\mu_1+\dots+a_N\mu_N \text{ with } 0\le \overline{h}(f,\mu_n)-h(f,\mu_n)\leq\eta.
 $$
The weak convergence $\nu_k\to\mu$ implies that $\nu_k$-typical orbits can be split into intervals, each of which is close to being $\mu_n$-typical for some $1\le n\le N$. Moreover the fraction of the time spent close to $\mu_n$ is close to $a_n$ (in particular, the leftover is a small fraction). 

One can apply Katok's formula for each $\mu_n$ in order to control the corresponding intervals. Since the measures $\mu_n$ are invariant, the previous intervals are very long and therefore do not create significant entropy.

\part{Conclusion}

\chapter{Some further results and questions}\label{chap-further}

In these notes we have tried to present the main ideas of \cite{BCS-2} in a way that avoids as many difficulties as possible, but still capturing the main ideas and perhaps the most important case, when the entropies of the measures converge to the maximum, i.e., to the topological entropy.

The interested reader is invited to read the original paper \cite{BCS-2} to understand the additional twists involved in other cases, already when the limiting measure is ergodic, and finally when it is not, as we believe that the added difficulties are somewhat transverse to the ideas presented here. 

We would like to draw the reader's attention to the related but different approach of Burguet \cite{Burguet-SRB}. Instead of focusing on neutral orbit segment, Burguet defines ``hyperbolic'' orbit segments. He obtains a similar but different decomposition with more control. In subsequent works he has been able to obtain precise, often sharp, results in finite smoothness, at least near the measures of maximal entropy. 

In our opinion, it is an interesting open problem to understand if there are other relevant variants of the neutral decomposition, especially decomposition that would satisfy additional properties, such as invariance under time reversal. Not only would this potentially improve the entropy bound, but it might open the door to attacking the problem of finding ergodic measures that maximize the Hausdorff dimension, for which our approach only gives partial results.

Several natural extensions of \cite{BCS-2} have been already obtained by now:
\begin{itemize}
 \item Hengyi Li has dealt with the situation where critical points are allowed, generalizing our results in the one-dimensional case (see \cite{Li-nonflat} for non-flat critical points; further results are being written in collaboration with Alexandre Delplanque \cite{Delplanque-Li-2025}). This later work relies on the Burguet's decomposition and reparametrization lemma mentioned above.
 \item As noticed by Yuntao Zang \cite{Zang-flow} and Liu Chuyi and Dawei Yang in \cite{Chiyi-Yang-2024}, one can use Burguet's approach as long as the entropy is still be given by curves: for three dimensional flows or diffeomorphisms by using some symmetries.
 \item Karina Marin, Mauricio Poletti, and Filiphe Veiga \cite{MPV} have considered certain skew products which can be considered as ``random composition of surface diffeomorphisms''.
\end{itemize}

To finish, we wish to point out how much remains to be done, especially in settings where the entropy can no longer be computed from the dynamics on curves.

\printindex

\bibliographystyle{plain}
\bibliography{IMPAN-Lectures}

\end{document}